\documentclass[10pt]{amsart}
\usepackage{mathtools,multicol}
\usepackage{kpfonts}
\usepackage{mathdots,enumerate}
\usepackage{multirow}
\newtheorem{theorem}{Theorem}[section]
\newtheorem{lemma}[theorem]{Lemma}
\newtheorem{proposition}[theorem]{Proposition}
\newtheorem{posledica}[theorem]{Corollary}

\theoremstyle{definition}

\theoremstyle{remark}
\newtheorem{remark}[theorem]{Remark}

\numberwithin{equation}{section}

\usepackage{amssymb,amsmath}
\usepackage{pst-node}

\usepackage{tikz-cd}
\usepackage{tikz}
\usepackage{enumitem}

\usepackage{tabularx}
\newcounter{rownumber}
\newcommand{\C}{\mathbb{C}}

\newcommand{\R}{\mathbb{R}}

\newcommand{\hra}{\hookrightarrow}

\def\id{\mathop{\rm id}\nolimits}

\def\dim{\mathop{\rm dim}\nolimits}

\def\rank{\mathop{\rm rank}\nolimits}

\def\Rea{\mathop{\rm Re}\nolimits}
\def\Ima{\mathop{\rm Im}\nolimits}

\def\Orb{\mathop{\rm Orb}\nolimits}

\def\Span{\mathop{\rm Span}\nolimits}

\parskip=\smallskipamount

\makeatletter 
\renewcommand\p@enumii{}
\makeatother

\begin{document}
\title[]
{On normal forms of complex points of small $\mathcal{C}^{2}$-perturbations of real $4$-manifolds embedded in a complex $3$-manifold}
\author{
Tadej Star\v{c}i\v{c}}
\address{Faculty of Education, University of Ljubljana, Kardeljeva Plo\v{s}\v{c}ad 16, 1000 Lju\-blja\-na, Slovenia}
\address{Institute of Mathematics, Physics and Mechanics, Jadranska
  19, 1000 Ljubljana, Slovenia}
\email{tadej.starcic@pef.uni-lj.si}
\date{January 7, 2019}


\keywords{CR manifolds, closure graphs, complex points, normal forms, perturbations, simultaneous reduction \\
\indent Research supported by grant P1-0291 and J1-7256 
from ARRS, Republic of Slovenia.}

\begin{abstract} 
We answer the question how arbitrarily small perturbations of a pair of one arbitrary and one symmetric $2\times 2$ matrix can change a normal form with respect to a certain linear group action. This result is then applied to describe the quadratic part of normal forms of complex points of small $\mathcal{C}^{2}$-perturbations of real $4$-manifolds embedded in a complex $3$-manifold.
\end{abstract}

\maketitle

\vspace{-5mm}
\section{Introduction} \label{intro}

The study of complex points was started in 1965 by E. Bishop with his seminal work on the problem of describing the hull of holomorphy of a submanifold near a point with one-dimensional complex tangent space \cite{Bishop}.
This is now very well understood for surfaces (see Bishop \cite{Bishop}, Kenig and
Webster \cite{KW}, Moser and Webster \cite{MW}),
and it later initiated in many researces in geometric analysis. For instance, (formal) normal forms for
real submanifolds in $\mathbb{C}^{n}$ near complex points were considered by Burcea \cite{Bur}, Coffman \cite{Coff},
Gong \cite{Gong1} and Gong and Stolovich \cite{GongStolo1}, Moser and Webster \cite{MW} among others.
We add that topological structure of complex points was first considered by
Lai \cite{Lai} and in the special case of surfaces by Forstneri\v c \cite{F}. Up
to $\mathcal{C}^0$-small isotopy complex points of real codimension $2$ submanifolds in complex manifolds were treated by Slapar \cite{S1, S2, S3}.

In this paper we describe the behavior of
the quadratic part of normal forms of complex points of small $\mathcal{C}^{2}$-perturbations of real $4$-manifolds embedded in a complex $3$-manifold (see Corrolary \ref{pospert}). It is a direct consequnce of a result that clarifies how a normal form for a pair of one arbitrary and one symmetric $2\times 2$ matrix with respect to a certain linear algebraic group action changes under small perturbations (see Theorem \ref{izrek}); by a careful analysis we also provide information how small the perturbations must be. Due to technical reasons, these results are precisely stated in Section \ref{Sec3} and then proved in later sections. 

Let $f\colon M^{2n}\hra X^{n+1}$ be a $\mathcal C^2$-smooth embedding of a real smooth $2n$-manifold into a complex $(n+1)$-manifold $(X,J)$. A point
$p\in M$ is \textit{CR-regular} if the dimension of the complex 
tangent space $T^C_pM=df(T_pM)\cap Jdf(T_pM)\subset T_{f(p)}X$ is $n-1$, while $p$ is called \textit{complex} when the complex dimension of $T_p^CM$ equals $n$, thus $T_p^CM=df(T_pM)$. 
By Thom's transversality theorem \cite[Section 29]{Arnold}, for generic embeddings the intersection
is transverse and so complex points are isolated.  
Using Taylor expansion $M$ can near a complex point $p\in M$ be seen as a graph:
\begin{equation*}
w=r^Tz+\overline{z}^TAz+\tfrac{1}{2}\overline{z}^TB\overline{z}+\tfrac{1}{2}z^TCz+o\big(|z|^2\big), \quad (w(p),z(p))=(0,0),
\end{equation*}
where $(z,w)=(z_1,z_2,\ldots,z_n,w)$ are suitable local coordinates on
$X$, and $A\in \mathbb{C}^{n\times n}$, $B,C\in \mathbb{C}^{n\times n}_S$, $r\in \mathbb{C}^n$. 
By $\mathbb{C}^{n\times n}$ we denote the group of all $n\times n$ complex matrices, and by $\mathbb{C}^{n\times n}_S$, $GL_n(\mathbb{C})$, respectively, its subgroups of symmetric and nonsingular matrices. After a simple change of coordinates $(\widetilde{z},\widetilde{w})=\big(z,w-b^Tz-\tfrac{1}{2}z^T(\overline{B}-C)z\big)$ it is achieved that $df(T_p M)\approx \{\widetilde{w}=0\}$, and the \textit{normal form up to quadratic terms} is:
\begin{equation}\label{BasForm1}
\widetilde{w}=\overline{z}^TAz+\Rea (z^TBz)+o(|z|^2), \quad (\widetilde{w}(p),z(p))=(0,0),\,\,\, A\in \mathbb{C}^{n\times n}, B\in \mathbb{C}^{n\times n}_S.
\end{equation}
%
A real analytic complex point $p$ is called \textit{flat}, if local coordinates can be chosen so that the graph (\ref{BasForm1}) lies in $\C^n_z\times \R\subset \C^n_z\times\C_w$. It is \textit{quadratically flat}, if 
the quadratic part of (\ref{BasForm1}) is real valued; this happens precisely when $A$ in (\ref{BasForm1}) is Hermitian.

Any holomorphic change of coordinates that
preserves $(0,0)$ and $ \{\widetilde{w}=0\}$ as a set in (\ref{BasForm1}), has the same effect on
the quadratic part as a complex-linear change
\[
\begin{bmatrix}
\widetilde{z} \\
\widetilde{w}
\end{bmatrix}
=
\begin{bmatrix}
P & r \\
0 & c
\end{bmatrix}
\begin{bmatrix}
\widehat{z} \\
\widehat{w}
\end{bmatrix}, \quad P\in GL_n(\mathbb{C}),\, r\in \C^n, \,c\in \mathbb{C}^{*}=\mathbb{C}\setminus \{0\}.
\]
Furthermore, using this linear changes of coordinates and a biholomorphic change
$
(\widehat{z},\widehat{w})\mapsto \left(\widetilde{z},\widehat{w}-\frac{1}{2}\widehat{z}^T\big(\frac{\overline
    c-c}{|c|^2}P^TBP\big)\widehat{z}\right)
$
transforms (\ref{BasForm1}) into the equation that can by a slight abuse of notation be written as
\begin{equation*}
w=\overline{z}^T\left(\tfrac{1}{c}P^{*}AP\right) z+\Rea \left(z^T (\tfrac{1}{\overline{c}} P^TBP) z\right)+o(|z|^2).
\end{equation*}
Studying the quadratic part of complex points thus means examining 
the action of a linear group $\mathbb{C}^*\times GL_n(\mathbb{C})$ on pairs of matrices $\mathbb{C}^{n\times n}\times \mathbb{C}^{n\times n}_S $, introduced in \cite{Coff}:
\begin{equation}\label{action1}
\Xi\colon\bigl((c,P),(A,B)\bigr)\mapsto (cP^{*}AP,\overline{c}P^TBP), \quad P\in GL_n(\mathbb{C}), c\in \mathbb{C}^{*}. 
\end{equation}
Problems of the quadratic part thus reduce to problems in matrix theory.

When $n=1$ complex points are
always quadratically flat and locally given by the equations
$w=z\overline{z}+\frac{\gamma}{2} (z^2+\overline{z}^2)+o(|z|^2)$, $0\leq \gamma
<\infty$ or $w=z^2+\overline{z}^2+o(|z|^2)$ (Bishop
\cite{Bishop}). If in addition they are real analytic and elliptic  ($\gamma<1$),
they are also flat (see \cite{MW}). A relatively simple
description of normal forms of the action (\ref{action1}) was obtained for $n=2$ (see Coffman
\cite{Coff} and Izotov \cite{Izotov}), while in dimensions $3$ and $4$ a complete list of normal forms has been given only in the case of quadratically flat complex points (see Slapar and Star\v{c}i\v{c} \cite{ST}). Nevertheless, if $B$ in (\ref{BasForm1}) is nonsingular the classification has been done even in higher dimensions by the result of Hong \cite{Hong89}.

The problem of normal forms of matrices under perturbations was first studied by Arnold (see e.g. \cite{Arnold}), who considered matrices depending on parameters under similarity (miniversal deformations). The change of Jordan canical form has been then succesfully investigated also through the works of Markus and Parilis \cite{MarkusParilis}, Edelman, Elmroth and K\aa gstrom \cite{EEK}, among others; the software Stratigraph \cite{EJK} contructs the relations between Jordan forms. However, the problem of normal forms for $*$-conjugation (or $T$-conjugation) under small perturbations seems to be much more involved, and has been so far inspected only in lower dimensions; check the papers Futorny, Klimenko and Sergeichuk \cite{FKS1}, Dmytryshyn, Futorny and Sergeichuk\cite{DFS} (Dmytryshyn, Futorny, K\aa gstr\"{o}m, Klimenko and Sergeichuk \cite{FKS2}). Virtually nothing has been known until the time of this writing about simultaneous small perturbations of pairs of matrices under these actions.

In connection to these problems we mention results of Guralnick \cite{Gura} and Leiterer \cite{Leiterer}, who respectively studied similarity of holomorphic maps from Riemann surfaces or Stein spaces to a set of matrices. 
We shall not consider this matter here.
 
%
%
%
%
%
\section{Normal forms in dimension $2$}\label{secNF}

We recall the basic properties of an action of a Lie group on a manifold (check e.g. 
\cite[Theorem IV.9.3]{Boot}).
These are well known and they are not difficult to prove.

\begin{proposition}\label{propA}
Let $\Phi\colon G\times X\to X$ be a smooth (analytic) action  of a real (complex) Lie group $G$ with a unit $e$, acting on a smooth (complex) manifold $X$, i.e. 
\[
\Phi(e,x)=x, \qquad \Phi\big(g,\Phi(h,x)\big)=\Phi(g h,x),\,g,h\in G,\,x\in X.
\]
Then $\Phi$ satisfies the following properties:
\begin{enumerate}
\item For any $g\in G$ the map $\Phi^g \colon X\to X$, $ x \mapsto \Phi(g,x)$ is an automorphism and the map $L_g\colon G\to G$, $h\mapsto g h$ is a transitive automorphism.
\item \label{propAb} For any $x\in X$ the orbit map $\Phi_x \colon G\to X$, $ g \mapsto \Phi(g,x)$ is transitive and equivariant 
\big(for any $g\in G$ we have $\Phi_x(L_g (h))=\Phi^g(\Phi_x(h))$, $h\in G$\big), and $\Phi_x$ is of constant rank with $d_{g}\Phi_x=d_{x}\Phi^g\circ d_e\Phi_x\circ (d_{e}L_g)^{-1}$.
\item \label{propAc} An orbit of $x\in X$, denoted by 
$\Orb_{\Phi}(x)=\{\Phi(g,x)\mid g\in G\}$,
is an immersed (locally embedded) homogeneous submanifold of dimension equal to $\rank (\Phi_x)$.
\end{enumerate}
\end{proposition}

Any orbit can be endowed (globally) with the structure of a manifold, but it does not necessarily coincide with the subspace topology (\cite[Theorem IV.9.6]{Boot}).

We proceed with the list of representatives of orbits (normal forms) of the action (\ref{action1}) for $n=2$ obtained by Coffman (see \cite[Sec. 7,Table 1]{Coff}). In addition, we compute tangent spaces of orbits and then arrange normal forms into a table according to dimensions of their orbits ($42$ types); these are calculated similarily as in the case of similarity (see e.g. Arnold \cite[Section 30]{Arnold}). 
To simplify the notation, $a\oplus d$ denotes the diagonal matrix with 
$a$, $d$ on the main diagonal, while the $2\times 2$ identity-matrix  
and the $2\times 2$ zero-matrix are $I_2$ 
and $0_2$, respectively.

\begin{lemma}\label{lemalist}
Orbits of the action (\ref{action1}) for $n=2$ (represented by its normal forms $(A,B)$) are immersed manifolds, their dimensions are given in the first columns of the tables:

\begin{tabular}{p{4mm}|| c| c  ||c | c || c | c ||c | c}
\small{$\dim$} & $A$ & $B$ & $A$ & $B$ & $A$ & $B$ & $A$ & $B$\\
\hline
9
&
\multirow{ 10}{*}{
\small{$1\oplus e^{i\theta}$}
}
&
$
\begin{bsmallmatrix}
0 & b \\
b & d
\end{bsmallmatrix}
$
&
\multirow{ 10}{*}{
$
\begin{bsmallmatrix}
0 & 1 \\
\tau & 0
\end{bsmallmatrix}
$
}
&
$
\begin{bsmallmatrix}
e^{i\varphi} & b \\
b & \zeta
\end{bsmallmatrix}
$
&
\multirow{ 10}{*}{
$
\begin{bsmallmatrix}
0 & 1 \\
1 & i
\end{bsmallmatrix}
$
}
&
$
\begin{bsmallmatrix}
0 & b \\
b & 0
\end{bsmallmatrix}
$
&
\multirow{ 11}{*}{
$
\begin{bsmallmatrix}
0 & 1 \\
0 & 0
\end{bsmallmatrix}
$
}
&
$
a\oplus 1
$
\\

&
&
$
\begin{bsmallmatrix}
a & r e^{i\varphi} \\
r e^{i\varphi} & d
\end{bsmallmatrix}
$
&
&
$
\begin{bsmallmatrix}
0 & b \\
b & e^{i\varphi}
\end{bsmallmatrix}
$
&
&
$
a\oplus \zeta
$
&
&
$
\begin{bsmallmatrix}
\zeta & b \\
b & 1
\end{bsmallmatrix}
$
\\

&
&
$
\begin{bsmallmatrix}
a & b \\
b & 0
\end{bsmallmatrix}
$
&
&
$
1\oplus \zeta
$
&
&
&
&
$
\begin{bsmallmatrix}
1 & b \\
b & 0
\end{bsmallmatrix}
$
\\

&
&
&
&
$
0\oplus 1
$
&
&
&
&
 \\
\cline{1-1}\cline{3-3}\cline{5-5}\cline{7-7}\cline{9-9}

8
&
&
$
0\oplus d
$
&
&
$
\begin{bsmallmatrix}
0 & b \\
b & 0
\end{bsmallmatrix}
$
&
&
$
0\oplus d
$
&
&
$
0\oplus 1
$
\\

&
&
$
\begin{bsmallmatrix}
0 & b \\
b & 0
\end{bsmallmatrix}
$
&
&
&
&
&
&
$
1\oplus 0
$
\\

&
&
$
a\oplus 0
$
&
&
&
&
&
&
\\
\cline{1-1}\cline{3-3}\cline{5-5}\cline{7-7}\cline{9-9}

7
&
&
$0_2$
&
&
$0_2$
&
&
$0_2$
&
&
$
\begin{bsmallmatrix}
0 & b \\
b & 0
\end{bsmallmatrix}
$
\\ 
\cline{1-7}\cline{9-9}

6
&
\multicolumn{5}{c}{}
&
&
&
$0_2$
\end{tabular}\\

\begin{tabular}{p{4mm}|| c| c  ||c | c || c | c ||c | c}
\small{$\dim$} & $A$ & $B$ & $A$ & $B$ & $A$ & $B$ & $A$ & $B$\\
\hline
$9$ & \multirow{6}{*}{$I_2$} & $a\oplus d$, $\scriptscriptstyle{a< d}$ &  \multirow{6}{*}{\small{$1\oplus -1$}} & $a\oplus d$, $\scriptscriptstyle{a< d}$ & 
\multirow{4}{*}{
$
\begin{bsmallmatrix}
0 & 1 \\
1 & 0
\end{bsmallmatrix}
$
}
& $1\oplus de^{i\theta}$ & 
\multirow{7}{*}{\small{$1\oplus 0
$}}
& $a \oplus 1$\\
  &  &  &   &  &  &
 $
\begin{bsmallmatrix}
0 & b \\
b & 1
\end{bsmallmatrix}
$  
  & &  \\
\cline{1-1}\cline{3-3}\cline{5-5}\cline{7-7}\cline{9-9}

8 & &  $d_0\oplus d$ &  & $d_0\oplus d$ & & $1 \oplus 0 $ & & $0\oplus 1$\\
 &  &    & &
 $
 \begin{bsmallmatrix}
 0 & b \\
 b & 0
 \end{bsmallmatrix}
$
 &  & & &
$
\begin{bsmallmatrix}
0 & 1 \\
1 & 0
\end{bsmallmatrix}
$ \\
\cline{1-1}\cline{3-3}\cline{5-7}\cline{9-9}
6 & & & & & \multirow{4}{*}{$0_2$} & $I_2$ & & \\
\cline{1-1}\cline{3-3}\cline{5-5}\cline{7-7}\cline{9-9}
5 & & $0_2$ & & $0_2$ & &&  & $a\oplus 0$\\
\cline{1-5}\cline{7-7}\cline{9-9}
4 &   \multicolumn{2}{c}{} &   \multicolumn{2}{c||}{} & & $1\oplus 0$ &  & $0_2$\\
\cline{7-9}
0 &   \multicolumn{2}{c}{} &   \multicolumn{2}{c||}{} & & $0_2$ &  \multicolumn{2}{c}{} \\
\end{tabular}\\
Here $0<\tau <1$, $0<\theta <\pi$,\,\,\, $a,b,d>0$,\, $d_0\in\{0,d\}$,\, $r\geq 0$,\, $\zeta \in \mathbb{C}$,\, $0\leq \varphi <\pi$.
\end{lemma}
A minor change is made in comparison to the original list in \cite{Coff}, as
$
\left(\begin{bsmallmatrix}
0 & 1\\
1 & 0
\end{bsmallmatrix},
\begin{bsmallmatrix}
1 & 0\\
0 & 0
\end{bsmallmatrix}\right)
$ is taken instead of
$
\left(\begin{bsmallmatrix}
1 & 0\\
0 & -1
\end{bsmallmatrix},
\begin{bsmallmatrix}
1 & 1\\
1 & 1
\end{bsmallmatrix}\right)
$; conjugate the later one with
$\frac{1}{2}\begin{bsmallmatrix}
1 & 2\\
1 & -2
\end{bsmallmatrix}
$.
\begin{proof}[Proof of Lemma \ref{lemalist}]
By Proposition \ref{propA} (\ref{propAb}), (\ref{propAc}), orbits of the action $\Xi$ in (\ref{action1}) for $n=2$ are immersed manifolds. To compute the tangent space of the orbit $\Orb_{\Xi}(A,B)$ we fix $(A,B)\in \mathbb{C}^{2\times 2}\times \mathbb{C}_S^{2\times 2}$, choose a path going through $(1,(I_2,I_2))$:
\[
\gamma\colon (-\delta,\delta)\to \mathbb{C}^{*}\times GL_2(\mathbb{C}), \quad \gamma(t)=(1+t\alpha,I+tX), \qquad \alpha\in \mathbb{C}, X\in \mathbb{C}^{2\times 2}, \delta>0
\]
and then calculate
\begin{align*}
\frac{d}{dt}\Big|_{t=0}\big((I+tX)^*A(I+tX)\big)& =\frac{d}{dt}\Big|_{t=0}\big(A+t(X^*A+AX)+t^2X^*AX\big)
                                        =(X^*A+AX),\\
\frac{d}{dt}\Big|_{t=0}\big((I+tX)^TB(I+tX)\big) &=\frac{d}{dt}\Big|_{t=0}\big(B+t(X^TB+BX)+t^2X^TBX\big)
                                        =(X^TB+BX),\\
\frac{d}{dt}\Big|_{t=0}\Xi\big( \gamma(t),(A,B)\big)& =\big(\alpha A+(X^*A+AX),\overline{\alpha} B+ (X^TB+BX)\big).
\end{align*}
Writing $X=\sum_{j,k=1}^2(x_{j,k}+iy_{j,k})E_{jk}$, $\alpha=u+iv$, where $E_{jk}$ is the elementary matrix with one in the $j$-th row and $k$-th column and zeros otherwise, we deduce that
\begin{align*}
\alpha A+(X^*A+AX) &=   (u+iv) A+\sum_{j,k=1}^2(x_{j,k}-iy_{j,k})E_{kj}A+\sum_{j,k=1}^2(x_{j,k}+iy_{j,k})AE_{jk}\\
                             &=   u A+v(iA)+\sum_{j,k=1}^2x_{jk}(E_{kj}A+AE_{jk})+\sum_{j,k=1}^2y_{jk}i(-E_{kj}A+AE_{jk})
\end{align*}
and in a similar fashion we conclude that
\begin{equation*}
\overline{\alpha} B+(X^TB+BX)  
                             =   u B+v(-iB)+\sum_{j,k=1}^2x_{jk}(E_{kj}B+BE_{jk})
                               +\sum_{j,k=1}^2y_{jk}i(E_{kj}B+BE_{jk}).
\end{equation*}
Let a $2\times 2$ complex (symmetric) matrix be identified with a vector in a real Euclidean space $\mathbb{R}^{8} \approx \mathbb{C}^{2\times 2}$ (and $\mathbb{R}^{6} \approx \mathbb{C}^{2\times 2}_S$), thus $\mathbb{C}^{2\times 2}\times \mathbb{C}^{2\times 2}_S\approx \mathbb{R}^{8}\times \mathbb{R}^{6}= \mathbb{R}^{14}$ with the standard basis $\{e_1,\ldots,e_{14}\}$. The tangent space of an orbit $\Orb_{\Xi}(A,B)$ can then be seen as the linear space spanned by the  vectors $\{w_1,w_2\}\cup \{ u_{jk},v_{jk}\}_{j,k\in\{1,2\} }$, where
\begin{align}\label{vecspan}
&w_1\approx (A,B), \qquad w_2\approx (iA,-iB), \\
v_{jk}\approx (E_{kj}A+AE_{jk},E_{kj}B+BE_{jk}), \quad & u_{jk}\approx i(-E_{kj}A+AE_{jk},E_{kj}B+BE_{jk}),\quad j,k\in \{1,2\}.\nonumber
\end{align}

We splitt our consideration of tangent spaces according to the list of normal forms in \cite[Sec. 7,Table 1]{Coff} into several cases.
In each case the tangent space will be written as a direct sum of linear subspaces $V_1\oplus V_2$ (with trivial intersection) such that $V_2\subset \Span\{e_j\}_{9\leq j\leq 14}$ and $V_1$ will be either trivial or of the form $\sum_{j=1}^{14}\alpha_j e_j$ with some  nonvanishing $\alpha_{j_0}$, $j_0 \in \{1,\ldots, 8\}$.

\begin{enumerate}[label={\bf Case \Roman*.},ref={Case \Roman*},wide=0pt,itemsep=5pt]
\item \label{s1i}
$
(A,B)=
\big(\begin{bsmallmatrix}
0 & 1\\
1 & i
\end{bsmallmatrix},
\begin{bsmallmatrix} 
a & b \\  
b &  d 
\end{bsmallmatrix}
\big)$,\quad $a,b\geq 0$, \,\,$d\in \mathbb{C}$

\quad 
From (\ref{vecspan}) we obtain that
{\small
\begin{align*}
&w_1 =e_3+e_5+e_8+a e_9+b e_{11}+ (\Rea d) e_{13}   +(\Ima d) e_{14} \\
&w_2 =e_4+e_6-e_7-a e_{10}-b e_{12}+(\Ima d) e_{13} -(\Rea d) e_{14}\end{align*}
\vspace{-8mm}
\begin{align*}
&v_{11} =e_3+e_5+ 2a e_9+b e_{11} & & v_{21} =2 e_1+ e_4+ e_6+ 2b e_{9}+(\Rea d) e_{11}+(\Ima d) e_{12}\\
&v_{12} =2 e_7+a e_{11}+2b e_{13}  & & v_{22} =e_3+e_5+2 e_8+b e_{11}+ 2(\Rea d) e_{13}+ 2(\Ima d) e_{14}\\
& u_{11} =-e_4+e_6+2a e_{10}+b e_{12}  & & u_{21} =e_3-e_5+2b e_{10}-(\Ima d) e_{11}+(\Rea d) e_{12}\\
&u_{12} =a e_{12}+2b e_{14}  &&u_{22} =e_4-e_6+b e_{12}-2(\Ima d) e_{13}+2(\Rea d) e_{14}
\end{align*}
}%
It is apparent that $v_{22}=2w_1-v_{11}$. By further setting  
\[
w_3=u_{22}+u_{11}=2a e_{10}+2b e_{12}-2(\Ima d) e_{13}+2(\Rea d) e_{14}
\]
we can choose $V_1=\Span\{w_1,w_2,v_{11},v_{12},v_{21},u_{11},u_{21}\}$, $V_2=\Span\{u_{12},w_{3}\}$ and observe that $\dim V_1=7$ 
$\dim V_2=\small{\left\{
\begin{array}{ll}
0,  &  a=b=d=0 \\
1,  &  a=b=0,d\neq 0\\
2,  &  \textrm{otherwise}
\end{array}\right.}
$.

\item \label{s100l}
$
(A,B)=
\big(1\oplus \lambda,
\begin{bsmallmatrix} 
a & b \\  
b &  d 
\end{bsmallmatrix}
\big)$,
\quad $\lambda\in \{0\}\cup \{e^{i\theta}\}_{\theta\in [0,\pi]}$,\,\, $a,d\geq 0$,\,\,$b\in \mathbb{C}$

\quad
We have (see (\ref{vecspan})):
\small
\begin{align*}
&w_1=e_1+(\Rea \lambda) e_7+ (\Ima \lambda) e_8+ a e_9+(\Rea b) e_{11}+(\Ima b) e_{12}+d e_{13}\nonumber\\
&w_2=e_2- (\Ima \lambda) e_7+(\Rea \lambda) e_8-a e_{10}+(\Ima b) e_{11}-(\Rea b) e_{12}-d e_{14}\nonumber\\
&v_{11}=2 e_1+2a e_9+(\Rea b) e_{11}+(\Ima b) e_{12}\nonumber\\
&v_{12}=e_3+e_5+a e_{11}+2(\Rea b) e_{13}+ 2(\Ima b) e_{14}\nonumber\\
&v_{21}=(\Rea \lambda) e_{3}+ (\Ima \lambda) e_4+ (\Rea \lambda) e_{5}+(\Ima \lambda) e_6+2(\Rea b) e_9+ 2(\Ima b) e_{10}+d e_{11}\\
&v_{22}=2 (\Rea \lambda) e_7+ 2(\Ima \lambda) e_8+(\Rea b) e_{11}+(\Ima b) e_{12}+2de_{13}\nonumber\\
&u_{11}=2a e_{10}-(\Ima b) e_{11}+(\Rea b) e_{12}\nonumber\\
%
&u_{12}=e_4-e_6+a e_{12}-2(\Ima b) e_{13}+2(\Rea b) e_{14}\nonumber\\
&u_{21}=(\Ima \lambda) e_3-(\Rea \lambda) e_4-(\Ima \lambda) e_5+(\Rea \lambda) e_6 -2(\Ima b) e_9 +2(\Rea b) e_{10}+d e_{12}\nonumber\\
&u_{22}=(-\Ima b) e_{11}+(\Rea b) e_{12}+2d e_{14}.\nonumber
\end{align*}
\normalsize
It is immediate that $v_{11}=2w_1-v_{22}$. 

\quad
For $\lambda=0$ we set 
$V_1=\Span\{w_1,w_2,v_{12},u_{12}\}$, $V_2=\Span\{v_{21},v_{22},u_{11},u_{21}\}$ and we have $\dim V_1=4$ and 
$\dim V_2=\small{\left\{
\begin{array}{ll}
0,  &  a=b=d=0 \\
1,  &   d=b=0,a>0\\
4,  &  a=b=0,d=1 \textrm{ or }a=d=0,b=1\\
5,  &  0<a,d=1,b=0
\end{array}\right.}
$.
Next, for $\lambda=e^{i\theta}$, $0<\theta<\pi$ we take 
$V_1=\Span\{w_1,w_2,v_{12},v_{21},u_{22},u_{12},u_{21}\}$ and $V_2=\Span\{u_{11},u_{22}\}$ with 
$\dim V_1=7$, 
$\dim V_2=\small{\left\{
\begin{array}{ll}
0,  &  a=b=d=0 \\
1,  &   ad\neq 0,a+d\neq 0 \textrm{ or } a,d=0,b\neq 0\\
2,  &  \textrm{otherwise}
\end{array}\right.}
$.

\quad
Finally, for $\lambda\in \{-1,1\}$ we set 
\begin{align*}
&w_3=v_{21}-\lambda v_{12}= 2(\Rea b) e_9+ 2(\Ima b) e_{10}+(d-\lambda a) e_{11}-2\lambda(\Rea b) e_{13}- 2\lambda(\Ima b) e_{14}\\
&w_4=u_{21}+\lambda u_{12}=-2(\Ima b) e_9 +2(\Rea b) e_{10}+(d+\lambda a) e_{12}+2\lambda(\Ima b) e_{13}-2\lambda(\Rea b) e_{14}.
\end{align*}
If we choose   
$V_1=\Span\{w_1,w_2,v_{12},v_{22},u_{12}\}$, $V_2=\Span\{u_{11},u_{22},w_{3},w_4\}$, 
then $\dim V_1=5$, 
$\dim V_2=\small{\left\{
\begin{array}{ll}
0,  &  a=b=d=0 \\
3,  &   d>0, a\in \{0,d\}, b=0 \textrm{ or } a,d=0,b\neq 0\\
4,  &  0<a< d,b=0
\end{array}\right.}
$.

\item \label{s01t0}
$
(A,B)=\big(\begin{bsmallmatrix}
0 & 1\\
t & 0
\end{bsmallmatrix},
\begin{bsmallmatrix} a & b \\  b&  d \end{bsmallmatrix}\big)
$, \quad $0\leq t\leq 1,\,\, b\geq 0$, \,\,$a,d\in\mathbb{C}$

\quad 
It follows from (\ref{vecspan}) that
{\small
\begin{align*}
&w_1= e_3+t e_5+ (\Rea a) e_9+ (\Ima a9 e_{10}+b e_{11}+(\Rea d) e_{13}+(\Ima d) e_{14} \\
&w_2= e_4+t e_6+(\Ima a) e_9-(\Rea a) e_{10}-b e_{12}+(\Ima d) e_{13}-(\Rea d) e_{14}
\end{align*}
\vspace{-8mm}
\begin{align*}
&v_{11}=e_3+t e_5+ 2 (\Rea a) e_9+2 (\Ima a) e_{10}+b e_{11} && u_{11}=-e_4+t e_6-2(\Ima a) e_9+2 (\Rea a) e_{10}+b e_{12}\\
&v_{12}=(1+t)e_7+(\Rea a) e_{11}+(\Ima a) e_{12}+2b e_{13} && u_{12}=(-1+t)e_8-(\Ima a) e_{11}+(\Rea a) e_{12}+2be_{14}\\
&v_{21}=(1+t)e_1+2b e_9+(\Rea d) e_{11}+(\Ima d) e_{12} && u_{21}=(1-t)e_2+2b e_{10}-(\Ima d) e_{11}+(\Rea d) e_{12}\\
&v_{22}=e_3+t e_5+b e_{11}+2(\Rea d) e_{13}+2(\Ima d) e_{14} && u_{22}=e_4-t e_6+b e_{12}-2(\Ima d) e_{13}+2(\Rea d) e_{14}.
\end{align*}
}%

\quad
Let us denote 
\begin{align*}
&w_3=v_{11}-w_{1}= -v_{22}+w_{1}= (\Rea a) e_9+(\Ima a) e_{10}-(\Rea d) e_{13}-(\Ima d) e_{14},\\
&w_4=u_{22}+u_{11}=-2(\Ima a) e_9+2 (\Rea a) e_{10}+2b e_{12}-2 (\Ima d) e_{13}+2(\Rea d) e_{14}.
\end{align*}
For $0<t<1$ we set
$V_1=\Span\{w_1,w_2,v_{12},v_{21},u_{11},u_{12},u_{21}\}$, $V_2=\Span\{w_{3},w_{4}\}$ and get that $\dim V_1=7$, 
$\dim V_2=\small{\left\{
\begin{array}{ll}
0,  &  a=b=d=0 \\
1,  &   a=d=0, b>0\\
2,  &  \textrm{otherwise}
\end{array}\right.}
$.
Next, if $t=1$ we take
$V_1=\Span\{w_1,w_2,v_{12},v_{12},u_{11}\}$, $V_2=\Span\{w_3,w_4,u_{21},u_{12}\}$. Observe that $\dim V_1=5$, 
$\dim V_2=\small{\left\{
\begin{array}{ll}
3,  &  d=b=0,a=1 \\
4,  &  0\neq 0,a=1 \textrm{ or } b>0,d=1
\end{array}\right.}
$.
Finally, when $t=0$ we set 
\begin{equation}
w_5=u_{11}+w_{2}= 2(\Rea b) e_9+ 2(\Ima b) e_{10}+(d-ta) e_{11}-2t(\Rea b) e_{13}- 2t(\Ima b) e_{14}
\end{equation}
with  
$V_1=\Span\{w_1,w_2,v_{12},v_{22},u_{12},u_{21}\}$, $V_2=\Span\{w_{3},w_4,w_{5}\}$.  
It follows that $\dim V_1=6$ and 
$\dim V_2=\small{\left\{
\begin{array}{ll}
0,  &  a=b=d=0 \\
1,  &  a=d=0, b>0\\
2,  &  ad= 0,b=0\\
3,  &  \textrm{otherwise}   
\end{array}\right.}
$.

\item \label{s0000}
$(A,B)=(0_2,a\oplus d)$, \quad $a,d\in \{1,0\}$, \,\, $a\geq d$

\quad 
By (\ref{vecspan}) we have
\small
\begin{align*}
&w_1= ae_7+de_{13} & &v_{11}=2ae_7        &  & u_{11}=2ae_8\\
& w_2=-ae_8-de_{14}& &v_{12}=ae_{9}+ae_{11}  &   & u_{12}=ae_{10}+ae_{12}\\
& & &v_{21}=de_{9}+de_{11}   &  & v_{21}=de_{10}+de_{12}\\
& & &v_{22}=de_{13}  &  & v_{21}=de_{14}.
\end{align*}
\normalsize
These vectors are contained in $\Span\{e_j\}_{7\leq j\leq 14}$ and $2w_1=v_{11}+2v_{22}$, $2w_1=v_{11}+2v_{22}$. It is now easy to compute the dimension of their linear span.
\end{enumerate}
This finishes the proof. 
\end{proof}

\begin{remark}
Sometimes it it is more informative to understand the stratification into bundles of matrices, i.e. sets of matrices having similar properties. Again, this notion was introduced first by Arnold \cite[Section 30]{Arnold}. 
Given a parameter set $\Lambda$ with smooth maps 
$\lambda \mapsto A_{\lambda}$,  
$\lambda \mapsto B_{\lambda}$,
one considers a bundle of pairs of matrices under the action $\Xi$ in (\ref{action1}), i.e. 
a union of orbits $\bigcup_{\lambda\in\Lambda}\Orb_{\Xi}(A_{\lambda},B_{\lambda})$. We set
\begin{equation}\label{actionL}
\Xi_{\Lambda}\colon \bigl(\mathbb{C}^*\times GL_n(\mathbb{C})\bigr)\times \Lambda\to \mathbb{C}^{n\times n}\times \mathbb{C}^{n\times n}_S, \quad (c,P,\lambda)\mapsto \Xi(c,P,A_{\lambda},B_{\lambda}),
\end{equation}
%
and observe that for any $g\in \mathbb{C}^*\times GL_n(\mathbb{C})$ we have $\Xi^g\circ \Xi_{\Lambda}=\Xi_{\Lambda} \circ(L_g\times \id_{\Lambda})$,
so the rank of $d_{(g,\lambda)}\Xi_{\Lambda}$ depends only on $\lambda\in \Lambda$.
In a similar manner as we computed the tangent space of an orbit, the tangent space of a bundle can be obtained. 
It follows that the generic $2\times 2$ pairs of one arbitrary and one symmetric matrix (forming a bundle with maximal dimension $14$) under the action (\ref{action1}) for $n=2$ are:
\begin{align*}
&\left( \begin{bmatrix}
0 & 1\\
\tau & 0
\end{bmatrix},
\begin{bmatrix}
e^{i\varphi} & b\\
b & \zeta
\end{bmatrix}\right),\quad 0 < \tau <1,\,\,\,b > 0,\, 0\leq \varphi <\pi,\, \zeta\in \mathbb{C},\\
&\left(\begin{bmatrix}
1 & 0\\
0 & e^{i\vartheta}
\end{bmatrix},
\begin{bmatrix}
a & r e^{i\varphi}\\
r e^{i\varphi} & d
\end{bmatrix}\right), \quad 0 < \theta <\pi,\,\,\,a, d>0,\, r\geq 0,\, 0\leq \varphi <\pi.
\end{align*}
Indeed, tangent spaces of these bundles are spanned by the tangent vectors in \ref{s100l} and \ref{s01t0} of the proof of Lemma \ref{lemalist} (for the appropriate parameters). 
\end{remark}

Note that using the list of normal forms in dimension $2$, recently a result on holomorphical flattenability of $CR$-nonminimal codimension $2$ real analytic submanifold near a complex point in $\mathbb{C}^{n}$, $n\geq 2$, was obtained through the works of Huang and Yin \cite{HY2,HY3}, Fang and
Huang \cite{FH}.

\section{Change of the normal form under small perturbations}\label{Sec3}

In this section we study how small deformations of a pair of one arbitrary and one symmetric matrix can change its orbit under the action (\ref{action1}) for $n=2$.

First recall that $(A,B)$, $(A',B')$ are in the same orbit with respect to the action (\ref{action1}) if and only if there exist $P\in GL_n(\mathbb{C}^n)$, $c\in \mathbb{C}^{*}$ such that $(A',B')=(cP^{*}AP,\overline{c}P^TBP)$. By real scaling $P$ we can assume that $|c|=1$, and after additional scaling $P$ by $\tfrac{1}{\sqrt{\overline{c}}}$ we eliminate the constant $\overline{c}$.
Thus the orbits of the action (\ref{action1}) are precisely the orbits of the action of $S^1 \times GL_n(\mathbb{C})$ acting on $\mathbb{C}^{n\times n}\times \mathbb{C}^{n\times n}_S $ by:
\begin{equation}\label{aAB}
\Psi\colon\big((c,P),(A,B)\big)\mapsto (cP^{*}AP,P^TBP), \quad P\in GL_2(\mathbb{C}),\, c\in S^{1}. 
\end{equation}
The projections are smooth actions as well (the second one is even holomorphic):
%
\begin{align}
\label{actionpsi1}
&\Psi_1\colon \big((c,P),(A,B)\big)\mapsto cP^{*}AP, \quad P\in GL_2(\mathbb{C}), \,c\in S^{1},\\
%
%
\label{actionpsi2}
&\Psi_2\colon\big((c,P),(A,B)\big)\mapsto P^TBP, \quad P\in GL_2(\mathbb{C}). 
\end{align}

Next, let $(A,B)$, $(A',B')$ be in the same orbit under the action (\ref{aAB}) ($A'=cP^{*}AP$, $B'=P^TBP$ for some $P\in GL_n(\mathbb{C})$, $c\in S^1$) and let $(E,F)$ be a perturbation of $(A,B)$:  
\[
cP^*(A+E)P=A'+cP^*EP, \qquad 
P^T(B+F)P=B'+P^TFP.
\]
A suitable perturbation of $(A,B)$ is in the orbit of an arbitrarily choosen perturbation of $(A',B')$.
It is thus sufficient to consider perturbations of normal forms.

Observe further that an arbitrarily small perturbation of $(\widetilde{A},\widetilde{B})$ is contained in $ \Orb_{\Psi}(A,B)$ if and only if $(\widetilde{A},\widetilde{B})$ (and hence the whole orbit $\Orb_{\Psi}(\widetilde{A},\widetilde{B})$) is contained in the closure of $ \Orb_{\Psi}(A,B)$. The same conclusion also holds for actions $\Psi_1$, $\Psi_2$.

For the sake of clarity the notion of a \emph{closure graph} for an action has been introduced. Given an action $\Phi$, the vertices in a closure graph for $\Phi$ are the orbits under $\Phi$, and there is an edge (a \emph{path}) from a vertex (an orbit) $\widetilde{\mathcal{V}}$ to a vertex (an orbit) $\mathcal{V}$ precisely when $\widetilde{\mathcal{V}}$ lies in the closure of $\mathcal{V}$. The path from $\widetilde{\mathcal{V}}$ to $\mathcal{V}$ is  denoted briefly by $\widetilde{\mathcal{V}}\to \mathcal{V}$. There are few evident properties of closure graphs:
\begin{itemize}
\item For every vertex $\mathcal{V}$, there exists  $\mathcal{V}\to \mathcal{V}$ (a trivial path),
\item Paths $\widehat{\mathcal{V}}\to\widetilde{\mathcal{V}}$ and $\widetilde{\mathcal{V}}\to \mathcal{V}$ imply the path $\widehat{\mathcal{V}}\to \mathcal{V}$ (usually not drawn).
\item If there is no path from $\widetilde{\mathcal{V}}$ to $\mathcal{V}$ (denoted by $\widetilde{\mathcal{V}}\not\to \mathcal{V}$), then for any vertex $\mathcal{W}$ it follows that either $\widetilde{\mathcal{V}}\not\to \mathcal{W}$ or $\mathcal{W}\not \to \mathcal{V}$ (or both).
\end{itemize}

To simplify the notation $\widetilde{\mathcal{V}}\to \mathcal{V}$, we usually write $\widetilde{V}\to V$, where $\widetilde{V}\in \widetilde{\mathcal{V}}$, $V\in \mathcal{V}$.
In case $\widetilde{V}\not\to V$ it is useful to know the distance of $\widetilde{V}$ from the orbit of $V$.  
We shall use the standard max norm $\|X\|=\max_{j,k\in \{1,\ldots n\}}|x_{j,k}|$, $X=[x_{j,k}]_{j,k=1}^{n}\in \mathbb{C}^{n\times n}$ To measure the distance between two matrices. This norm is not submultiplicative, but $\|XY\|\leq n\|X\|\,\|Y\|$ (see \cite[p. 342]{HornJohn}).

Proceed with basic properties of closure graph for the actions 
(\ref{aAB}), (\ref{actionpsi1}), (\ref{actionpsi2}). 

\begin{lemma}\label{pathlema}
Suppose $A,\widetilde{A},E \in \mathbb{C}^{n\times n}$, $B,\widetilde{B},F \in \mathbb{C}^{n\times n}_S$, and $p=|\det \widetilde{A}\det B|-|\det \widetilde{B}\det A|$.

\begin{enumerate}[label=(\arabic*),ref=\arabic*]

\item \label{pathlema1} There exists a path $\widetilde{A}\to A$ (a path $\widetilde{B}\to B$) in the closure graph for the action (\ref{actionpsi1}) (for the action (\ref{actionpsi2})) if and only if there exist sequences $P_j\in GL_n(\mathbb{C})$, $c_j\in S^{1}$ (a sequence $Q_j\in GL_n(\mathbb{C})$), such that
\begin{equation}\label{NC1}
c_jP_j^* A P_j\stackrel{j\to \infty}{\longrightarrow}\widetilde{A} \qquad (Q_j^T B Q_j\stackrel{j\to \infty}{\longrightarrow}\widetilde{B}).
\end{equation}
%
%
\begin{enumerate}[label=(\alph*),ref=\alph*]
\item \label{pathlema0} The existence of a path $\widetilde{A}\to A$ (a path $\widetilde{B}\to B$) implies the following:
\begin{enumerate}[label=(\roman*),ref=\roman*]
\item \label{pathlema1b} If $\det \widetilde{A}\neq 0$ (or $\det \widetilde{B}\neq 0$), then $\det A\neq 0$ (or $\det B\neq 0$). Apparently, $\widetilde{A}\neq 0$ (or $\widetilde{B}\neq 0$), then $ A\neq 0$ (or $ B\neq 0$).
\item \label{pathlema1a} $\dim \Orb_{\Psi_1}(A)>\dim \Orb_{\Psi_1}(\widetilde{A})\quad \big(\dim \Orb_{\Psi_2}(B)>\dim \Orb_{\Psi_2}(\widetilde{B})\big)$. 
%
\end{enumerate}
\item \label{pathlema1c} If $\det \widetilde{A}\neq 0$, $\det A= 0$, $\|E\|<\|\widetilde{A}^{-1}\|^{-1}$ ($\det \widetilde{B}\neq 0$, $\det B= 0$, $\|F\|< \|\widetilde{B}^{-1}\|^{-1}$), then $\widetilde{A}+F\not\in \Orb_{\Psi_1}(A)$ ($\widetilde{B}+F \not\in \Orb_{\Psi_2}(B)$). Trivially, if $\widetilde{A}\neq 0$, $\|E\|< \|\widetilde{A}\|$ (and $\widetilde{B}\neq 0$, $\|F\|<\|\widetilde{B}\|$), then $ \widetilde{A}+E\neq 0$ ($\widetilde{B}+F\neq 0$).
\end{enumerate}

\item \label{pathlema2} There exists a path $(\widetilde{A},\widetilde{B})\to (A,B)$ in the closure graph for an action (\ref{aAB}) precisely when there exist sequences $P_j\in GL_n(\mathbb{C})$, $c_j\in S^1$ such that
\begin{equation}\label{NC2}
c_jP_j^* A P_j\stackrel{j\to \infty}{\longrightarrow}\widetilde{A}, \qquad P_j^T B P_j\stackrel{j\to \infty}{\longrightarrow}\widetilde{B}.
\end{equation}
%
Morover, if $A$ and $\widetilde{A}$ ($B$ and $\widetilde{B}$) are in the same orbit and sufficiently close to each other, then it may be assumed in (\ref{NC2}) that $c_jP_j^* B P_j=\widetilde{A}$ ($P_j^T B P_j=\widetilde{B}$).
\begin{enumerate}[label=(\alph*),ref=\alph*]
\item \label{pathlemai2} When $(\widetilde{A},\widetilde{B})\to (A,B)$, it follows that
\begin{enumerate}[label=(\roman*),ref=\roman*]
%
%
\item \label{NC3} $\widetilde{A}\to A$,\quad $\widetilde{B}\to B$, \quad and \quad $p=0$. 
\item \label{NC4} $\dim \Orb_{\Psi}(A,B)>\dim \Orb_{\Psi}(\widetilde{A},\widetilde{B})$.
\end{enumerate}
\item \label{pathlema2ii} If $\widetilde{A}, \widetilde{B}, A, B\in GL_2(\mathbb{C})$ are such that $p\neq 0$ and $\|E\|<\max\big\{1, \frac{|p|}{4|\det B|(2\|\widetilde{A}\| +1 )} \big\}$, $\|F\|<\max\big\{1, 
\frac{|p|}{4|\det A|(2\|\widetilde{B}\| +1 )} \big\}$,
then $(\widetilde{A}+E,\widetilde{B}+F)\not \in \Orb_{\Psi}(A,B)$.
\end{enumerate}

\end{enumerate}

\end{lemma}
\vspace{-3mm}
\begin{proof}
By definition $\widetilde{A}\to A$ ($\widetilde{B}\to B$) if and only if $\widetilde{A}+F_j\in \Orb_{\psi_1}(A)$  ($\widetilde{B}+G_j\in \Orb_{\psi_2}(B)$) for some $G_j\to 0$ ($F_j\to 0$). It is equivalent to (\ref{NC1}) (see (\ref{actionpsi1}), (\ref{actionpsi2})), so the first part of (\ref{pathlema1}) is proved.
Apparently, $(\widetilde{A},\widetilde{B})\to (A,B)$ 
is then equivalent to (\ref{NC2}).

Since the orbit map of the action $\Psi_1$ (the action $\Psi_2$) is by Lemma \ref{propA} (\ref{propAb}) of constant rank and hence locally a submersion (see e.g. \cite[Theorem II.7.1]{Boot}), this action has the so-called local Lipschitz property, i.e. if $A$, $A'$ ($B$, $B'$) are sufficiently close and $cP^{*}AP=A'$ ($P^TBP=B'$), then $P$ can be chosen near to identity and $c$ near $1$. 
For any sufficiently small $E$ (or $F$) such that $\widetilde{A}$,  $\widetilde{A}+E$ ($\widetilde{B}$, $\widetilde{B}+F$) are in the same orbit, then there exists some $P$ close to the identity-matrix and $c$ close to $1$, so that $(\widetilde{A}+E,\widetilde{B}+F)$ is equal to $(cP^{*}\widetilde{A}P,\widetilde{B}+F)=(\widetilde{A},P^{-T}(\widetilde{B}+F)P^{-1})$ \big(equal to $(\widetilde{A}+E,P^{T}\widetilde{B}P)=(P^{-*}(\widetilde{A}+E)P^{-1},\widetilde{B})$\big).
As the inverse map $X\to X^{-1}$  
is continuous, $P^{-1}$ is close to identity-matrix, too. 
Hence $(\widetilde{A}+E,\widetilde{B}+F)$ is in the orbit of $(\widetilde{A},B+F')$ (or $(\widetilde{A}+E',\widetilde{B})$), where $F'=P^{-*}FP$ ($E'=P^{-T}EP$) is arbitrarily close to the zero-matrix. This concludes the proof of the first part of (\ref{pathlema2}).

Next, applying the determinant to (\ref{NC1}) we get 
\begin{equation}\label{detcP}
c_j^2|\det P_j|^2\det \widetilde{A}\stackrel{j\to \infty}{\longrightarrow} \det A, \qquad (\det Q_j^2\det \widetilde{B}\stackrel{j\to \infty}{\longrightarrow} \det B)
\end{equation}
%
for some $\{P_j\}_j\subset GL_n(\mathbb{C})$, $\{c_j\}_j\subset S^1$ ($\{Q_j\}_j\subset GL_n(\mathbb{C})$). 
This implies (\ref{pathlema1}) (\ref{pathlema0}) (\ref{pathlema1b}).

A necessary condition for $(\widetilde{A},\widetilde{B})$ to be in the closure of $\Orb_{\Psi} (A,B)$ is that $\widetilde{A}$ and $\widetilde{B}$ are in the closures of $\Orb_{\Psi_1} (A)$ and $\Orb_{\Psi_2} (B)$, respectively. Further, by multiplying the limits in (\ref{detcP}) for $P_j=Q_j$ by $\det B$ or $\det A$, and by  comparing the absolute values of the expressions, we deduce (\ref{pathlema2}) (\ref{pathlemai2}) (\ref{NC3}).

It is well known that the distance from a nonsingular matrix $X$ to the nearest singular matrix with respect to the norm $\|\cdot\|$  is equal to $\|X^{-1}\|^{-1}$ (see e.g. \cite[Problem  5.6.P47]{HornJohn}). Thus (\ref{pathlema1}) (\ref{pathlema1c}) follows.
Next, applying the triangle inequality, estimating the absolute values of the entries of the matrices by the max norm of the matrices, and by slightly simplifying, we obtain for $X,D\in \mathbb{C}^{2\times 2}$:
\begin{align}\label{detxe}
\big||\det (X+D)|-|\det (X)|\big|& \leq \big|\det (X+D)-\det (X )\big|\leq 
\|D\|\big (4\|X\| +2\|D\|\big).
\end{align}
Let $\widetilde{A}$, $B$, $\widetilde{B}$, $A$ be nonsingular matrices,  
%
$
p=|\det \widetilde{A}\det B|- |\det \widetilde{B}\det A|\neq 0.
$
%
Using (\ref{detxe}) for $X=\widetilde{A}$, $D=E$ and $X=\widetilde{B}$, $D=F$ with $\|E\|,\|F\|\leq 1$, respectively, 
we get
\begin{align}\label{ocenaAB}
\big||\det (\widetilde{A}+E)\det B|-|\det (\widetilde{A})\det B|\big| 
                                            & \leq \big( 4 \|\widetilde{A}\|\, \|E\| +2\|E\| ^2 \big)|\det B| \\
                                            & \leq \|E\|\,|\det B|\big(4\|\widetilde{A}\| +2 \big),\nonumber
\end{align}
%
\begin{align}\label{ocenaBA}
\big||\det (\widetilde{B}+F)\det A|-|\det (\widetilde{B})\det A|\big| 
\leq \|F\|\,|\det A|\big(4\|\widetilde{B}\| +2 \big).
\end{align}
To estimate the left-hand sides of the above inequalities from above by $\frac{|p|}{2}$ it suffices to take $\|E\|<\frac{|p|}{4|\det B|(2\|\widetilde{A}\| +1 )}$ and $\|F\|<\frac{|p|}{4|\det A|(2\|\widetilde{B}\| +1 ) }$. By combining (\ref{ocenaAB}), (\ref{ocenaBA}) and applying the triangle inequality we then conclude 
\[
\Big|\big|\det (\widetilde{A}+E)\det B\big|- \big|\det (\widetilde{B}+F)\det A\big|\Big|>0.
\]
By (\ref{pathlema2}) (\ref{pathlemai2}) (\ref{NC3}) (allready proved) we have $(\widetilde{A}+E,\widetilde{B}+F)\not\to (A,B)$ and this gives (\ref{pathlema2}) (\ref{pathlema2ii}).

It is left to show the orbit-dimension inequalities (\ref{pathlema1}) (\ref{pathlema0}) (\ref{pathlema1a}) and (\ref{pathlema2}) (\ref{pathlemai2}) (\ref{NC4}). 
Since orbits of $\Psi$, $\Psi_1$, $\Psi_2$ can be seen as nonsingular algebraic subsets in Euclidean space (zero loci of polynomials), 
these facts can be deduced by using a few classical results in real (complex) algebraic geometry \cite[Propositions 2.8.13,2.8.14]{RAG} (or \cite[Propositions 21.4.3, 21.4.5]{Tauvel}, \cite[Exercise 14.1.]{Harris}). Indeed, orbits $\Orb_{\Psi} (\widetilde{A},\widetilde{B})$, $\Orb_{\Psi_1}(\widetilde{A})$, $\Orb_{\Psi_2}(\widetilde{B})$ are contained in the closures (also with respect to a coarser Zariski topology) of orbits $\Orb_{\Psi}(A,B)$, $\Orb_{\Psi_1}(A)$, $\Orb_{\Psi_2}(B)$, respectively. Hence algebraic dimensions of orbits mentioned first are strictly smaller than algebraic dimensions of the later orbits. 
Finally, orbits are locally regular submanifolds and their manifold dimensions agree with their algebraic dimensions.
\end{proof}

\begin{remark}\label{remarkocena}
Lemma \ref{pathlema} provides a quantitative information on the distance from an element to another orbit. We observe that it suffices to consider only normal forms.
Given any $Q\in GL_n(\mathbb{C})$ the induced operator norms of maps $\mathbb{C}^{n\times n}\to \mathbb{C}^{n\times n}$, $X\mapsto Q^{*}XQ$ and $\mathbb{C}^{n\times n}_S\to \mathbb{C}^{n\times n}_S$, $X\mapsto Q^{T}XQ$ are bounded from above by $n^{2}\|Q^{*}\|\,\|Q\|$ and $n^{2}\|Q^{T}\|\,\|Q\|$, respectively. 
If $(\widetilde{A},\widetilde{B})$, $(\widetilde{A},\widetilde{B})$ are in the same orbit (i.e. $\widetilde{A}=\widehat{c}Q^{*}\widehat{A}Q$, $\widetilde{B}=Q^{T}\widehat{B}Q$ with $Q\in GL_n(\mathbb{C})$, $\widehat{c}\in S^{1}$), then for any $P\in GL_n(\mathbb{C})$, $c\in S^{1}$ we get:
\begin{align*}
&\|\widehat{A}-cP^{*}AP\|\geq \tfrac{1}{n^{2}\|Q^{*}\|\|Q\|}\|Q^{*}\widehat{A}Q-cQ^*P^{*}APQ\|\geq \tfrac{1}{n^{2}|\widehat{c}|\,\|Q^{*}\|\,\|Q\|}\|\widetilde{A}-c\widehat{c}(PQ)^{*}APQ\|,\\
&\|\widehat{B}-P^{T}BP\|\geq \tfrac{1}{n^{2}\|Q^{T}\|\,\|Q\|}\|Q^{T}\widehat{B}Q-Q^TP^{T}BPQ\|\geq \tfrac{1}{n^{2}\|Q^{T}\|\,\|Q\|}\big\|\widetilde{B}-(PQ)^{T}B(PQ)\big\|.
\end{align*}

When inspecting $(\widetilde{A},\widetilde{B})\to (A,B)$ where either $A$, $\widetilde{A}$ or $B$, $\widetilde{B}$ are in the same orbit and sufficiently close, it is by (\ref{pathlema2}) enough to analyse perturbations of the matrix $\widetilde{A}$ (the matrix $\widetilde{B}$). Unfortunately, we do not know how close the matrices should be since the constant rank theorem (even the quantitative version \cite[Theorem 2.9.4]{Hub}) does not provide  the size of the local charts which define the orbits.  
\end{remark}

By Autonne-Takagi factorization (see e.g. \cite[Corolarry
4.4.4]{HornJohn}), any complex symmetric matrix is
unitary $T$-congruent to a diagonal matrix with non-negative diagonal entries, hence $T$-congruent to a diagonal matrix with ones and zeros on the diagonal. 
Therefore $2\times 2$ symmetric matrices with respect to $T$-conjugacy consist of three orbits, each containing matrices of the same rank. Their closure graph is thus very simple. (For closure graphs of all $2\times 2$ or $3\times 3$ matrices see \cite{FKS1}.)

\begin{lemma}\label{lemapsi2}
The closure graph for the action (\ref{actionpsi2}) ($T$-conjugacy on $\mathbb{C}^{2\times 2}_S$) is 
\[
0_2 \to 1\oplus 0\to I_2,
\]
where $\dim \Orb (1\oplus 0)=2$, $\dim \Orb (I_2)=3$. 
Moreover, if $\widetilde{B}$, $B$ are vertices in the above graph, and such $\widetilde{B}\not \to B$ with $P^T BP=\widetilde{B}+F$ for some $P\in GL_2(\mathbb{C})$, then $\|F\|\geq 1$.
\end{lemma}

The (non)existence of most paths in the closure graph for the action $\Psi_1$ in (\ref{actionpsi1}) follows immediately from the (non)existence of paths in the closure graph for $*$-conjugacy (\cite[Theorem 2.2]{FKS2}). The remaining paths are treated by a slight adaptation of the $*$-cojugacy case. By a careful analysis we provide neccesarry (sufficient) conditions for the existing paths; see Lemma \ref{lemapsi1} (its proof is given in Sec. \ref{proofL}). These turn out to be essential in the proof of Theorem \ref{izrek}. Furthermore, if $\widetilde{A}\not\to A$ we find a lower bound for the distance from $\widetilde{A}$ to $\Orb_{\Psi_1}(A)$. Note that normal forms for $\Psi_1$ were first observed by Coffman \cite[Theorem 4.3]{Coff2}, and by calculating their stabilizers eventually normal forms for the action (\ref{aAB}) were obtained (\cite[Subsection 2.4]{Coff}).

\begin{lemma}\label{lemapsi1}
Closure graph for the action (\ref{actionpsi1}) is the following ($0<\theta<\pi$, \,\,$0<\tau<1$):
\vspace{-1mm}
\small
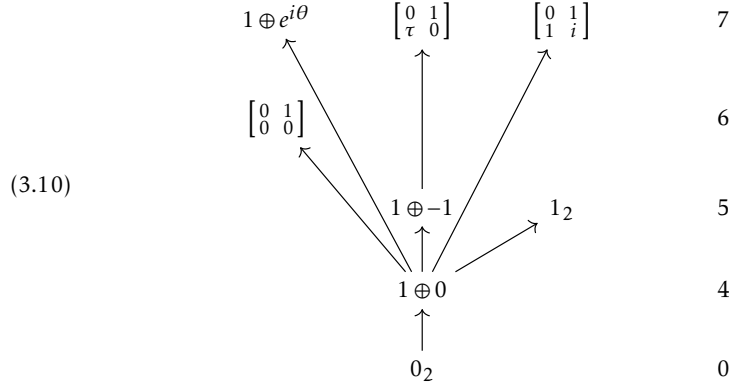
\begin{equation}\label{graph1}
\begin{tikzcd}
1\oplus e^{i\theta} & 
\begin{bsmallmatrix}
0 & 1\\
\tau & 0
\end{bsmallmatrix} & 
\begin{bsmallmatrix}
0 & 1\\
1 & i
\end{bsmallmatrix} & \qquad 7\\
\begin{bsmallmatrix}
0 & 1\\
0 & 0
\end{bsmallmatrix}&  & & \qquad 6 \\
 & 
1\oplus -1\arrow[uu, ""] & 
1_2
& \qquad 5 \\
& 
1\oplus 0
\arrow[u, ""] \arrow[ur, ""] \arrow[uul, ""] \arrow[uuur, ""]\arrow[uuul, ""] & & \qquad 4\\
 & 
0_2
\arrow[u, ""]  & & \qquad 0
\end{tikzcd}
\end{equation}
\normalsize
\vspace{-1mm}
\noindent
The closure graph contains an infinite set of vertices corresponding to the orbits with normal forms 
$1\oplus e^{i\theta}$, 
$\begin{bsmallmatrix}
0 & 1\\
\tau & 0
\end{bsmallmatrix}$,
indexed by the parameters $\theta$, $\tau$, respectively.
Orbits at the same horizontal level have the same dimension and these are indicated on the right. 

Furthermore, let $\widetilde{A}$, $A$ be normal forms in (\ref{graph1}), and let $E=cP^{*} AP-\widetilde{A}$ for some $c\in S^{1}$, $P=\begin{bsmallmatrix}
x & y\\
u & v
\end{bsmallmatrix}\in GL_2(\mathbb{C})$, $E\in \mathbb{C}^{2\times 2}$. Then one of the following statements holds:
\begin{enumerate}
\item \label{lemapsi1a} If $\widetilde{A}\not\to A$, then there exists a positive constant $\mu$ such that $\|E\|\geq \mu$.
\item \label{lemapsi11} If $\widetilde{A}\to A$, then there is a positive constant $\nu$ such that the moduli of the expressions (depending on $c$, $P$) listed in the fourth column (and in the line corresponding to $\widetilde{A}$, $A$) of the table below are bounded from above by $\nu\|E\|$. (If $A,\widetilde{A}\in GL_2(\mathbb{C})$ then also $\|E\|\leq \tfrac{|\det \widetilde{A}|}{8\|\widetilde{A}\|+4}$ is assumed.) 
Conversely, if $\widetilde{A}$, $A$ correspond to any of the lines  (C\ref{r3}), (C\ref{r4}), (C\ref{r7}), (C\ref{r10}), then there exists a positive constant $\rho$ such that: if the moduli of expressions listed in the fourth column of this line  
are bounded from above by $s\in (0,1]$, then $\|E\|\leq \rho \sqrt{s}$.
\end{enumerate}
\vspace{-3mm}
\small
\begin{center}		
\begin{tabular}{|c|c|c|l|l|}
\hline
 &	$\widetilde{A}$ & $A$ &              &    \\
\hline
\refstepcounter{rownumber}
\label{r3}  C\ref{r3}  & 
$
\alpha \oplus 0
$
&
$
1 \oplus e^{i\theta} 
$
& 
 $|x|^2+e^{i\theta}|u|^2-c^{-1}\alpha$, $y^2,v^2$ & $\scriptstyle{\alpha\in\{0,1\}}$, $\scriptstyle{0\leq \theta<\pi}$\\
		\hline
\refstepcounter{rownumber}
\label{r4}  C\ref{r4} &
$
\alpha\oplus 0
$ &
$
\begin{bsmallmatrix}
0 & 1 \\
\tau &  0
\end{bsmallmatrix}
$
&
$\overline{y}v$, $\overline{x}v$, $\overline{u}y$, $(1+\tau)\Rea(\overline{x}u)+i(1-\tau)\Ima(\overline{x}u)-\frac{\alpha}{c}$ & $\scriptstyle{0\leq \tau< 1,\alpha\in\{0,1\}}$\\
\hline
\refstepcounter{rownumber}
\label{r7}  C\ref{r7}  &	
$
\begin{bsmallmatrix}
\alpha & \beta \\
\beta & \omega 
\end{bsmallmatrix}
$
&
$
\begin{bsmallmatrix}
0 & 1 \\
1 & 0
\end{bsmallmatrix}
$ 
&        $2\Rea(\overline{y}v)-(-1)^{k}\omega$, $2\Rea(\overline{x}u)-(-1)^{k}\alpha$ & $\scriptstyle{k\in \mathbb{Z}}$;\,\,$\scriptscriptstyle{\alpha=\omega=0,\beta=1\textrm{ or }}$  \\
&    &	    &  $(\overline{x}v+\overline{u}y)-(-1)^{k}\beta$   &  $\scriptscriptstyle{\beta=0,\alpha\in \{0,1\},\omega\in\{0,\pm\alpha\}}$ \\
\hline	
\refstepcounter{rownumber}
\label{r10}  C\ref{r10}  &	$
\alpha\oplus 0
$
&
$
\begin{bsmallmatrix}
0 & 1 \\
1 & i
\end{bsmallmatrix}
$ 
&
$\overline{x}v+\overline{u}y$, $\overline{u}v$, $\Rea(\overline{y}u)$, $v^{2}$, $2\Rea(\overline{x}u)+i|u|^{2}-\frac{\alpha}{c}$ & $\scriptstyle{\alpha\in \{0,1\}}$   \\
\hline
 \refstepcounter{rownumber}
\label{r1}  C\ref{r1}  &
$
1 \oplus e^{i\theta} 
$ &
$
1 \oplus e^{i\theta} 
$
&
$u^{2},y^{2}$, $|x|^{2}-1$, $|v|^{2}-1$ 
  &   $\scriptstyle{0<\theta<\pi}$
\\
		\hline
\refstepcounter{rownumber}
\label{r5}  C\ref{r5} &	
$
\begin{bsmallmatrix}
\alpha & \beta \\
\beta & \omega 
\end{bsmallmatrix}
$ &
$
\begin{bsmallmatrix}
0 & 1 \\
1 & i 
\end{bsmallmatrix}
$ 
& $2\Rea(\overline{x}u)-(-1)^{k}\alpha$, $2\Rea(\overline{y}v)-(-1)^k\Rea (\omega)$ &  $\scriptstyle{k\in \mathbb{Z}}$;\,$\scriptscriptstyle{\beta=1,\alpha=0,\omega\in\{0,i\}}$ \\
&   &  	 & 
$\overline{x}v+\overline{u}y-(-1)^k\beta$, $u^{2}$ , $|v|^2-(-1)^k\Ima (\omega)$ &  $\scriptscriptstyle{\textrm{ or }\beta=0,-\omega=\alpha\in \{0,1\}}$
   \\
\hline
\refstepcounter{rownumber}
\label{r6}  C\ref{r6}  &
$
\begin{bsmallmatrix}
0 & 1 \\
\tau & 0
\end{bsmallmatrix}
$
&
$
\begin{bsmallmatrix}
0 & 1 \\
\tau & 0
\end{bsmallmatrix}
$  & $\overline{x}u$, \,$\overline{y}v$,\, $\overline{y}u$\, $\overline{v}x-c^{-1}$ &   $\scriptstyle{0\leq\tau< 1}$ \\
\cline{4-5}
  &   &  &   $c-(-1)^{k}$ &   $\scriptstyle{0<\tau< 1}$,\,\,$\scriptstyle{k\in \mathbb{Z}}$ \\   
 \hline
\refstepcounter{rownumber}
\label{r9}  C\ref{r9}  &	$
\alpha\oplus \omega
$
&
$
1\oplus \sigma
$
 &
 $(|x|^2+\sigma|u|^2)-c^{-1}\alpha$, $\overline{x}y+\sigma \overline{u} v$  &  $\scriptscriptstyle{\alpha=1,\omega\in \{\sigma,0\} \textrm{ or } \alpha=\omega=0}$  \\
&  &   	& $(|y|^2+\sigma |v|^2)-c^{-1}\omega$     & $\scriptstyle{\sigma\in\{1,-1\}}$ \\
\cline{4-5}
  &   &  &   $c-(-1)^{k}$ \quad ($k=0\textrm{ for }\sigma=1$) &  $\scriptstyle{\sigma=\omega\in\{1,-1\}}$, $\scriptstyle{k\in \mathbb{Z}}$ \\
\hline
\refstepcounter{rownumber}
\label{r11}  C\ref{r11}  & 	$
\alpha\oplus 0
$ &
$
1\oplus 0
$  
&   $y^2$, $|x|^2-\alpha$  &  $\scriptstyle{\alpha\in\{0,1\}}$ \\
\cline{4-5}
 &  &  & $c- 1$   &  $\scriptstyle{\alpha=1}$, $\scriptstyle{\|E\|\leq \frac{1}{2} }$\\
\hline
\refstepcounter{rownumber}
\label{r12}  C\ref{r12}  &	$
\begin{bsmallmatrix}
0 & 1 \\
1 & 0
\end{bsmallmatrix}
$ &
$
1\oplus -1
$ 
& $\overline{x}y-\overline{u}v-(-1)^{k}$, $|x|^2-|u|^2$, $|y|^2-|v|^2$  &  \\
		\hline		
	\end{tabular}
\end{center}
\normalsize
\end{lemma}

\begin{remark} 
Constants $\mu$ or $\nu$ in Lemma \ref{lemapsi1} are calculated for any given pair $\widetilde{A},A$ (see Lemma \ref{pathlema} and the proof of Lemma \ref{lemapsi1}). The existence of the constant $\rho$ (computable as well) in the converse in Lemma \ref{lemapsi1} (\ref{lemapsi11}) is showed only in those cases where it turns to be useful in the proof of Theorem \ref{izrek}, though it is expected to be proved (possibly by a slight modification) for most of the cases.
\end{remark}

We are ready to state the main results of the paper. The following theorem de\-scri\-bes the closure graph for the action (\ref{aAB}). Its proof is given in Sec. \ref{proofT}.
It is expected that by adapting the proof a similar result should hold for the re\-stri\-ction of the action (\ref{aAB}) with $c=1$; in this case there are few more types of orbits. 

\begin{theorem}\label{izrek}
Orbits with normal forms from Lemma \ref{lemalist} are vertices in the closure graph for the action (\ref{aAB}). The graph has the following properties:
\begin{enumerate}
\item If $(\widetilde{A},\widetilde{B})\not \to (A,B)$, it is possible to provide some lower bound for the distance from $(\widetilde{A},\widetilde{B})$ to the orbit of $(A,B)$. 
%
\item \label{izrek3} There is a path from $(0_2,0_2)$ to any orbit. There exist paths from $(1\oplus 0,0_2)$ to all orbits, except to $\big(1\oplus 0, 
\begin{bsmallmatrix}
0 & 1\\
1 & 0
\end{bsmallmatrix}\big)$, 
$\big( 
\begin{bsmallmatrix}
0 & 1\\
\tau & 0
\end{bsmallmatrix}, 
\begin{bsmallmatrix}
0 & b\\
b & 0
\end{bsmallmatrix}\big)$ or $b>0$, and to $(0_2,B)$, $B\in \mathbb{C}_S^{2\times 2}$.
\item \label{izrek4} There exist paths from $(1\oplus 0,\widetilde{a}\oplus 0)$ with $\widetilde{a}>0$ to all orbits, except to $\big(1\oplus 0, 
\begin{bsmallmatrix}
0 & 1\\
1 & 0
\end{bsmallmatrix}\big)$, 
$\big(
\begin{bsmallmatrix}
0 & 1\\
1 & i
\end{bsmallmatrix},0\oplus d\big)$ for $d>\widetilde{a}$, $(0_2,B)$ for $B\in \mathbb{C}_S^{2\times 2}$,
$\big( 
\begin{bsmallmatrix}
0 & 1\\
\tau & 0
\end{bsmallmatrix}, 
\begin{bsmallmatrix}
0 & b\\
b & 0
\end{bsmallmatrix}\big)$ for $b>0$ such that $\widetilde{a}\not\in [\frac{2b}{1+\tau},\frac{2b}{1-\tau}]$, and to 
$\big(1\oplus e^{i\theta}, 
\begin{bsmallmatrix}
a & b\\
b & d
\end{bsmallmatrix} \big)$ for $0\leq \theta<\pi$ and such that $\widetilde{a}>M$, where $M$ is the maximum of the function given with a constraint
\begin{equation}\label{izrazf7}
f(r,t,\beta)=|ar^{2}e^{i\beta}+2brs+ds^{2}e^{-i\beta}|, \quad r^{4}+2r^{2}t^{2}\cos \theta + t^{4}=1, \,\, r,t\geq 0,\,\, \beta \in \mathbb{R}.
\end{equation}
\item All nontrivial paths $(\widetilde{A},\widetilde{B})\not \to (A,B)$ with nontrivial $(\widetilde{A},\widetilde{B})\neq (1\oplus 0,\widetilde{a}\oplus 0)$ for $\widetilde{a}\geq 0$ (not mentioned in (\ref{izrek3}), (\ref{izrek4})) are noted in the following two diagrams. Orbits at the same horizontal level have equal dimension (indicated on the right).
\end{enumerate}  
\end{theorem}
\vspace{-3mm}
\small
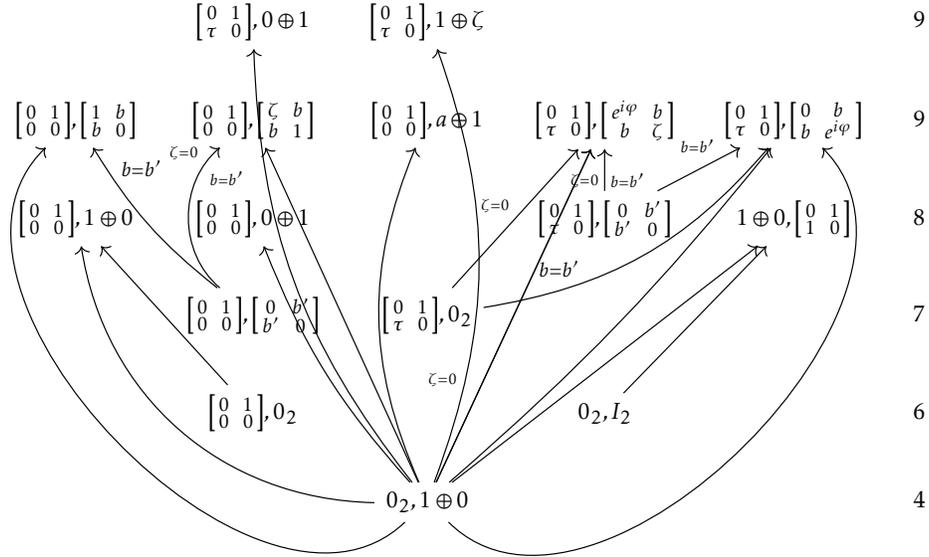
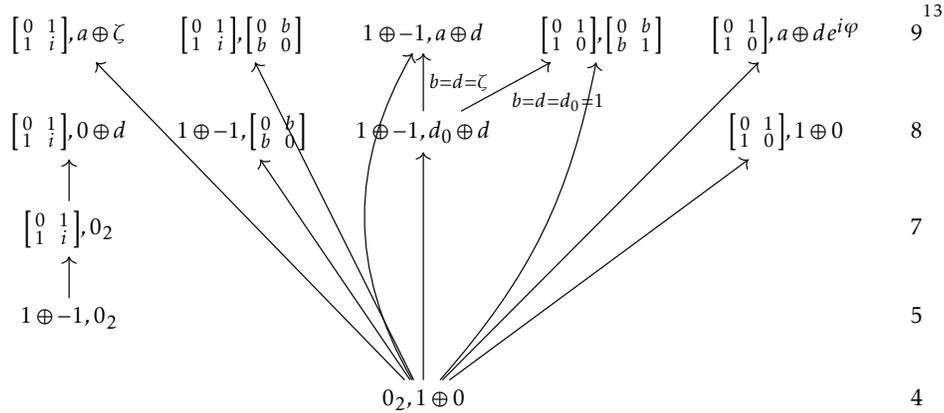
\begin{equation*}\label{graph3}
\hspace*{-12mm}
\begin{tikzcd}[column sep=small]
 & 
 \begin{bsmallmatrix}
0 & 1\\
\tau & 0
\end{bsmallmatrix}, 0\oplus 1 &  
\begin{bsmallmatrix}
0 & 1\\
\tau & 0
\end{bsmallmatrix}, 1\oplus \zeta & & & 9\\
\begin{bsmallmatrix}
0 & 1\\
0 & 0
\end{bsmallmatrix}, 
\begin{bsmallmatrix}
1 & b\\
b & 0
\end{bsmallmatrix} & 
\begin{bsmallmatrix}
0 & 1\\
0 & 0
\end{bsmallmatrix}, 
\begin{bsmallmatrix}
\zeta & b\\
b & 1
\end{bsmallmatrix} & 
\begin{bsmallmatrix}
0 & 1\\
0 & 0
\end{bsmallmatrix} ,a\oplus 1  & 
\begin{bsmallmatrix}
0 & 1\\
\tau & 0
\end{bsmallmatrix}, 
\begin{bsmallmatrix}
e^{i\varphi} & b\\
b & \zeta
\end{bsmallmatrix}  & 
\begin{bsmallmatrix}
0 & 1\\
\tau & 0
\end{bsmallmatrix} ,\begin{bsmallmatrix}
0 & b\\
b & e^{i\varphi}
\end{bsmallmatrix} & 9\\
\begin{bsmallmatrix}
0 & 1\\
0 & 0
\end{bsmallmatrix},1\oplus 0
& 
\begin{bsmallmatrix}
0 & 1\\
0 & 0
\end{bsmallmatrix}, 0\oplus 1
& &
  \begin{bsmallmatrix}
0 & 1\\
\tau & 0
\end{bsmallmatrix} ,\begin{bsmallmatrix}
0 & b'\\
b' & 0
\end{bsmallmatrix} \arrow[u,near start,"\scriptscriptstyle{b=b'}"',"\scriptscriptstyle{\zeta=0}"] \arrow[ur,near end,"\scriptscriptstyle{b=b'}"]
&
1\oplus 0 ,\begin{bsmallmatrix}
0 & 1\\
1 & 0
\end{bsmallmatrix}
&   8 \\
& 
\begin{bsmallmatrix}
0 & 1\\
0 & 0
\end{bsmallmatrix} ,\begin{bsmallmatrix}
0 & b'\\
b' & 0
\end{bsmallmatrix}\arrow[uul, swap,pos=0.8,bend left=10,"b=b'"]\arrow[uu,swap,pos=0.85,bend left=50,"\scriptscriptstyle{b=b'}","\scriptscriptstyle{\zeta=0}"']
&
\begin{bsmallmatrix}
0 & 1\\
\tau & 0
\end{bsmallmatrix} ,0_2 \arrow[uur, "\scriptscriptstyle{\zeta=0}"]\arrow[uurr, bend right=20,pos=0.3, "b=b'"]
& & & 7 \\
& \begin{bsmallmatrix}
0 & 1\\
0 & 0
\end{bsmallmatrix} ,0_2 \arrow[uul, ""] & & 0_2,I_2 \arrow[uur,""]
&  & 6 \\
 & 
  & 0_2,1\oplus 0
\arrow[uuuull,bend left=90, ""] \arrow[uuull, bend left=40, ""] \arrow[uuul, bend left=10,""] \arrow[uuuurr,bend right=90, ""]
 \arrow[uuuur, ""]\arrow[uuuurr, bend right=5 ""] \arrow[uuurr, ""]\arrow[uuuur, ""]\arrow[uuuu, bend left=23, ""]\arrow[uuuul, ""]\arrow[uuuuul, bend left=18,""]\arrow[uuuuu, bend right=20, "\scriptscriptstyle{\zeta=0}",pos=0.2]
& & & 4
\end{tikzcd}
\end{equation*}
\normalsize
\vspace{-1cm}
\small
\begin{equation*}\label{graph2}
\begin{tikzcd}[column sep=small]
\begin{bsmallmatrix}
0 & 1\\
1 & i
\end{bsmallmatrix}, a\oplus \zeta & 
\begin{bsmallmatrix}
0 & 1\\
1 & i
\end{bsmallmatrix}, 
\begin{bsmallmatrix}
0 & b\\
b & 0
\end{bsmallmatrix} & 
1\oplus -1 ,a\oplus d  & 
\begin{bsmallmatrix}
0 & 1\\
1 & 0
\end{bsmallmatrix}, 
\begin{bsmallmatrix}
0 & b\\
b & 1
\end{bsmallmatrix}  & \begin{bsmallmatrix}
0 & 1\\
1 & 0
\end{bsmallmatrix} ,a\oplus d e^{i\varphi} & 9\\
\begin{bsmallmatrix}
0 & 1\\
1 & i
\end{bsmallmatrix}, 0\oplus d 
& 1\oplus -1, \begin{bsmallmatrix}
0 & b\\
b & 0
\end{bsmallmatrix}
& 1\oplus -1,d_0\oplus d \arrow[u, swap,"b=d=\zeta"]
                         \arrow[ur, swap,"b=d=d_0=1"]
& 
 &
  \begin{bsmallmatrix}
0 & 1\\
1 & 0
\end{bsmallmatrix} ,1\oplus 0 & 8 \\
\begin{bsmallmatrix}
0 & 1\\
1 & i
\end{bsmallmatrix}, 0_2 \arrow[u, ""] & 
& 
&
& & 7 \\
1\oplus -1,0_2 \arrow[u, ""] & 
& &   &  & 5\\
 & 
  & 0_2,1\oplus 0
\arrow[uuuu,bend left, ""] \arrow[uuuull, ""] \arrow[uuuul, ""] \arrow[uuu, ""] \arrow[uuul, ""] \arrow[uuuur,bend right= 15, ""] \arrow[uuuurr, ""] \arrow[uuurr, ""] & & & 4
\end{tikzcd}
\end{equation*}
\normalsize


\begin{remark}
The non-existence of some paths in the closure graph for the action (\ref{aAB}) follow immediately from Lemma \ref{lemapsi2}, Lemma \ref{lemapsi1} and Lemma \ref{pathlema} (\ref{pathlemai2}) (\ref{pathlema2}) (\ref{NC3}), (\ref{NC4}). However, to get some lower bound on the distance from a certain normal form to another orbit, many intrigueing and sometimes tedious estimates need to be done (see the proof of Theorem \ref{izrek} in Sec. \ref{proofT}).
It makes the proof of the theorem much longer and more involved. Moreover, one must also work with inequalities instead of inspecting the convergence of sequences (see Lemma \ref{pathlema}).

It would be interesting to have the closure graph for bundles of matrices with respect to the action (\ref{aAB}) (see (\ref{actionL})) or its restriction. One would need to consider the same equations as in our case, but possibly with less constraints. 
\end{remark}

%
%
%
%
%

Let $M$ be a compact codimension $2$ submanifold in a complex manifold $X$. Since the complex dimension of the complex tangent spaces near a regular point of $M$ is preserved under small perturbations, complex points of a small deformation $M'$ of $M$ can arise only near complex points of $M$.
Recall that $(\widetilde{A},\widetilde{B})\in \mathbb{C}^{2 \times 2}\times \mathbb{C}_S^{2 \times 2}$ (a normal form up to quadratic terms) can be associated to a complex point $p\in M$, so that in a neighborhood of $p$ the submanifold $M$ is of the form (\ref{BasForm1}) for $A=\widetilde{A}$, $B=\widetilde{B}$. If $M'$ is a $\mathcal{C}^2$-small deformation of $M$, then it is seen near $p$ as a graph:
\begin{equation*}
w=w_0+s^T\overline{z}+r^Tz+\overline{z}^T\widehat{A}z+\tfrac{1}{2}\overline{z}^T\widehat{B}\overline{z}+\tfrac{1}{2}z^T\widetilde{C}z+o(|z|^2), \quad (w(p),z(p))=(0,0),
\end{equation*}
where $(z,w)$ are local coordinates with $w_0\in \mathbb{C}$, $r,s\in \mathbb{C}^2$, $\widehat{C}\in \mathbb{C}^{2\times 2}$ small, and $\widehat{A}\in \mathbb{C}^{2 \times 2}$, $\widehat{B}\in \mathbb{C}_S^{2 \times 2}$ close to $\widetilde{A}$, $\widetilde{B}$, respectively. Similarily as in the exposition in Sec. \ref{intro} a complex point on this graph are put into the standard position (\ref{BasForm1}) for $A=\widehat{A}$, $B=\widehat{B}$. (Translate the complex point to $(0,0)$, use a complex-linear transformation close to identity to insure the tangent space at $(0,0)$ to be $\{w=0\}$, and finally eliminate $z$-terms.) The next result is hence a direct consequence of Theorem \ref{izrek}.

\begin{posledica}\label{pospert}
Let $M$ be a compact real $4$-manifold embedded $\mathcal{C}^2$-smoothly in a complex $3$-manifold $X$ and let $p_1,\ldots,p_k\in M$ be its isolated complex points with the corresponding normal forms up to quadratic terms $(A_1,B_1),\ldots,(A_k,B_k)\in \mathbb{C}^{n\times n}\times \mathbb{C}_S^{n\times n}$. Asumme further that $M'$ is a deformation of $M$ obtained by a smooth isotopy of $M$, and let $p\in M'$ be a complex point with the corresponding normal form $(A,B)$. If the isotopy is sufficiently $\mathcal{C}^2$-small then $p$ is arbitrarily close to some $p_{j_0}$, $j_{0}\in \{1,\ldots,k\}$, and there is a path $(A_{j_{0}},B_{j_{0}})\to(A,B)$ in the closure graph for the action (\ref{aAB}).
\end{posledica}

\begin{remark}
In the proof of Theorem \ref{izrek} the lower estimates for the distances from normal forms to other orbits are provided, therefore it can be told how small the isotopy $M'$ (of $M$) in the assumption of Corollary \ref{pospert} needs to be.
\end{remark}

\vspace{-3mm}
\section{Proof of Lemma \ref{lemapsi1}}\label{proofL}

In this section we prove Lemma \ref{lemapsi1}. We start with the  following technical lemma related to actions $(\ref{actionpsi1})$ and $(\ref{actionpsi2})$.

\begin{lemma}\label{lemadet}
\begin{enumerate}
\item \label{lemadeta} Suppose $\widetilde{A},A,P\in GL_2(\mathbb{C})$, $E\in \mathbb{C}^{2\times 2}$, $c\in S^{1}$ and such that $cP^*AP=\widetilde{A}+E$,
$\|E\|\leq 
\tfrac{|\det \widetilde{A}|}{8\|\widetilde{A}\|+4}$. Denote further $\Delta=\arg \bigl(\tfrac{\det \widetilde{A}}{\det A}\bigr)$. It then follows that
%
\begin{align}\label{cE}
&c= (-1)^{k}e^{ \frac{i\Delta}{2}}+g, \quad c^{-1}= (-1)^{k}e^{-\frac{i\Delta}{2}}+\overline{g}, \qquad k\in \mathbb{Z},\,\,
|g|\leq \tfrac{\|E\|(8\|\widetilde{A}\|+4)}{|\det \widetilde{A}|},\\
\label{PAE}
&|\det P|=\bigl|\tfrac{\det \widetilde{A}}{\det A}\bigr|^{\frac{1}{2}}+r, \qquad |r|\leq \tfrac{\|E\|(4\|\widetilde{A}\|+2)}{\sqrt{|\det \widetilde{A}\det A}|}.
\end{align}
\item \label{lemadetb} If $F\in \mathbb{C}^{2\times 2}$, $\widetilde{B},B,P\in GL_2(\mathbb{C})$ and such that $P^TAP=\widetilde{B}+F$,
$\|F\|\leq 
\tfrac{|\det \widetilde{B}|}{8\|B\|+4}
$, then
\begin{align}\label{PBF}
\det P=\sqrt{\tfrac{\det \widetilde{B}}{\det B}}+r, \qquad |r|\leq \tfrac{\|E\|(4\|\widetilde{B}\|+2)}{\sqrt{|\det \widetilde{B}\det B}|}.
\end{align}
\end{enumerate}
\end{lemma}

\begin{proof}
First, we observe the following simple fact. For $\xi,\zeta,h\in \mathbb{C}$:
\begin{equation}\label{ocenah}
\xi=\zeta+h, \,|h|\leq \tfrac{|\zeta|}{2}\neq 0 \quad\textrm{implies}\quad \arg (\xi)-\arg(\zeta)=\psi \in (-\tfrac{\pi}{2},\tfrac{\pi}{2}),\,\,|\sin \psi|\leq 2|\tfrac{h}{\zeta}|.
\end{equation}
Indeed, we have $\xi\zeta^{-1}=1+\frac{h}{\zeta}=|1+\frac{h}{\zeta}|e^{i \psi}$ with $|\frac{h}{\zeta}|\leq \frac{1}{2}$, hence $\psi \in (-\tfrac{\pi}{2},\tfrac{\pi}{2})$ and 
$|\sin \psi|=\bigl|\Ima \big(\tfrac{1+\frac{h}{\zeta}}{|1+\frac{h}{\zeta}|}\big)\bigr|
\leq 
\tfrac{|\Ima \frac{h}{\zeta}|}{|1+\frac{h}{\zeta}|} 
\leq \tfrac{|\frac{h}{\zeta}|}{1-|\frac{h}{\zeta}|}\leq \tfrac{2|h|}{|\zeta|}$.

Next, the right-hand side of (\ref{detxe}) leads to
\begin{equation}\label{XEf}
|\tfrac{\det (X+D)}{\det (X)}-1|                                                \leq \tfrac{\|D\|(4\|X\|+2)}{|\det X|}, \qquad X\in GL_2(\mathbb{C}),\,\,D\in \mathbb{C}^{2\times 2}.
\end{equation} 
By assuming $\|D\|\leq \tfrac{|\det X|}{8\|X\|+4}$ and applying (\ref{ocenah}) to (\ref{XEf}) we obtain 
\begin{equation}\label{detex2}
\psi=\arg \big(\tfrac{\det (X+D)}{\det X}\big)\in (-\tfrac{\pi}{2},\tfrac{\pi}{2}),\qquad|\sin \psi |
\leq \tfrac{\|D\|(8\|X\|+4)}{|\det X|}. 
\end{equation}

We apply $\det$ to $cP^{*}AP=\widetilde{A}+E$, $Q^{T}BQ=\widetilde{B}+F$ and after simplifying we get
\begin{equation}\label{cP}
c^2|\det P|^2=\tfrac{\det (\widetilde{A}+E)}{\det A}=\tfrac{\det (\widetilde{A}+E)}{\det \widetilde{A}}\tfrac{\det \widetilde{A}}{\det A},\qquad  (\det Q)^2=\tfrac{\det (\widetilde{B}+F)}{\det B}.
\end{equation}
From (\ref{detex2}) for $X=\widetilde{A}$, $D=E$ and (\ref{cP}) it follows that
$c=(-1)^{k}e^{i(\frac{\Delta}{2}+\frac{\psi}{2})}$, $k\in \mathbb{Z}$, $\Delta=\arg \bigl(\frac{\det \widetilde{A}}{\det A}\bigr)$ with $\psi$ as in (\ref{detex2}) with $X=\widetilde{A}$, $D=E$.
Using well known facts $e^{i\frac{\psi}{2}}=1+2i\sin (\frac{\psi}{4})e^{i\frac{\psi}{4}}$ and 
$2|\sin \frac{\psi}{4}|\leq |\frac{\psi}{2}|\leq  |\sin \psi | 
$ for $\psi\in (-\tfrac{\pi}{2},\tfrac{\pi}{2})$, we deduce (\ref{cE}). 
Furthermore, (\ref{XEf}) for $X=\widetilde{A}$, $D=E$ and (\ref{cP}) for $X=\widetilde{B}$, $D=F$ give
\small
\begin{align}\label{detPQ}
|\det P|^{2}=\tfrac{|\det \widetilde{A}|}{|\det A|}|1+p|, \,\, |p|\leq \tfrac{\|E\|(4\|\widetilde{A}\|+2)}{|\det \widetilde{A}|},\quad
\small{(\det Q)^{2}=\tfrac{\det \widetilde{B}}{\det B}(1+q), \,\, |q|\leq \tfrac{\|E\|(4\|\widetilde{B}\|+2)}{|\det \widetilde{B}|}},
\end{align}
\normalsize
respectively. To conclude the proof we observe another simple fact. If $|s|\leq 1$ then 
\begin{equation}\label{ocenakoren}
\sqrt{1+s}=(-1)^{l}(1+s'), \qquad l\in \mathbb{Z}, \,\,\Rea (s')\geq -1  , \,\,
|s'|\leq |s|.
\end{equation}
To see this, we take $\rho=1+s'$ to be the square root of $1+s$ (thus $\rho^{2}=1+s$) with $\Rea (\rho)\geq 0$ and $\Rea (s')\geq -1$. It yields $|s'|=|\rho-1|=|\frac{\rho^{2}-1}{\rho+1}|\leq \frac{|s|}{1}$.
For $\|E\|\leq \tfrac{|\det \widetilde{A}|}{8\|\widetilde{A}\|+4}$ and
$\|F\|\leq \tfrac{|\det \widetilde{B}|}{8\|\widetilde{B}\|+4}$ ($|p|,|q|\leq \frac{1}{2}\leq 1$ in (\ref{detPQ})), 
we apply (\ref{ocenakoren}) to (\ref{detPQ}) for $s=p$ and $s=q$, respectively. It implies (\ref{PAE}) and (\ref{lemadetb}).
\end{proof}

\begin{proof}[Proof of Lemma \ref{lemapsi1}.]
For actions $\Psi$, $\Psi_1$ (see (\ref{aAB}) and (\ref{actionpsi1})), it follows that $(A',B')\in \Orb_{\Psi}(A,0)$ if and only if $B'=0$ and $A'\in \Orb_{\Psi_1}(A)$. Hence $\dim \bigl(\Orb_{\Psi_1}(A)\bigr)=\dim \bigl(\Orb_{\Psi}(A,0)\bigr)$, where dimensions of orbits of $\Psi_1$ are obtained from Lemma \ref{lemalist}.

To prove $\widetilde{A}\to A$ it is sufficient to find $c(s)\in S^{1}$,
$P(s)\in GL_2(\mathbb{C})$ such that 
\begin{equation}\label{cPepsipsi1}
c(s)(P(s))^*AP(s)-\widetilde{A}\to 0 \textrm{   as   } s \to 0.
\end{equation}
It is straightforward to see that 
$
P(s)=1\oplus s
$ 
with $c(s)=1$ in (\ref{cPepsipsi1}) yields $
1\oplus 0
\to
1\oplus \lambda
$, while
to prove  
$
1\oplus 0
\to
\begin{bsmallmatrix}
0 & 1 \\
\tau & 0
\end{bsmallmatrix}
$,
$0\leq \tau\leq 1$, and 
$
1\oplus -1
\to
\begin{bsmallmatrix}
0 & 1 \\
1 & i
\end{bsmallmatrix}
$,
we take 
$
P(s)=\frac{1}{\sqrt{1+\tau}}\begin{bsmallmatrix}
1 & 0 \\
1 & s
\end{bsmallmatrix}
$ 
and
$
P(s)=\frac{1}{2}\begin{bsmallmatrix}
\frac{1}{s} & \frac{1}{s} \\
s & -s
\end{bsmallmatrix}
$, respectively, and in both cases again with $c(s)=1$. 
(Compositions of these paths represent paths as well.) 
It is then left to find necessary (sufficient) conditions for the existence of these paths, i.e. given $A$, $\widetilde{A}$, $E$ satisfying
\begin{equation}\label{eqcPAE}
cP^*AP=\widetilde{A}+E, \qquad c\in S^{1}, P\in GL_2(\mathbb{C}),
\end{equation}
we must find out how $c$, $P$ depend on $E$ (how $E$ depends on $c$, $P$).

On the other hand, if (\ref{eqcPAE}) fails for every sufficiently small $E$, it gives $\widetilde{A}\not \to A$. To prove (\ref{lemapsi1a}), 
upper estimates for $\|E\|$ will be provided in such cases. 
This has allready been done for $\widetilde{A}\neq 0$, $A=0$ and $\det \widetilde{A}\neq 0$, $\det A=0$ (see Lemma \ref{pathlema} (\ref{pathlema1}) (\ref{pathlema1b})).

Throughout the rest of the proof we denote  
\begin{equation}\label{wAE}
\widetilde{A}=
\begin{bmatrix}
\alpha & \beta\\
\gamma & \omega
\end{bmatrix}
,\qquad
E=
\begin{bmatrix}
\epsilon_1 & \epsilon_2\\
\epsilon_3 & \epsilon_4
\end{bmatrix}
,\qquad
P=
\begin{bmatrix}
x & y\\
u & v
\end{bmatrix},
\end{equation}
and splitt our consideration of the remaining paths $\widetilde{A}\to A$
into several cases.

\begin{enumerate}[label={\bf Case \Roman*.},ref={Case \Roman*.},   ,wide=0pt,itemsep=10pt]
\item \label{p1i2}
$
A=\begin{bsmallmatrix}
0 & 1\\
1 & i
\end{bsmallmatrix}
$\\
It is straightforward to compute
\begin{align*}
cP^*AP=
c
\begin{bmatrix}
\overline{x} & \overline{u}\\
\overline{y} & \overline{v}
\end{bmatrix}
\begin{bmatrix}
0 & 1\\
1 & i
\end{bmatrix}
\begin{bmatrix}
x & y\\
u & v
\end{bmatrix}
&=
c
\begin{bmatrix}
2\Rea(\overline{x}u)+i|u|^2 & \overline{x}v+\overline{u}y+i\overline{u}v\\
\overline{ (\overline{x}v+\overline{u}y)}+i\overline{\overline{u}v} & 2\Rea(\overline{y}v)+i|v|^2
\end{bmatrix}.
\end{align*}
Multipling the equation (\ref{eqcPAE}) by $c^{-1}$ and then writing them componentwise yields 
\begin{align}\label{eq01i}
&2\Rea(\overline{x}u)+i|u|^2=c^{-1}(\alpha+\epsilon_1), \quad \overline{x}v+\overline{u}y+i\overline{u}v=c^{-1}(\beta+\epsilon_2),\\
&\overline{ (\overline{x}v+\overline{u}y)}+i\overline{\overline{u}v}=c^{-1}(\gamma+\epsilon_3),\quad 2\Rea(\overline{y}v)+i|v|^2=c^{-1}(\omega+\epsilon_4).\nonumber
\end{align}
The real and the imaginary parts of the first and the last equation of (\ref{eq01i}) give:
\begin{align}\label{eq01i2}
&2\Rea(\overline{x}u)=\Rea (c^{-1}\alpha)+\Rea(c^{-1}\epsilon_1), \quad |u|^2=\Ima (c^{-1}\alpha)+\Ima (c^{-1}\epsilon_1),\\
&2\Rea(\overline{y}v)=\Rea(c^{-1}\omega)+\Rea(c^{-1}\epsilon_4),\quad |v|^2=\Ima(c^{-1}\omega)+\Ima (c^{-1}\epsilon_4),\nonumber
\end{align}
while by adding (subtracting) the second and the complex-conjugated third equation of (\ref{eq01i}) for $\beta,\gamma\in \mathbb{R}$ we deduce
\begin{align}\label{011iuv}
&2i\overline{u}v=(c^{-1}\beta-\overline{c}^{-1}\gamma)+(c^{-1}\epsilon_2-\overline{c}^{-1}\epsilon_3), \\
&2(\overline{x}v+\overline{u}y)= (c^{-1}\beta+\overline{c}^{-1}\gamma)+(c^{-1}\epsilon_2+\overline{c}^{-1}\epsilon_3).\nonumber
\end{align}

\begin{enumerate}[label={(\alph*)},wide=0pt,itemindent=2em,itemsep=6pt]

\item
$\widetilde{A}=\begin{bsmallmatrix}
0 & 1\\
\tau & 0
\end{bsmallmatrix}$, $0\leq \tau <1$\\
For $\alpha=\omega=0$ the second and the last equation of (\ref{eq01i2}) give $|u|^{2}\leq \|E\|$ and $|v|^{2}\leq \|E\|$ (hence $|uv|\leq \|E\|$). Further, the first equation of (\ref{011iuv}) for $\beta=1$, $\gamma=\tau$ yields
\[
2|\overline{u}v|\geq |c^{-1}|-|\overline{c}^{-1}\tau|-|c^{-1}\epsilon_2|-|\overline{c}^{-1}\epsilon_3|\geq 1-\tau-2\|E\|.
\]
If $\|E\| < \frac{1-\tau}{2}$ we get a contradiction (remember $1>\tau\geq 0$).

\item $\widetilde{A}=\alpha\oplus 0$,\quad $\alpha\in \{0,1\}$\\
The first and the last equation of (\ref{eq01i}) for $\omega=0$ imply 
\[
\bigl|2\Rea(\overline{x}u)+i|u|^{2}-c^{-1}\alpha\bigl|\leq \|E\|,\qquad |v|^{2}\le \|E\|,\quad 2\bigl|\Rea(\overline{y}v)\bigr|\leq \|E\|,
\]
respectively. Next, from (\ref{011iuv}) for $\beta=\gamma=0$ we obtain
\begin{equation}\label{est01iuv}
|\overline{x}v+\overline{u}y|\leq \|E\|, \qquad |\overline{u}v|\leq \|E\|.
\end{equation}
Thus the first part of (\ref{lemapsi11}) for (C\ref{r10}) follows; the converse is immediate by (\ref{eq01i}).

\item $\det\widetilde{A}=e^{i\vartheta}$,\,\, $0\leq \vartheta\leq \pi$, \quad $\|\widetilde{A}\|=1$\\
By (\ref{cE}) in Lemma \ref{lemadet} we have $c^{-1}=i(-1)^k e^{-i\frac{\vartheta}{2}}+\overline{g}$ with $k\in\mathbb{Z}$, $|g|\leq 12\|E\|$, provided that $\|E\|\leq \frac{1}{12}$. Using this and rearranging the terms in  (\ref{eq01i2}), (\ref{011iuv}) we deduce:
\small
\begin{align}\label{eq01iAA}
&\overline{x}v+\overline{u}y-\tfrac{i(-1)^k}{2}( e^{-i\frac{\vartheta}{2}}\beta- e^{i\frac{\vartheta}{2}}\gamma)=\tfrac{1}{2}(\overline{g}\beta+g\gamma+c^{-1}\epsilon_2+\overline{c}^{-1}\epsilon_3)\\
2\Rea(\overline{x}u)-  &  (-1)^k\Rea (ie^{-i\frac{\vartheta}{2}}\alpha)=\Rea (c^{-1}\epsilon_1+\alpha \overline{g}), \quad 
|u|^2=(-1)^k\Ima (ie^{-i\frac{\vartheta}{2}}\alpha)+\Ima (c^{-1}\epsilon_1+\alpha \overline{g}),\nonumber\\
2\Rea(\overline{y}v)-  & (-1)^k\Rea (ie^{-i\frac{\vartheta}{2}}\omega)=\Rea (c^{-1}\epsilon_4+\omega \overline{g}),\quad 
|v|^2=(-1)^k\Ima (ie^{-i\frac{\vartheta}{2}}\omega)+\Ima (c^{-1}\epsilon_4+\omega \overline{g})\nonumber.
\end{align}
\normalsize
Observe that for $\beta=\gamma$ the first equation in (\ref{eq01iAA}) yields
\begin{equation}\label{uvE1}
\bigl|(\overline{x}v+\overline{u}y)-(-1)^{k}\beta\sin (\tfrac{\vartheta}{2})\bigr|\leq 12|\beta|\,|E\|+\|E\|,
\end{equation}
while for $\omega=0$ the last equation of (\ref{eq01i}) implies that $2\big|\Rea (\overline{y}u)\big|\leq \|E \|$, $|v|^{2}\leq \|E\|$.

\begin{enumerate}[label={(\roman*)},wide=0pt,itemindent=4em,itemsep=3pt]

\item $\widetilde{A}=1\oplus e^{i\vartheta}$,\,\,  $0\leq \vartheta< \pi$ \\
After multiplying the third and the fifth equation of (\ref{eq01iAA}) for $\alpha=1$, $\omega=e^{i\widetilde{\theta}}$, $\beta=\gamma=0$, estimating the imaginary parts by its moduli, applying the triangle inequality and using $|g|\leq 12\|E\|$, we obtain
\begin{align*}
|uv|^2\geq &\bigl|\cos (\tfrac{-\vartheta}{2})\cos (\tfrac{\vartheta}{2})\bigr|-|c^{-1}\epsilon_4+\omega \overline{g}|-|c^{-1}\epsilon_1+\alpha \overline{g}|-\bigl|(c^{-1}\epsilon_1+\alpha \overline{g})(c^{-1}\epsilon_4+\omega \overline{g})\bigr|\\
\geq & \tfrac{1}{2}|1+\cos \vartheta|-26\|E\|-169\|E\|^2.
\end{align*}
Combining it with (\ref{est01iuv}) and $\|E\|\geq \|E\|^{2}$ leads to a contradiction for $\|E\|<\tfrac{|1+\cos \vartheta|}{392}$.

\item
$\widetilde{A}$ is either equal to $1\oplus -1$, $\begin{bsmallmatrix}0 & 1\\ 1 &0    \end{bsmallmatrix}$ or $\begin{bsmallmatrix}0 & 1\\ 1 &i    \end{bsmallmatrix}$ \quad ($\vartheta=\pi$)\\
The statement (\ref{lemapsi11}) for (C\ref{r5}) follows immediately from (\ref{eq01iAA}), (\ref{uvE1}) for $\vartheta=\pi$ and either $\alpha=1$, $\omega=-1$, $\beta=0$ or  $\alpha=\omega=0$, $\beta=1$ or $\alpha=0$, $\beta=1$, $\omega=i$. 
\end{enumerate}

\end{enumerate}


\item \label{p100l}
$
A=1\oplus \lambda
$, \quad $|\lambda|\in \{1,0\}$\\
We have
\begin{align*}
cP^*AP=
c
\begin{bmatrix}
\overline{x} & \overline{u}\\
\overline{y} & \overline{v}
\end{bmatrix}
\begin{bmatrix}
1 & 0\\
0 & \lambda
\end{bmatrix}
\begin{bmatrix}
x & y\\
u & v
\end{bmatrix}
& =
c
\begin{bmatrix}
|x|^2+\lambda|u|^2 & \overline{x}y+\lambda \overline{u}v\\
\overline{y}x+\lambda\overline{v}u & |y|^2+\lambda|v|^2
\end{bmatrix}.
\end{align*}
Thus (\ref{eqcPAE}) multiplied by $c^{-1}$ and written componentwise (also rearranged) yields: 
\begin{align}\label{eq1l}
& |x|^2+\lambda|u|^2 -c^{-1}\alpha=c^{-1}\epsilon_1, \quad && \overline{x}y+\lambda \overline{u}v-c^{-1}\beta=c^{-1}\epsilon_2,\\
&\overline{y}x+\lambda\overline{v}u-c^{-1}\gamma=c^{-1}\epsilon_3,\quad && |y|^2+\lambda|v|^2-c^{-1}\omega=c^{-1}\epsilon_4.\nonumber
\end{align}
Subtracting the second complex-conjugated equation (and multiplied by $\lambda$) from the third equation (and multiplied by $\overline{\lambda}$) for $\beta, \gamma\in \mathbb{R}$ further gives
\begin{align}\label{eq1l3}
&2\Ima (\lambda)\overline{v}u-c^{-1}\gamma+\overline{c}^{-1}\beta= c^{-1}\epsilon_3- \overline{c}^{-1}\overline{\epsilon}_2,\\
&-2\Ima (\lambda)\overline{y}x-c^{-1}\overline{\lambda}\gamma+\overline{c}^{-1}\lambda\beta=c^{-1}\overline{\lambda}\epsilon_3-\overline{c}^{-1}\lambda\overline{\epsilon}_2\nonumber .
\end{align}

\begin{enumerate}[label=(\alph*),wide=0pt,itemindent=2em,itemsep=6pt]

\item 
$\lambda=e^{i\theta}$, \quad $0\leq \theta \leq \pi$

By taking the imaginary and the real parts ob the first and the last equation of (\ref{eq1l}) for $\lambda=e^{i\theta}$ we obtain 
\begin{align}\label{eqa0}
&(\sin \theta) |u|^2=\Ima (c^{-1}\alpha+c^{-1}\epsilon_1), \quad |x|^2+(\cos \theta)|u|^2=\Rea (c^{-1}\alpha+c^{-1}\epsilon_1),\\
&(\sin \theta) |v|^2=\Ima (c^{-1}\omega+c^{-1}\epsilon_1), \quad |y|^2+(\cos \theta)|v|^2=\Rea (c^{-1}\omega+c^{-1}\epsilon_1).\nonumber
\end{align}

\begin{enumerate}[label=(\roman*),wide=0pt,itemindent=4em,itemsep=3pt]

\item
$
\widetilde{A}=
\begin{bsmallmatrix}
0 & 1\\
\tau & 0
\end{bsmallmatrix}
$,
$0\leq \tau\leq 1$\\
If $\theta=0$ then (\ref{eqa0}) for $\theta=\alpha=\omega=0$ implies that $|x|^{2},|u|^{2},|y|^{2},|v|^{2}\leq \|E\|$, which contradicts the second equation of (\ref{eq1l}) for $\beta=\lambda=1$ ($\theta=0$) with $\|E\|<\frac{1}{3}$.
Next, let $\theta=\pi$ ($\lambda=-1$). From (\ref{eq1l3}) for $\Ima(\lambda)=0$, $\beta=1$, $\gamma=\tau$ it then follows $|1-\tau|\leq 2\|E\|$, which fails if $\tau\neq 1$, $\|E\|<\frac{1-\tau}{2}$. Further, when $\tau=1$ we deduce from (\ref{cE}) (Lemma \ref{lemadet}) that $c^{-1}=(-1)^{k}+\overline{g}$, $|\overline{g}|\leq 12\|E\|$, so the second equation of (\ref{eq1l}) for $\lambda=-1$, $\beta=1$ yields $\big|\overline{x}y-\overline{u}v-(-1)^{k}\big|\leq 13\|E\|$.
By combining it with the first and the last equation of (\ref{eq1l}) for $\lambda=-1$, $\alpha=\omega=0$ we get (\ref{lemapsi11}) with (C\ref{r12}).

\quad
It is left to consider $0< \theta <\pi$. 
From (\ref{eqa0}) for $\alpha=\omega=0$ it follows that
\begin{equation}\label{eqo0}
|u|^2,|v|^2\leq \tfrac{\|E\|}{\sin \theta}, \qquad 
|x|^2,|y|^2\leq \|E\|(1+\cot \theta),\qquad 0< \theta<\pi.
\end{equation}
Applying the triangle inequality to the second equation of (\ref{eq1l}), and using the estimates (\ref{eqo0}) leads to an inequality which fails for $\|E\|<(2+\cot \theta+\tfrac{1}{\sin \theta})^{-1}$:
\[
1-\|E\|\leq |\overline{x}y+\lambda \overline{u}v|\leq 
\|E\|(1+\cot \theta)+\tfrac{1}{\sin \theta }\|E\|.
\]

\item
$
\widetilde{A}=
\begin{bsmallmatrix}
0 & 1\\
1 & i
\end{bsmallmatrix}
$\\
By (\ref{cE}) in Lemma \ref{lemadet} we have $c^{-1}=-i(-1)^{k}e^{i\frac{\theta}{2}}+\overline{g}$, $|g|\leq 12\|E\|$, provided that $\|E\|\leq \frac{1}{12}$.
%
The third equation of (\ref{eqa0}) for $\theta=\pi$, $\omega=i$ yields $0=(-1)^k+\Ima (i\overline{g}+c^{-1}\epsilon_4)$, which failes for $\|E\|<\frac{1}{13}$. 
If $\theta=0$, then the second and the last equation of (\ref{eqa0}) for $\alpha=0$, $\omega=i$ imply $|x|^{2},|u|^{2}\leq \|E\|$ and $|y|^{2},|v|^{2}\leq 1+\|E\|$, respectively. From the second equation of (\ref{eq1l}) for $\beta=1$, $\lambda=e^{i \theta}$ we then conclude $1-\|E\|\leq 2\sqrt{\|E\|(1+\|E\|)}$, so we have a contradiction for any $\|E\|$ small enough.

\quad
Finally, for $0  <\theta<\pi$ we use (\ref{eqa0}) for $\alpha=0$, $\omega=i$, $c^{-1}=-i(-1)^{k}e^{i\frac{\theta}{2}}+\overline{g}$ to get $|u|^2\leq \tfrac{\|E\|}{\sin \theta}$, $|x|^2\leq \|E\|(1+\cot \theta)$ (see (\ref{eqo0})) and
\begin{align*}
&|\sin\theta|\, |v|^{2}=|(-1)^k\sin (\tfrac{\theta}{2})+\Ima (i\overline{g}+c^{-1}\epsilon_4)|\leq |\sin \tfrac{\theta}{2}|+13\|E\|\leq 1+13\|E\|,\\
&|y|^{2}=\big|(-1)^k \cos(\tfrac{\theta}{2})+\Rea(i\overline{g}+c^{-1}\epsilon_4)-\cos \theta|v|^2\big|\leq (1+13\|E\|)+\cot \theta\big (1+13\|E\| \big).
\end{align*}
Applying these estimates to the second equation of (\ref{eq1l}) for $\beta=1$, $\lambda=e^{i \theta}$ we get an inequality, which fails for every sufficiently small $\|E\|$:
\[
1-\|E\|\leq |\overline{x}y+e^{i\theta }\overline{u}v|\leq 
\big( (\cot \theta+1)+\tfrac{1}{|\sin \theta|}  \big)\sqrt{\|E\|(1+13\|E\|)}.
\]

\item $\widetilde{A}=\alpha\oplus 0$, $\alpha\in \{0,1\}$\\
If $0\leq \theta <\pi$ ($\theta=\pi$) the equations (\ref{eq1l}) for $\omega=\beta=\gamma=0$, $\lambda=e^{i\theta}$ give the first part of (the complete statement) (\ref{lemapsi11}) with (C\ref{r3}) (with (C\ref{r9}) for $\omega=0$, $\sigma=-1$). To see the converse for (C\ref{r3}), we fix $s>0$ and assume $|y|^{2},|v|^{2}\leq s$ and $\big||x|^{2}+e^{i\theta}|u|^{2} -c^{-1}\alpha \big|\leq s$. By observing the imaginary and the real part of $|x|^{2}+e^{i\theta}|u|^{2} -c^{-1}\alpha$ for $\theta=0$ we deduce that $|x|^{2},|u|^{2}\leq |\alpha|+s$ and
\[
|\sin \theta|\,|u|^{2} \leq |\alpha|+s, \quad
|x|^{2}\leq |\alpha|+|\cot \theta| \big(|\alpha|+s\big)+s=\big(|\alpha|+|\cot \theta|\big)\big(|\alpha|+s\big), \,\,0< \theta<\pi.
\]
The second (third) equation of (\ref{eq1l}) for $\beta=0$ ($\gamma=0$), $\lambda=e^{i\theta}$ then yields $|\epsilon_2|,|\epsilon_3|\leq \sqrt{(|\alpha|+|\cot \theta|)(|\alpha|+s)s}$, so (\ref{lemapsi11}) with (C\ref{r3}) is proved. The converse for (C\ref{r9}) is trivial.

\item
$
\widetilde{A}=
1\oplus \widetilde{\theta}
$,\quad
$0\leq \widetilde{\theta} \leq \pi$.\\
By (\ref{cE}) (Lemma \ref{lemadet}) we have 
$c^{-1}=(-1)^k e^{i\frac{\theta-\widetilde{\theta}}{2}}+\overline{g}$, $|\overline{g}|\leq 12\|E\|$, assuming that $\|E\|\leq \frac{1}{12}$. 
Thus the first and the last equation of (\ref{eq1l}) for $\alpha=1$, $\omega=e^{i\widetilde{\theta}}$ are of the form:
\begin{align}\label{eq1l5}
&|x|^2+e^{i\theta}|u|^2 =(-1)^k e^{i\frac{\theta-\widetilde{\theta}}{2}}+(\overline{g}+c^{-1}\epsilon_1),\\
&|y|^2+e^{i\theta}|v|^2=(-1)^k e^{i\frac{\widetilde{\theta}+\theta}{2}}+(\overline{g}e^{i\widetilde{\theta}}+c^{-1}\epsilon_4).\nonumber
\end{align}
We now take the 
imaginary parts of equations (\ref{eq1l5}), slightly rearrange the terms and use the triangle inequality:
\begin{align}\label{eq1lim}
& \bigl| |u|^2\sin \theta-(-1)^k\sin (\tfrac{\theta-\widetilde{\theta}}{2})\bigr| \leq  \bigl| \Ima (\overline{g})+\Ima (\epsilon_1)\bigr|\leq 13\|E\|,\\
& \bigl| |v|^2\sin \theta-(-1)^k\sin (\tfrac{\widetilde{\theta}+\theta}{2})\bigr|\leq \bigl| \Ima (\overline{g}e^{i\widetilde{\theta}}+\overline{c}^{-1}\epsilon_4)\bigr|\leq 13\|E\|.\nonumber
\end{align}
%
%
In particular we have
\begin{equation*}
|u|^2\sin \theta \geq |\sin (\tfrac{\widetilde{\theta}-\theta}{2})|-13\|E\|, \quad 
|v|^2\sin \theta \geq |\sin (\tfrac{\widetilde{\theta}+\theta}{2})|-13\|E\|.
\end{equation*}
Multiplying these inequalities and using the triangle inequality we deduce that 
\begin{align}\label{eq1l4}
\sin ^2\theta|uv|^2\geq & 
 \big|\sin (\tfrac{\widetilde{\theta}-\theta}{2})\sin (\tfrac{\widetilde{\theta}+\theta}{2})\big|-13\|E\|\bigl(\big|\sin (\tfrac{\widetilde{\theta}-\theta}{2})\big|+\big|\sin (\tfrac{\widetilde{\theta}+\theta}{2})\big|  \bigr)-169\|E\|^2 .				   
\end{align}
From (\ref{eq1l3}) for $\beta=\gamma=0$, $\Ima (\lambda)=\sin \theta$ we get
$\big|(\sin \theta)\overline{v}u\big|\leq  \|E\|$.
Combining this with (\ref{eq1l4}) we obtain that 
\begin{equation}\label{eqtnwt}
 \|E\|^{2}\geq \big|\sin (\tfrac{\widetilde{\theta}-\theta}{2})\sin (\tfrac{\widetilde{\theta}+\theta}{2})\big|-26\|E\|\sin \widetilde{\theta}\cos \theta -169\|E\|^2.
\end{equation}

\quad
If $\widetilde{\theta}=\theta\in \{0,\pi\}$ then equations (\ref{eq1l5}) and the second equation of (\ref{eq1l}) for $\lambda=e^{i\theta}$ allready give us the statement (\ref{lemapsi11}) for (C\ref{r9}) in case $\omega=\sigma\in \{1,-1\}$. (Note that if $\theta=\widetilde{\theta}=0$, then $k$ is even.) 
Further, when $\widetilde{\theta}\neq \theta\in \{0,\pi\}$ the first equation of (\ref{eq1lim}) fails for $\|E\|\leq \frac{1}{13} \big|\sin (\tfrac{\theta-\widetilde{\theta}}{2})\big|$.

\quad 
Next, let $0<\theta<\pi$. For $\theta\neq\widetilde{\theta}$ we have $\frac{1}{2}(\theta+\widetilde{\theta}),\frac{1}{2}(\theta-\widetilde{\theta})\neq l\pi$, $l\in \mathbb{Z}$, hence it is easy to choose $\|E\|$ so small that(\ref{eqtnwt}) fails. 
If $\theta=\widetilde{\theta}$, then the first equation of (\ref{eq1lim}) leads to $|u|^{2}\leq \frac{\|E\|}{\sin \theta}$, and by comparing the real parts of the first equation of (\ref{eq1l5}), slightly rearranging the terms, we furter get:
\begin{align*}
\bigl||x|^2 -(-1)^{k}\bigr|  = \bigl|\Rea (\overline{g}+\overline{c}^{-1}\epsilon_1)-\cos \theta |u|^2\bigr|  \leq 
13(\cot \theta +1)\|E\|.
\end{align*}
The second equation of (\ref{eq1lim}) also yields $\big||v|^{2}-(-1)^{k} \big|\leq \frac{13}{\sin \theta}\|E\|$, thus $\|E\|<\frac{\sin \theta}{13}$ implies that $k$ is even. This concludes the proof of (\ref{lemapsi11}) about (C\ref{r1}). 

\end{enumerate}

\item $\lambda=0$\ \quad (It suffices to consider the case when $\det \widetilde{A}=0$.)\\
When 
$\widetilde{A}=\alpha\oplus 0$, $\alpha\in\{0,1\}$ 
the statement (\ref{lemapsi11}) with (C\ref{r11}) follows immediately from the first and the third equation of (\ref{eq1l}) for $\omega=\lambda=0$. Applying (\ref{ocenah}) for $\|E\|\leq \frac{1}{2}$ to the first equation of (\ref{eq1l}) for $\alpha=1$, $\lambda=0$ (multiplied by $c$), yields $\psi=\arg(c)\in (-\frac{\pi}{2},\frac{\pi}{2})$, $|\sin \psi|\leq 2\|E\|$. Therefore $c-1=e^{i\psi}-1=2i\sin (\frac{\psi}{2})e^{i\frac{\psi}{2}}$ with $|\sin (\frac{\psi}{2})|\leq|\sin \psi|\leq 2\|E\|$. 

\quad
If $\widetilde{A}=\begin{bsmallmatrix}
0 & 1\\
0 & 0
\end{bsmallmatrix}$, the first and the last equation of (\ref{eq1l}) for $\lambda=\alpha=\omega=0$ yield $|x|^{2},|y|^{2}\leq \|E\|$, thus the third equation of (\ref{eq1l}) for $\lambda=0$, $\gamma=0$
fails for $\|E\|<\frac{1}{2}$.

\end{enumerate}

\item \label{p01t0}
$
A=\begin{bsmallmatrix}
0 & 1\\
\tau & 0
\end{bsmallmatrix}
$, \quad $0\leq \tau \leq 1$\\
We calculate
\begin{align*}
cP^*AP=
c
\begin{bmatrix}
\overline{x} & \overline{u}\\
\overline{y} & \overline{v}
\end{bmatrix}
\begin{bmatrix}
0 & 1\\
\tau & 0
\end{bmatrix}
\begin{bmatrix}
x & y\\
u & v
\end{bmatrix}
& =
c
\begin{bmatrix}
\overline{x}u+\tau\overline{u}x & \overline{x}v+\tau \overline{u}y\\
\tau \overline{v}x+\overline{y}u & \overline{y}v+\tau\overline{v}y
\end{bmatrix}, \qquad 0\leq \tau\leq 1.
\end{align*}
Thus (\ref{eqcPAE}) multiplied by $c^{-1}$ and rearranged is equivalent to 
\begin{align}\label{equ01t}
\overline{x}u+\tau\overline{u}x -c^{-1}\alpha=c^{-1}\epsilon_1, \quad \overline{x}v+\tau \overline{u}y-c^{-1}\beta=c^{-1}\epsilon_2,\\
\tau \overline{v}x+\overline{y}u-c^{-1}\gamma=c^{-1}\epsilon_3,\quad \overline{y}v+\tau\overline{v}y-c^{-1}\omega=c^{-1}\epsilon_4.\nonumber
\end{align}
%
Rearranging the terms of the first and the last equation immediately yields 
\begin{align}\label{eq01t2}
(1+\tau)\Rea(\overline{x}u)+i(1-\tau)\Ima(\overline{x}u)=c^{-1}\alpha+c^{-1}\epsilon_1,\\
(1+\tau)\Rea(\overline{y}v)+i(1-\tau)\Ima(\overline{y}v)=c^{-1}\omega+c^{-1}\epsilon_4,\nonumber
\end{align}
while multiplying the third (second) complex-conjugated equation with $\tau$, subtracting it from the second (third) equation, and rearranging the terms, give
\begin{align}\label{eq011t3}
(1-\tau^2)\overline{x}v=&c^{-1}(\beta+\epsilon_2)-\tau\overline{c}^{-1}(\overline{\gamma}+\overline{\epsilon_3})= 
(c^{-1}\beta-\tau\overline{c}^{-1}\overline{\gamma})+(c^{-1}\epsilon_2-\tau\overline{c}^{-1}\overline{\epsilon_3})\\
(1-\tau^{2})\overline{y}u= &c^{-1}(\gamma+\epsilon_3)-\tau \overline{c}^{-1}(\overline{\beta}+\overline{\epsilon_2})=
(c^{-1}\gamma-\tau \overline{c}^{-1}\overline{\beta})+(c^{-1}\epsilon_3-\tau \overline{c}^{-1}\overline{\epsilon_2}).\nonumber
\end{align}

\begin{enumerate}[label=(\alph*),wide=0pt,itemindent=2em,itemsep=3pt]
\item $\tau=1$\\
Since 
$A=
\begin{bsmallmatrix}
0 & 1\\
1 & 0
\end{bsmallmatrix}$
is $*$-congruent to $1\oplus -1$, the existence of paths in this case was allready analysed in \ref{p100l} It is only left to check (\ref{lemapsi11}) with (C\ref{r7}). If either
$\widetilde{A}=
\begin{bsmallmatrix}
0 & 1\\
1 & 0
\end{bsmallmatrix}
$ or $\widetilde{A}=1\oplus-1$, then by (\ref{cE}) we have $c^{-1}=(-1)^{k}+\overline{g}$, $k\in \{0,1\}$, $|g|\leq 12\|E\|$, provided that $\|E\|\leq \frac{1}{12}$. 
The second (third) equation of (\ref{equ01t}) and (\ref{eq01t2}) for $\tau=1$ then imply:
\begin{align}\label{equal10}
&|\overline{x}v+\overline{u}y-(-1)^{k}\beta|\leq \|E\|+12|\beta|\,\|E\|, \quad \beta\in \{0,1\}\\
&\bigl|2\Rea(\overline{x}u)-(-1)^{k}\alpha\bigr| \leq \|E\|+12|\alpha|\,\|E\|, \; \bigl|2\Rea(\overline{y}v)-(-1)^{k}\omega\bigr| \leq \|E\|+12|\omega|\,\|E\|.\nonumber
\end{align}

\quad
The inequalities (\ref{equal10}) are valid also if we consider (\ref{equ01t}) (and (\ref{eq01t2})) for $\alpha\in \{0,1\}$, $\beta=\gamma=\omega=0$ ($\widetilde{A}=\alpha\oplus 0$).
Note that for $\alpha=1$ the first equation of (\ref{eq01t2}) for $\tau=1$, and multiplied by $c$, is of the form 
$2c\Rea(\overline{x}u)=1+\epsilon_1$. Therefore, by applying (\ref{ocenah}) for $\|E\|\leq \frac{1}{2}$ we get that $c=(-1)^{k}e^{i\psi}$, $k\in \mathbb{Z}$, $\psi\in (-\frac{\pi}{2},\frac{\pi}{2})$, $|\sin \psi|\leq 2\epsilon$. Moreover, $c-(-1)^{k}=(-1)^{k}2i\sin (\frac{\psi}{2})e^{i\frac{\psi}{2}}$ with $|\sin (\frac{\psi}{2})|\leq|\sin \psi|\leq 2\|E\|$. 

\quad
Conversely, we assume that the expressions (C\ref{r7}) for $c=(-1)^{k}$ (precisely the left-hand sides in (\ref{equ01t})) are bounded from above by some $s>0$. Thus the right-hand sides of (\ref{equ01t}) (and hence $\|E\|$) are bounded from above by $s$ as well.

\item $0\leq\tau<1$\\
From (\ref{eq01t2}) we obtain that
\begin{equation}\label{xyuvt2}
(1+\tau)|\overline{x}u|\geq |\alpha+\epsilon_1|\geq (1-\tau)|\overline{x}u|,\qquad (1+\tau)|\overline{y}v|\geq |\omega+\epsilon_4|\geq (1-\tau)|\overline{y}v|.
\end{equation}
By multiplying the left-hand and the right-hand sides of these inequalities we get
\begin{align}\label{xyuvt22}
(1+\tau)^{2}|\overline{x}u\overline{y}v|\geq |\alpha\omega|-\bigl(|\alpha|+|\omega|\bigr)\|E\|-\|E\|^2,\\
%
\label{xyuvt}
|\alpha\omega|+\bigl(|\alpha|+|\omega|\bigr)\|E\|+\|E\|^2 \geq (1-\tau)^{2}|\overline{x}u\overline{y}v|.
\end{align}

\begin{enumerate}[label=(\roman*),wide=0pt,itemindent=2em,itemsep=3pt]

\item $\widetilde{A}=\begin{bsmallmatrix}
0 & 1\\
\gamma & \omega
\end{bsmallmatrix}$, $0 \leq \gamma\leq 1$, $\gamma\neq \tau$; and either $\omega=0$ or $\gamma=1$, $\omega=i$\\
Equations (\ref{eq011t3}) for $\beta=1$, $0\leq \gamma \leq 1$ imply 
\begin{align*}
&
(1-\tau^2)|\overline{x}v|\geq |\tau \gamma-1|-(\tau+1)\|E\|,\\
&
(1-\tau^{2})|\overline{u}y|\geq |\gamma-\tau|-(1+\tau)\|E\|.\nonumber
\end{align*}
By combining these inequalities and making some trivial estimates we obtain
\[
(1-\tau^{2})^2|\overline{y}u\overline{x}v|\geq |\tau\gamma-1|\,|\gamma-\tau|-(1+\tau)\bigl(\tau\gamma+1+\gamma+\tau\bigr)\|E\|-(1+\tau)^2\|E\|^2.
\]
Together with (\ref{xyuvt}) for $\alpha=0$ and using $\|E\|\geq \|E\|^{2}$ we get 
\[
(1+\tau)^2(1+|\omega|)\|E\|\geq |\tau\gamma-1|\,|\gamma-\tau|-(1+\tau)^{2}(\gamma+1)\|E\|-(1+\tau)^2\|E\|,
\]
which fails for $\|E\|<\frac{|\tau\gamma-1|\,|\gamma-\tau|}{(1+\tau)^2(\gamma+|\omega|+3)}$; remember $\gamma\neq \tau$, $0\leq\gamma\leq 1$, $0\leq\tau< 1$.

\item
$
\widetilde{A}=\begin{bsmallmatrix}
0 & 1\\
\tau & 0
\end{bsmallmatrix}
$\\
By (\ref{cE}) in Lemma \ref{lemadet} for $\tau\neq 0$, $\|E\|\leq \frac{\tau}{12}\leq \frac{1}{12}$ we have $c^{-1}=(-1)^{k}+\overline{g}$, $k\in\mathbb{Z}$, $|g|\leq \frac{12}{\tau}\|E\|$, thus (\ref{eq011t3}) for $\beta=1$, $\gamma=\tau$ yields
\begin{align*}
&(1-\tau^2)\overline{x}v=\big((-1)^{k}(1-\tau^2)-g\tau^{2}+\overline{g}\big)+(c^{-1}\epsilon_2-\tau\overline{c}^{-1}\overline{\epsilon_3})\\
&(1-\tau^{2})\overline{y}u=-2\tau\Ima(\overline{g})+(c^{-1}\epsilon_3-\tau \overline{c}^{-1}\overline{\epsilon_2}).\nonumber
\end{align*}
It further implies  
\[
(1-\tau^{2})|\overline{y}u|\leq 12\tau\|E\|+(1+\tau)\|E\|, \quad (1-\tau^2)\big|\overline{x}v-(-1)^{k}\big|\leq \tfrac{12(\tau+1)}{\tau}\|E\|+(\tau+1)\|E\|.
\]
Moreover, from (\ref{eq011t3}) for $\tau=\gamma=0$, $\beta=1$ we deduce $|\overline{y}u|\leq \|E\|$, $|\overline{x}v-c^{-1}|\leq \|E\|$, and (\ref{xyuvt2}) for $\alpha=\omega=0$ concludes the proof of (\ref{lemapsi11}) for  (C\ref{r6}). (The converse is apparent.)

\item $\widetilde{A}=\alpha\oplus \omega$ 

From (\ref{eq011t3}) for $\beta=\gamma=0$ it follows
\begin{equation}\label{xbarv}
(1-\tau^{2})|\overline{x}v|\leq (1-\tau)\|E\|, \quad (1-\tau^{2})|\overline{u}y|\leq (1-\tau)\|E\|,\quad (1+\tau)^{2}|\overline{x}v\overline{u}y|\leq  \|E\|^2.
\end{equation}
%
Combining this with (\ref{xyuvt22}), rearranging the terms we get
\begin{equation}\label{estad}
|\alpha\omega|\leq 2\|E\|^2+\big(|\alpha|+|\omega|\big)\|E\|.
\end{equation}
If $\alpha,\omega\neq 0$ then by choosing $\|E\|<\frac{|\alpha\omega|}{|\alpha|+|\omega|+2}$ we contradict the above inequality. 
Furthermore, (\ref{eq01t2}), (\ref{xyuvt2}), (\ref{xbarv}) give the first part of (\ref{lemapsi11}) for (C\ref{r4}).
Conversely, assuming that the expressions of (C\ref{r4}) for $\alpha=1$ are bounded from above by $s>0$, then (\ref{equ01t}) (and \ref{eq01t2}) imply that $\|E\|\leq 3s$.
%
\end{enumerate} 

\end{enumerate}

\end{enumerate}

This completes the proof of the lemma.
\end{proof}

\section{Proof of Theorem \ref{izrek}}\label{proofT}

\begin{proof}[Proof of Theorem \ref{izrek}.]
Recall that the existence of a path $(\widetilde{A},\widetilde{B})\to (A,B)$ in the closure graph for the action (\ref{aAB}),
immediately implies (see Lemma \ref{pathlema}): 
\begin{equation}\label{pogoj}
\widetilde{A}\to A, \qquad \widetilde{B}\to B, \qquad \left|\det \widetilde{A}\det B\right|=\left|\det \widetilde{B}\det A\right|.
\end{equation}
When any of the conditions (\ref{pogoj}) is not fullfiled, then $(\widetilde{A},\widetilde{B})\not \to (A,B)$ and we allready have a lower estimate on the distance from $(\widetilde{A},\widetilde{B})$ to the orbit of $(A,B)$ (see Lemma \ref{pathlema}, Lemma \ref{lemapsi2}, Lemma \ref{lemapsi1}). Further, $(\widetilde{A},0_2)\to (A,0_2)$ if and only if $\widetilde{A}\to A$, and trivially $(A,B)\to (A,B)$ for any $A,B$.

From now on we suppose $(\widetilde{A},\widetilde{B})\neq (A,B)$, $B\neq 0$, and such that (\ref{pogoj}) is valid. Let
\begin{equation}\label{eqABEF}
cP^{*}AP=\widetilde{A}+E, \quad P^{T}BP=\widetilde{B}+F, \qquad c\in S^{1},\,P\in GL_2(\mathbb{C}), E,F\in \mathbb{C}^{2\times 2}.
\end{equation}
Due to Lemma \ref{lemapsi1}, Lemma \ref{lemadet} the first equation of (\ref{eqABEF}) yields the restrictions on $P$, $c$ imposesed by $\|E\|$. Using these we then analyse the second equation of (\ref{eqABEF}).
When it implies an inequality that fails for any sufficiently small $E$, $F$, it proves $(\widetilde{A},\widetilde{B})\not \to(A,B)$. The inequality just mentioned also provides the estimates how small $E$, $F$ should be; this calculation is very straightforward but is often omitted.

On the other hand, if given matrices $A$, $B$, $\widetilde{A}$, $\widetilde{B}$ we can choose $E$ and $F$ in (\ref{eqABEF}) to be arbitrarily small, this will yield $(\widetilde{A},\widetilde{B})\to (A,B)$.
In most cases we find $c(s)\in S^{1}$, $P(s)\in GL_2(\mathbb{C})$ such that
\begin{equation}\label{cPepsi}
c(s) \bigl(P(s)\bigr)^*AP(s)-\widetilde{A}=E(s)\stackrel{s\to 0}{\longrightarrow} 0, \quad \bigl(P(s)\bigr)^TBP(s)-\widetilde{A}=F(s)\stackrel{s\to 0}{\longrightarrow} 0. 
\end{equation}
However, to confirm the existence of a path we can also prove the existence of suitable solutions of (\ref{eqABEF}) by using the last part of Lemma \ref{lemapsi1} (\ref{lemapsi11}).

Throughout the rest the proof we denote $\delta=\nu\|E\|$ (the constant $\nu>0$ is provided by Lemma \ref{lemapsi1}), $\epsilon=\|F\|$,
\begin{equation*}
B=\begin{bmatrix}
a & b\\
b & d
\end{bmatrix},\quad 
\widetilde{B}=\begin{bmatrix}
\widetilde{a} & \widetilde{b}\\
\widetilde{b} & \widetilde{d}
\end{bmatrix},\qquad 
F=\begin{bmatrix}
\epsilon_1 & \epsilon_2\\
\epsilon_2 & \epsilon_4
\end{bmatrix},
\qquad
P=\begin{bmatrix}
x & y\\
u & v
\end{bmatrix},
\end{equation*}
where sometimes polar coordinates for $x,y,u,v$ in $P$ might be prefered:
\begin{equation}\label{polarL}
x=|x|^{i\phi},\quad y=|y|e^{i\varphi},\quad u=|u|e^{i\eta},\quad v=|v|e^{i\kappa}, \qquad \phi,\varphi,\eta,\kappa\in \mathbb{R}.
\end{equation}
The second matrix equation of (\ref{eqABEF}) can thus be written componentwise as:
\begin{align}\label{eqBF1}
&ax^2+2bux+du^2=\widetilde{a}+\epsilon_1,\nonumber \\
&axy+buy+bvx+duv=\widetilde{b}+\epsilon_2,\\ 
&ay^2+2bvy+dv^2=\widetilde{d}+\epsilon_4. \nonumber
\end{align}
For the sake of simplicity some estimates in the proof are crude, and it is allways assumed $\epsilon,\delta\leq \frac{1}{2}$. When appying Lemma \ref{lemadet} with $A,\widetilde{A}$ or $B,\widetilde{B}$ nonsingular we in addition take $\frac{\delta}{\nu}=\|E\|\leq \tfrac{|\det \widetilde{B}|}{8\|\widetilde{B}\|+4}$ or $\epsilon=\|F\|\leq \tfrac{|\det \widetilde{B}|}{8\|\widetilde{B}\|+4}$, respectively.
Furthermore, we use the notation $(\widetilde{A},\widetilde{B})\dashrightarrow (A,B)$ when the existence of a path is yet to be considered. 

We splitt our analysis
into several cases. (For normal forms recall Lemma \ref{lemalist}.)

\begin{enumerate}[label={\bf Case \Roman*.},ref={Case \Roman*},   wide=0pt,itemsep=10pt]

\item \label{p1t1t}
$
(1\oplus e^{i\theta},\widetilde{B})\dashrightarrow
(1\oplus e^{i\theta},B)
$, \quad $0<\theta<\pi$,\qquad $B\neq \widetilde{B}$

Denoting $u^{2}=\delta_2$, $y^{2}=\delta_1$ and slightly rearranging the terms in (\ref{eqBF1}) yields
\begin{align}\label{eqBF1te}
&ax^2-\widetilde{a}=\epsilon_1-2bx\sqrt{\delta_2}-d\delta_2\nonumber \\
&bvx-\widetilde{b}=\epsilon_2-(ax\sqrt{\delta_1}+b\sqrt{\delta_1\delta_2}+dv\sqrt{\delta_2})\\ 
&dv^2-\widetilde{d}=\epsilon_4-2bv\sqrt{\delta_1}-d\delta_1. \nonumber
\end{align}
From Lemma \ref{lemapsi1} (\ref{lemapsi11}) for (C\ref{r1}) we get $|\delta_1|,|\delta_2|\leq \delta$ and $\bigl||v|^{2}-1\bigr|,\bigl||x|^{2}-1\bigr|\leq \delta$ (therefore $\sqrt{1-\delta}\leq |x|,|v|\leq \sqrt{1+\delta}$, $\bigl||vx|-1\bigr|\leq \delta$).
By applying the triangle inequality we conclude from the first equation of (\ref{eqBF1te}) that
\begin{align*}
|d|\delta+2|b|\sqrt{\delta^{2}+\delta} +\epsilon \geq |ax^2-\widetilde{a}| 
 \geq 
\bigl||ax^2|-|\widetilde{a}|\bigr|
 = \bigl|(|a|-|\widetilde{a}|)+|a|(|x|^2-1)\bigr|\geq \bigl||a|-|\widetilde{a}|\bigr|-|a|\delta 
\end{align*}
and similarly the last two equations of (\ref{eqBF1te}) yield
\begin{align*}
&\big(|a|+|d|\big)\delta\sqrt{1+\delta}+|b|\delta +\epsilon \geq |bxv-\widetilde{b}|\geq \Bigl|\bigl(|b|-|\widetilde{b}|\bigr)+|b|\bigl(|vx|-1\bigr)\Bigr|\geq \bigl||b|-|\widetilde{b}|\bigr|-|b|\delta, \\ 
&|a|\delta+2|b|\sqrt{\delta}\sqrt{1+\delta} +\epsilon \geq |dv^2-\widetilde{d}|\geq \Bigl|\bigl(|d|-|\widetilde{d}|\bigr)+|d|\bigl(|v|^2-1\bigr)\Bigr|\geq \bigl||d|-|\widetilde{d}|\bigr|-|d|\delta. 
\end{align*}
Since $B\neq \widetilde{B}$, a comparison of the left-hand and the right-hand sides of the above inequalities implies that at least one of them fails
for
$\epsilon, \delta$ such that 
\[
2\|B\|(\delta+\sqrt{\delta+\delta^{2}})+\epsilon<\max\Bigl\{\bigl||a|-|\widetilde{a}|\bigr|,\bigl||b|-|\widetilde{b}|\bigr|,\bigl||d|-|\widetilde{d}|\bigr| \Bigr\}\neq 0.
\]

\item \label{p10011001} $
(1\oplus \sigma,\widetilde{B})\dashrightarrow
(1\oplus \sigma,B)
$,\quad $\sigma\in\{1,-1\}$\\
From Lemma \ref{lemapsi1} (\ref{lemapsi11}) for (C\ref{r9}) for $\alpha=1$, $\omega=\sigma\in \{1,-1\}$ we have
\begin{align}\label{lemapsi111}
\bigl||x|^2+\sigma|u|^2-(-1)^{k}\bigr|\leq \delta, \,\, |\overline{x}y+\sigma\overline{u} v|\leq \delta,\, \,\bigl||y|^2+\sigma |v|^2-\sigma (-1)^{k}\bigr|\leq \delta, \quad k\in \mathbb{Z}.
\end{align}

\begin{enumerate}[label=(\alph*),wide=0pt,itemindent=2em,itemsep=6pt]
\item \label{p1t1tbi}
$\widetilde{B}=\widetilde{a}\oplus \widetilde{d}$,\,\, $0\leq a\leq d$, \quad $B=a\oplus d$, \,\,$0\leq \widetilde{a}\leq \widetilde{d}$, \qquad $ad=\widetilde{a}\widetilde{d}$ (see (\ref{pogoj})) 

From (\ref{eqBF1}) applied for $b=\widetilde{b}=0$ we get:
\begin{align}\label{eqBFadad}
&ax^{2}+du^2=\widetilde{a}+\epsilon_1,\nonumber \\
&axy+duv=\epsilon_2,\\ 
&ay^2+dv^{2}=\widetilde{d}+\epsilon_4. \nonumber
\end{align}
From the second equation of (\ref{eqBFadad}) and the second inequality (\ref{lemapsi111}), it follows that $\big||axy|-|duv|\big|\leq \epsilon$ and $\big||dxy|-|duv|\big|\leq d\delta$, respectively, and thus 
\begin{equation}\label{caseIb1}
|a-d|\,|xy|\leq d\delta+\epsilon.
\end{equation}

\begin{enumerate}[label=(\roman*),wide=0pt,itemindent=4em,itemsep=3pt]
\item
$a=\widetilde{a}=0$, $\widetilde{d}\neq d$, $d> 0$\\
Applying the triangle inequality to (\ref{eqBFadad}) for $a=\widetilde{a}=0$ yields 
\begin{equation}\label{equat0}
|u|^2\leq \tfrac{\epsilon}{d}, \qquad 
\bigl||v|^2-\tfrac{\widetilde{d}}{d}\bigr|\leq \bigl|v^2-\tfrac{\widetilde{d}}{d}\bigr|\leq \tfrac{\epsilon}{d},
\end{equation}
respectively. Using (\ref{lemapsi111}) and (\ref{caseIb1}) for $a=d$ we further obtain 
$ |y|^2\leq \tfrac{(\delta+\tfrac{\epsilon}{d})^2}{1-\delta-\tfrac{\epsilon}{d}}$ and 
\begin{equation*}
|y|^2\leq \tfrac{(\delta+\tfrac{\epsilon}{d})^2}{1-\delta-\tfrac{\epsilon}{d}}, \qquad 
\bigl||v|^2-(-1)^{k}\bigr|
 =\Big|\big( \sigma|y|^2+ |v|^2- (-1)^{k}\big)-\sigma|y|^{2}   \Big|
 \leq \delta+\tfrac{(\delta+\tfrac{\epsilon}{d})^2}{1-\delta-\tfrac{\epsilon}{d}}.
\end{equation*}
These inequalities and the last inequality in (\ref{equat0}) eventually lead to
\[
|\tfrac{\widetilde{d}}{d}-1|\leq
\big|\tfrac{\widetilde{d}}{d}-(-1)^{k}\big|
\leq \big|\tfrac{\widetilde{d}}{d}-|v|^{2}\big|+ \big||v|^{2}-(-1)^{k}\big|
\leq \tfrac{\epsilon}{d}+\delta+\tfrac{(\delta+\tfrac{\epsilon}{d})^2}{1-(\delta+\tfrac{\epsilon}{d})},
\]
which fails if $\epsilon,\delta$ are such that $\Omega=\frac{\epsilon}{d}+\delta$ is so small that 
$\Omega+\frac{\Omega^{2}}{1-\Omega}<|\frac{\widetilde{d}}{d}-1|$.

\item
$0<a\leq d$,
$0<\widetilde{a}\leq \widetilde{d}$, \quad $a\neq \widetilde{a}$,$a\neq \widetilde{d}$ \,(remember $ad=\widetilde{a}\widetilde{d}$, $B\neq \widetilde{B}$)

Using the notation (\ref{polarL}) the following calculation is validated trivially:
\begin{align}\label{exu2v2}
&x^{2}+\tfrac{d}{a}u^{2}= e^{2i\phi}\big(|x|^{2}+\sigma |u|^{2}\big)-u^{2} (\sigma e^{2i(\phi-\eta)}-\tfrac{d}{a}), \qquad \sigma\in \{-1,1\},\\
&y^{2}+\tfrac{d}{a}v^{2}=e^{2i\varphi}(|y|^{2}+\sigma |v|^{2})-v^{2} (\sigma e^{2i(\varphi-\kappa)}-\tfrac{d}{a}).\nonumber
\end{align}
Furthermore, one easily computes:
\begin{align}\label{exuv}
\overline{x}y+\sigma\overline{u}v&=e^{-2i\phi}\big((xy+\tfrac{d}{a}uv)+uv(\sigma e^{2i(\phi-\eta)}-\tfrac{d}{a})\big), \qquad \sigma\in \{-1,1\}\\
x\overline{y}+\sigma u\overline{v}&=e^{-2i\varphi}\big((xy+\tfrac{d}{a}uv)+uv(\sigma e^{2i(\varphi-\kappa)}-\tfrac{d}{a})\big).\nonumber
\end{align}
Using the second equation of (\ref{eqBFadad}) and the second inequality of (\ref{lemapsi111}) we conclude
\begin{equation*}
\big|uv(\sigma e^{2i(\phi-\eta)}-\tfrac{d}{a})\big|\leq \tfrac{\epsilon}{a}+\delta, \qquad 
\big|uv(\sigma e^{2i(\varphi-\kappa)}-\tfrac{d}{a})\big|\leq \tfrac{\epsilon}{a}+\delta, \qquad \sigma\in \{-1,1\}.
\end{equation*}
The above implies that at least one of the moduli of the second terms on the right-hand sides of equations (\ref{exu2v2}) is bounded from above by $\frac{\epsilon}{a}+\delta$, 
while from the first and the last inequality in (\ref{lemapsi111}) it follows that the moduli of the first terms on the right-hand sides of (\ref{exu2v2}) are bounded from above by $1+\delta$ and from below by $1-\delta$. 
For $\epsilon=\delta<\frac{\min\{a-\widetilde{a},\widetilde{d}-a\}}{a+2}$ thus the first or the last equation of (\ref{eqBFadad}) fails ($a\neq\widetilde{a},\widetilde{d}$).

\end{enumerate}

\item \label{tjepi}
$B=\begin{bsmallmatrix}
0 & b\\
b & 0
\end{bsmallmatrix}$, $b > 0$, \quad
$\widetilde{B}=\begin{bsmallmatrix}
\widetilde{a} & 0\\
0 & \widetilde{d}
\end{bsmallmatrix}$, $0< \widetilde{a}\leq \widetilde{d}$, \quad $\sigma=-1$ (see Lemma \ref{lemalist})\\
From (\ref{eqBF1}) for $a=d=0$ we obtain that
\begin{align}\label{eqadb}
&2bux=\widetilde{a}+\epsilon_1,\nonumber\\
&buy+bvx=\epsilon_2,\\
&2bvy=\widetilde{d}+\epsilon_4.\nonumber
\end{align}
By using Lemma \ref{lemadet} (\ref{lemadetb}) and $1=|\frac{\det \widetilde{A}}{\det A}|=|\frac{\det \widetilde{B}}{\det B}|=|\frac{\widetilde{a}\widetilde{d}}{-b^{2}}|$ (see (\ref{pogoj})) we deduce that 
%
$\det P=xv-yu=i(-1)^l+\delta'$ with $|\delta'|\leq \frac{\epsilon(4\widetilde{d}+2)}{b^{2}}$, $l\in \mathbb{Z}$.
%
Combining it further with the second equation of (\ref{eqadb}) we get $vx=\frac{1}{2}\big(i(-1)^k+\delta'+\frac{\epsilon_2}{b}\big)$ with $|vx|\geq \frac{1}{2}(1-\frac{\epsilon(4\widetilde{d}+2+b)}{b^{2}})$.
By taking $\epsilon \leq \frac{b}{4}(1-\frac{\epsilon(4\widetilde{d}+2+b)}{b^{2}})$ we guarantie $|\epsilon_2|\leq \frac{|bvx|}{2}$, hence
by applying (\ref{ocenah}) to the second equation of (\ref{eqadb}) we conclude (see (\ref{polarL})):
\begin{equation}\label{estpsi1}
\psi_1=(\phi+\kappa)-(\varphi+\eta+\pi)+2l_1\pi\in (-\tfrac{\pi}{2},\tfrac{\pi}{2}), \quad |\sin \psi_1|\leq \tfrac{4b\epsilon}{b^{2}-\epsilon(\widetilde{d}+2+b)},\,\,l_1\in \mathbb{Z}.
\end{equation}
Multplying the first and the last equation of (\ref{eqadb}) and using the triangle inequality gives 
$4b^{2}|uvxy|\geq \widetilde{a}\widetilde{d}-(\widetilde{a}+\widetilde{d})\epsilon-\epsilon^2$, so $|xy|^{2}$ or $|uv|^{2}$ (or both) is at least equal to $\frac{1}{2b}(\widetilde{a}\widetilde{d}-(\widetilde{a}+\widetilde{d})\epsilon-\epsilon^2)$.
Using (\ref{ocenah}) the second inequality of (\ref{lemapsi111}) for $\sigma=-1$ implies
\begin{equation*}
\psi_2=(\varphi-\phi)-(\kappa-\eta)+2l_2\pi\in (-\tfrac{\pi}{2},\tfrac{\pi}{2}), \qquad |\sin \psi_2|\leq \tfrac{4b|\delta|}{\sqrt{\widetilde{a}\widetilde{d}-(\widetilde{a}+\widetilde{d})\epsilon-\epsilon^2}}, \,\,\, l_2\in \mathbb{Z}.
\end{equation*}
Adding it to (\ref{estpsi1}) yields $\psi_1+\psi_2=\pi+(2l_1+2l_2+1)\pi$. Applying $\sin$ finally gives
\[
1=\big|\sin (\psi_1+\psi_2)\big|\leq |\sin \psi_1|+|\sin\psi_2|\leq \tfrac{4b|\delta|}{\sqrt{\widetilde{a}\widetilde{d}-(\widetilde{a}+\widetilde{d})\epsilon-\epsilon^2}}+\tfrac{4b\epsilon}{b^{2}-\epsilon(4\widetilde{\delta}+2+b)}.
\]
It is now easy to see that for any appropriately small $\epsilon$ and $\delta$ we get a contradiction.

\item
$\widetilde{B}=\begin{bsmallmatrix}
0 & \widetilde{b}\\
\widetilde{b} & 0
\end{bsmallmatrix}$, $\widetilde{b}>0$,\quad
$B=\begin{bsmallmatrix}
a & 0\\
0 & d
\end{bsmallmatrix}$, $0< a\leq d$, \quad $\sigma=-1$ (see Lemma \ref{lemalist})\\
From (\ref{eqBF1}) for $b=0$, $\widetilde{a}=\widetilde{d}=0$ we obtain that
\begin{align}\label{eqbad}
&ax^2+du^2=\epsilon_1, \nonumber\\
&axy+duv=\widetilde{b}+\epsilon_2, \\
&ay^2+dv^2=\epsilon_4.\nonumber
\end{align}
The first and the last equation imply 
\begin{equation}\label{eqbad13}
\bigl||ax^2|-|du^2|\bigr|\leq \epsilon,\qquad \big||ay^2|-|dv^2|\big|\leq \epsilon
\end{equation}
and by adding them to the first and the last equation of (\ref{lemapsi111}) for $\sigma=-1$ multiplied with $d$ and then using the triangle inequality, we get ($0< a\leq d$):
\[
\bigl|(d-a)|x^2|-d(-1)^k\bigr|\leq d\delta+\epsilon, \qquad \bigl|(d-a)|y^2|+d (-1)^k\bigr|\leq d\delta+\epsilon.
\]
One of the left-hand sides is at least $d$, so for $\epsilon=\delta<\frac{d}{2(d+1)}$ that inequality fails.

\end{enumerate}

\item \label{p01t001t0}
$
\bigl(\begin{bsmallmatrix}
0 & 1\\
\tau & 0 
\end{bsmallmatrix},
\begin{bsmallmatrix}
\widetilde{a} & \widetilde{b}\\
\widetilde{b} & \widetilde{d}
\end{bsmallmatrix}\bigr)\dashrightarrow
\bigl(\begin{bsmallmatrix}
0 & 1\\
\tau & 0
\end{bsmallmatrix},
\begin{bsmallmatrix}
a & b\\
b & d
\end{bsmallmatrix}\bigr)
$, \quad $\tau\in [0,1)$, $\widetilde{b},b \geq 0$

It is easy to check that 
$
P(s)=
\begin{bsmallmatrix}
s & s^2\\
0 & s^{-1}
\end{bsmallmatrix}
$ and 
$
P(s)=
\begin{bsmallmatrix}
s^{-1} & 0\\
s^2 & s
\end{bsmallmatrix}
$ with $c(s)=1$ in (\ref{cPepsi}) prove
$
\big(\begin{bsmallmatrix}
0 & 1\\
\tau & 0
\end{bsmallmatrix},
\begin{bsmallmatrix}
0 & b\\
b & 0
\end{bsmallmatrix}\big)
\to 
\big(
\begin{bsmallmatrix}
0 & 1\\
\tau & 0
\end{bsmallmatrix},
\begin{bsmallmatrix}
a & b\\
b & 0
\end{bsmallmatrix}\big)$
and 
$
\big(\begin{bsmallmatrix}
0 & 1\\
\tau & 0
\end{bsmallmatrix},
\begin{bsmallmatrix}
0 & b\\
b & 0
\end{bsmallmatrix}\big)
\to 
\big(
\begin{bsmallmatrix}
0 & 1\\
\tau & 0
\end{bsmallmatrix},
\begin{bsmallmatrix}
0 & b\\
b & d
\end{bsmallmatrix}\big)$, respectively.

\quad
By Lemma \ref{lemapsi1} (\ref{lemapsi11}) for (C\ref{r6}) we have
\begin{equation}\label{1txv}
|xu|,|yu|,|vy|\leq \delta, \qquad \big||vx|-1\big|\leq \delta.
\end{equation}
Set $ux=\delta_1$, $uy=\delta_2$, $vy=\delta_4$ and after rearranging the terms we write (\ref{eqBF1}) as 
\begin{align}\label{eqBF33}
&\widetilde{a}-ax^2=du^2-\epsilon_1+2b\delta_1,\nonumber \\
&bvx-\widetilde{b}=-axy-duv+\epsilon_2-b\delta_2,\qquad |\delta_1|,|\delta_2|,|\delta_3|\leq \delta\\ 
&dv^2-\widetilde{d}=\epsilon_4-ay^2-2b\delta_4, \nonumber
\end{align}
$|\delta_1|,|\delta_2|,|\delta_3|\leq \delta$.
Applying the triangle inequality to the second equation gives
\begin{align}\label{ocenabwb}
|axy+duv|+b\delta+\varepsilon 
\geq |bvx-\widetilde{b}|
\geq \big|b|vx|-\widetilde{b}\big|
= \big|b-\widetilde{b}+b(|vx|-1)\big|\geq |b-\widetilde{b}|-b\delta.
\end{align}
Next, multiplying $|xu|\leq \delta$ and $|vy|\leq \delta$ with $|vx|\leq 1+\delta$ (see (\ref{1txv})) gives $|x^2uv|\leq \delta(1+\delta)\leq 2\delta$ and $|v^2xy|\leq \delta(1+\delta)\leq 2\delta$, respectively (recall $\delta\leq \frac{1}{2}\leq 1$). Thus either $|x^2|$ or $|v^2|$ or $|uv|$, $|xy|$ (or more of them) are bounded by $\sqrt{2\delta}$.

\begin{enumerate}[label=(\alph*),wide=0pt,itemindent=2em,itemsep=6pt]

\item \label{p01t001t0a} $|x^2|\leq \sqrt{2\delta}$\\
From (\ref{1txv}) we get $|vx|\geq 1-\delta\geq \frac{1}{2}$, $|yv|\leq \delta$ and it further yields $|xy|=\frac{|yv|\,|x|^{2}}{|xv|}\leq 2\delta\sqrt{2\delta}$.
Applying the triangle inequality to the last equation of (\ref{eqBF33}) multiplied by $x^{2}$ we obtain an inequality that fails for $d\neq 0$ and any sufficiently small $\epsilon,\delta$:
\[
\tfrac{|d|}{2}-|\widetilde{d}|\sqrt{2\delta} \leq
\big| d(vx)^{2}-\widetilde{d}x^{2} \big|=|\epsilon_4x^{2}-a(xy)^{2}-2b\delta_4x^{2}|\leq (2b\delta+\epsilon)\sqrt{2\delta}+8|a|\delta^{3}.
\]

\quad
Further, if $d=0$ the last and the first equality of (\ref{eqBF33}) yield
\[
|\widetilde{d}|\leq 2b\delta+2|a|\delta+\epsilon, \qquad \sqrt{\epsilon}+2b\sqrt{\delta}\geq \epsilon+2b\delta\geq |\widetilde{a}-ax^{2}|\geq \big| |\widetilde{a}|-|ax^{2}|  \big|\geq |\widetilde{a}|-|a|\sqrt{2\delta},
\]
respectively. For $\widetilde{a}\neq 0$ (if $\widetilde{d}\neq 0$) the second (the first) inequality fails for $\epsilon=\delta<\frac{|\widetilde{d}|}{2|a|+2b+1}$ (for $\epsilon^{2}=\delta^{2}<\frac{|\widetilde{a}|}{\sqrt{2}|a|+2b+1}$).
Since $|x|^{2}\leq \sqrt{2\delta}$, $|y|^{2}\leq 2\delta$ we have $|axy|^{2}\leq |a|^{2}\sqrt{8\delta^{3}}$, so (\ref{ocenabwb}) for $d=0$ gives $|a|\sqrt[4]{4\delta^{3}}+b\delta+\epsilon \geq |b-\widetilde{b}|-b\delta$.
When $b\neq\widetilde{b}$ it is not too difficult to choose $\delta,\epsilon$ such that the above inequality fails.
Remember that the case $d=\widetilde{a}=\widetilde{d}=0$, $b=\widetilde{b}$ has allready been considered.

\item $|v^2|\leq \sqrt{2\delta}$\\
We deal with this case in the same manner as in \ref{p01t001t0} \ref{p01t001t0a}, we only replace $x,y,v,a,d,\widetilde{a},\widetilde{d}$ by $v,u,x,d,a,\widetilde{d},\widetilde{a}$, respectively; we get similar estimates for $\delta,\epsilon$. 

\item $|uv|,|xy|\leq \sqrt{2\delta}$\\
\quad
Here the inequality (\ref{ocenabwb}) fails for $b\neq\widetilde{b}$, $\sqrt{\epsilon}=\sqrt{\delta}<\frac{|\widetilde{b}-b|}{(|a|+|d|)\sqrt{2}+2b}$. Moreover, if $b=\widetilde{b}$:
\begin{equation}\label{bewb}
\sqrt{2\delta}\big(|a|+|d|\big)+b\delta+\epsilon 
\geq b|vx-1|.
\end{equation}

\quad
Next, using (\ref{1txv}) we obtain $\delta\sqrt{2\delta}\geq\big|(xu)(vu)\big|=|u^{2}xv|\geq |u|^2(1-\delta)$ and $\delta\sqrt{2\delta}\geq\big|(xy)(vy)\big|=|y^{2}xv|\geq |y|^2(1-\delta)$. Hence $ |y|^2,|u|^2\leq \frac{\delta\sqrt{2\delta}}{1-\delta}\leq 2\delta$ (recall $\frac{1}{2}\geq \delta$), so the first and the last equation of (\ref{eqBF33}) give 
\begin{equation}\label{1tdv2}
|\widetilde{a}-ax^{2}|\leq \epsilon + 2b\delta+2|d|\delta,\quad 
\big|\widetilde{d}-dv^{2}\big|\leq \epsilon+ 2b\delta+2|a|\delta.
\end{equation}

\quad
We set $x,y,u,v$ as in (\ref{polarL}) and let $d=|d|e^{i\vartheta}$, $\widetilde{d}=|\widetilde{d}|e^{i\widetilde{\vartheta}}$, $a=e^{i\iota}$, $\widetilde{a}=e^{i\widetilde{\iota}}$. Applying  (\ref{ocenah}) to the inequality (\ref{bewb}), to the estimates (\ref{1tdv2}) and to $\big|x\overline{v}-(-1)^{k}\big|\leq \delta$, $k\in \mathbb{Z}$ (see Lemma \ref{lemapsi1} (\ref{lemapsi11}) for (C\ref{r6}) with $0<\tau <1$), then leads to:
\begin{align}\label{ocenasinpsi}
&\psi_1=\kappa+\phi+2\pi l_1, \quad & &l_1\in \mathbb{Z},\,\,|\sin \psi_2|\leq \tfrac{2}{b}\big( \sqrt{2\delta}\big(|a|+|d|\big)+b\delta+\epsilon\big),\qquad \\
&\psi_2=2\phi+\iota-\widetilde{\iota}+2\pi l_2, \quad & &l_2\in \mathbb{Z},\,\,|\sin \psi_2|\leq \tfrac{2}{|\widetilde{a}|}\big( \epsilon + 2b\delta+2|d|\delta\big),\qquad \nonumber\\
&\psi_3=2\kappa+\vartheta-\widetilde{\vartheta}+2\pi l_3, \quad & &l_3\in \mathbb{Z},\,\,|\sin \psi_3|\leq \tfrac{2}{|\widetilde{d}|}\big( \epsilon + 2b\delta+2|a|\delta\big),\qquad \nonumber\\
&\psi_4=\phi-\kappa-k\pi+2\pi l_4,  \quad &&l_4\in \mathbb{Z},\,\,|\sin \psi_4|\leq 2\delta\qquad (\textrm{ if } 0<\tau <1), \nonumber
\end{align}
respectively. Subtracting (adding) the first and the last equation from the second-one (the third-one), rearranging the terms, and applying sin, gives for $1>\tau> 0$:
\[
\iota-\widetilde{\iota}=\psi_2-\psi_1-\psi_4+2\pi (l_4+l_1-l_2)-k\pi, \quad \vartheta-\widetilde{\vartheta}=\psi_3-\psi_1+\psi_4+2\pi( l_1-l_4-l_3)+k\pi,
\]
\begin{equation}\label{iotatheta}
\big|\sin (\iota-\widetilde{\iota})\big|,\big|\sin (\vartheta-\widetilde{\vartheta})\big| \leq \sum_{j=1}^{4}|\sin \psi_j|.
\end{equation}

\begin{enumerate}[label=(\roman*),wide=0pt,itemindent=4em,itemsep=3pt]
\item $ad\widetilde{a}\widetilde{d}=0$

First, let $\widetilde{a}=a=0$ (or $\widetilde{d}=d=0$). Since $\widetilde{B}\neq B$, we have $\widetilde{d}\neq d$ (or $\widetilde{a}\neq a$). The case $\widetilde{d}=0$ (or $\widetilde{a}=0$) has allready been considered, thus it is left to check the case $0<\tau<1$ with $\widetilde{d}=e^{i\widetilde{\vartheta}}$, $d=e^{i\vartheta}$, $\vartheta,\widetilde{\vartheta}\in [0,\pi)$ (from $\widetilde{a}=e^{i\widetilde{\iota}}$, $a=e^{i\iota}$, $\iota,\widetilde{\iota}\in [0,\pi)$); see Lemma \ref{lemalist}. 
By (\ref{iotatheta}), (\ref{ocenasinpsi}) it is easy to choose suitable $\epsilon,\delta$ to get a contradiction.

\quad
For $\widetilde{a}\neq 0$, $a=0$ (for $\widetilde{d}\neq 0$, $d=0$) the first (the second) inequality of (\ref{1tdv2}) fails for $\epsilon=\delta<\frac{|\widetilde{a}|}{1+2b+2|d|}$ (for $\epsilon=\delta<\frac{|\widetilde{d}|}{1+2b+2|a|}$).

\quad
Next, let $a,d\neq 0$.
Multiplying the estimates for $|ax^2|$, $|dv^2|$ given by (\ref{1tdv2}) yields
\[
|adx^{2}v^{2}|\leq \big(|\widetilde{a}|+\epsilon + 2b\delta +2|d|\delta\big)\big(|\widetilde{d}|+\epsilon + 2b\delta +2|a|\delta\big).
\]
From (\ref{bewb}) we further deduce $|vx|\geq 1-\tfrac{1}{b}\big(\sqrt{2\delta}(|a|+|d|)+b\delta+\epsilon\big)$, so by combining it with the above inequality we obtain an inequality that fails to hold, 
provided that at least one of $\widetilde{a},\widetilde{d}$ vanishes and $\epsilon,\delta$ are small enough.

\item $a,d,\widetilde{a},\widetilde{d}\neq 0$, \qquad $\frac{|ad|}{|\widetilde{a}\widetilde{d}|}=1$ (see (\ref{pogoj}))\\
From (\ref{1tdv2}) it follows that 
\[
x^{2}=\tfrac{\widetilde{a}}{a}(1+\delta_5),\quad |\delta_5|\leq \tfrac{|a|}{|\widetilde{a}|}\big(\epsilon + 2\delta(b+|d|)\big),\qquad v^{2}=\tfrac{\widetilde{d}}{d}(1+\delta_6),\quad |\delta_6|\leq \tfrac{|d|}{|\widetilde{d}|}\big(\epsilon + 2\delta(b+|a|)\big).
\]
Using (\ref{ocenakoren}) we get
\begin{align*}
&x=(-1)^{l_1}\sqrt{\tfrac{\widetilde{a}}{a}}(1+\delta_5'),\, \;|\delta_5'|\leq |\delta_5|,\quad v=(-1)^{l_2}\sqrt{\tfrac{\widetilde{d}}{d}}(1+\delta_6'),\,\; |\delta_6'|\leq |\delta_6|,\quad l_1,l_2\in \mathbb{Z},\\
&vx=(-1)^{l_1+l_2}\sqrt{\tfrac{\widetilde{a}\widetilde{d}}{ad}}(1+\delta_5'+\delta_6'+\delta_5'\delta_6').
\end{align*}
By choosing $\epsilon,\delta$ such that $\frac{|a|}{|\widetilde{a}|}\big(\epsilon + 2\delta(b+|d|)\big)\leq 1$ we assure that $|\delta_5'|, |\delta_6'|\leq 1$ and so $|\delta_5'\delta_6'|\leq |\delta_6'|$. Further, using (\ref{bewb}) and $\frac{|ad|}{|\widetilde{a}\widetilde{d}|}=1$ yields:
\begin{align*}
\sqrt{2\delta}\big(|a|+|d|\big)+b\delta  +\epsilon 
 &\geq b|vx-1| \geq\\
&\geq b\big|(-1)^{l_1+l_2}\sqrt{\tfrac{\widetilde{a}\widetilde{d}}{ad}}-1\big|-b\max\{\tfrac{|a|}{|\widetilde{a}|},\tfrac{|d|}{|\widetilde{d}|}\}\big(3\epsilon + 2\delta(2b+2|a|+|d|)\big).
\end{align*}
If $\tau=0$, then $d=\widetilde{d}=1$, $a\neq \widetilde{a}$ (see by Lemma \ref{lemalist} for $B\neq \widetilde{B}$), so we easily find $\epsilon,\delta$ to contradict this inequality. Similarly we treat the cases $\tau\in (0,1)$ with $B=1\oplus d$, $\widetilde{B}=1\oplus \widetilde{d}$, $d\neq \widetilde{d}$ or $
B=\begin{bsmallmatrix}
e^{i\iota} & b \\
b & d
\end{bsmallmatrix}
$,
$
\widetilde{B}=\begin{bsmallmatrix}
e^{i\widetilde{\iota}} & b \\
b & \widetilde{d}
\end{bsmallmatrix}
$, where either $|d|\neq|\widetilde{d}|$ or $\iota=\widetilde{\iota}$, $d= -\widetilde{d}$.
Finally, let $\tau\in (0,1)$, $a=e^{i\iota}$, $\widetilde{a}=e^{i\widetilde{\iota}}$, $\iota,\widetilde{\iota}\in [0,\pi)$, $|d|=|\widetilde{d}|$, and let $d\neq -\widetilde{d}$ (so  $|\vartheta-\widetilde{\vartheta}|\neq \pi$) precisely when $\iota=\widetilde{\iota}$.
From (\ref{iotatheta}), (\ref{ocenasinpsi}) we get a contradiction for any small $\epsilon,\delta$. 

\end{enumerate}

\end{enumerate}

\item \label{p1i}
$
\big(
\begin{bsmallmatrix}
0 & 1\\
1 & i
\end{bsmallmatrix},
\widetilde{B}\big)\dashrightarrow
\big(\begin{bsmallmatrix}
0 & 1\\
1 & i
\end{bsmallmatrix},
B\big)
$\\
Lemma \ref{lemapsi1} (\ref{lemapsi11}) with (C\ref{r5}) for $\alpha=0$, $\beta=1$, $\omega=i$, $k=0$ (since $\delta\leq \frac{1}{2}<1$) gives
\begin{align}\label{lemapsi11i}
\big|\overline{x}v+\overline{u}y-1\big|\leq \delta ,\quad |u|^2\leq \delta, \quad \big||v|^{2}-1\big|\leq \delta.
\end{align}

\begin{enumerate}[label=(\alph*),wide=0pt,itemindent=2em,itemsep=6pt]
\item 
$\widetilde{B}=\begin{bsmallmatrix}
0 & \widetilde{b}\\
\widetilde{b} & 0
\end{bsmallmatrix}$, $\widetilde{b}>0$, \quad
$B=a\oplus d
$, $a>0$, $d\in \mathbb{C}^{*}$\\
From (\ref{eqBF1}) we obtain the same equations as in (\ref{eqbad}) and the inequality (\ref{eqbad13}) but with $d\in\mathbb{C}^{*}$. Applying the triangle inequality to these (in)equalities, and 
using (\ref{lemapsi11i}) gives $|ax^2|\leq |d|\delta+\epsilon$, $|ay^2|\leq |d v^2|+\epsilon \leq |d|(1+\delta)+\epsilon$ and finally 
\begin{align*}
|\widetilde{b}|-\epsilon\leq|\widetilde{b}+\epsilon_2|=|axy+duv|\leq \sqrt{\big(|d|\delta+\epsilon\big)\big(|d|(1+\delta)+\epsilon\big)}+|d|\sqrt{\delta(1+\delta)}.
\end{align*}
It is straightforward to find $\delta,\epsilon$ such that that the last inequality fails.

\item 
$\widetilde{B}=\widetilde{a}\oplus \widetilde{d}
$, $\widetilde{a}>0$, $\widetilde{d}\in \mathbb{C}^{*}$,\quad  $B=\begin{bsmallmatrix}
0 & b\\
b & 0
\end{bsmallmatrix}$, $b>0$

We have the same equations as are (\ref{eqadb}), but with $\widetilde{d}\in \mathbb{C}^{*}$. 
From the first and the last of these equations we obtain the estimates $2b|xu|\geq \widetilde{a}-\epsilon$ and $2b|y|(1-\delta)\leq |\widetilde{d}|-\epsilon$ ($|v|^2\geq 1-\delta$ by (\ref{lemapsi11i})), respectively. 
Combining them and $|u|^{2}\leq \delta$, $|v|^{2}\geq 1-\delta$ with the second equation of (\ref{eqadb}) multiplied by $u$, yields
\[
\epsilon\sqrt{\delta}+\delta\tfrac{|\widetilde{d}|-\epsilon}{2(1-\delta)}\geq |bu^{2}y-\epsilon_2u|= |b(xu)v|\geq \tfrac{\widetilde{a}-\epsilon}{2}\sqrt{1-\delta}.
\]
Clearly, for any small enough $\epsilon,\delta$ we get a contradiction.

\item 
$\widetilde{B}=\widetilde{a}\oplus \widetilde{d}
$,\quad
$B=a\oplus d
$, \qquad $\widetilde{a},a\geq 0$, $\widetilde{d},d\in \mathbb{C}$, $a|d|=\widetilde{a}|\widetilde{d}|$ (see (\ref{pogoj}))\\
By slightly rearranging the terms in (\ref{eqBF1}) for $b=\widetilde{b}=0$, $\widetilde{d},d\in \mathbb{C}$ (see also (\ref{eqBFadad})), using the last two estimates in (\ref{lemapsi11i}) and applying the triangle inequality, we get
%
%
%
\begin{align}\label{ocena3adad}
\big|a|x|^2-\widetilde{a}\big| &\leq |ax^2-\widetilde{a}|=|\epsilon_1-du^2|\leq  \epsilon+|d|\delta,\\
|axy|^{2}         &=|\epsilon_2-duv|^2\leq \big(|d|\sqrt{\delta(1+\delta)}+\epsilon\big)^{2},\nonumber\\
\epsilon+ |ay^2|  &\geq|\epsilon_3-ay^2|= |dv^2-\widetilde{d}|\geq \big| |d|\,|v|^2-|\widetilde{d}|\big|= \big | |d|-|\widetilde{d}|+|d|(|v|^{2}-1)\big|\geq \nonumber \\
              &\geq \big||d|-|\widetilde{d}|\big|-|d|\delta\nonumber.
\end{align}
On the other hand the third equation of (\ref{eqBFadad}) gives $a|y|^{2}=|dv^2-\widetilde{d}+\epsilon|\leq |d|(1+\delta)+|\widetilde{d}|+\epsilon$.
Using this and the first two estimates of (\ref{lemapsi11i}) we conclude that
\begin{equation}\label{an0}
\delta+\sqrt{\delta}\sqrt{\tfrac{1}{a}\big(|d|(1+\delta)+|\widetilde{d}|+\epsilon\big)}\geq \delta+|\overline{u}y|\geq \big|\overline{x}v-1\big|\geq \big||vx|-1   \big|, \qquad  a\neq 0.
\end{equation}

\quad
Note that according to the list in Lemma \ref{lemalist}, $a=0$ ($\widetilde{a}=0$) implies $d\geq 0$ ($\widetilde{d}\geq 0$).

\begin{enumerate}[label=(\roman*),wide=0pt,itemindent=4em,itemsep=3pt]

\item $a=0$ or $\widetilde{a}=0$\\
If $a=0$, $\widetilde{a}\neq 0$ the first inequality of (\ref{ocena3adad}) fails for $\epsilon=\delta<\frac{\widetilde{a}}{1+|d|}$, while for $a=\widetilde{a}=0$ (hence $d,\widetilde{d}\geq 0$, $d\neq \widetilde{d}$) with $\epsilon=\delta<\frac{|d-\widetilde{d}|}{1+|d|}$ the third inequality of (\ref{ocena3adad}) fails. If $\widetilde{a}=0$, $a\neq 0$, then the first inequality of (\ref{ocena3adad}) yields $|x^{2}|\leq \frac{|d|\delta+\epsilon}{a}$. Since $|v|^{2}\leq 1-\delta$, it is now not to difficult to contradict (\ref{an0}) for any sufficiently  small $\epsilon,\delta$.

\item $a,\widetilde{a},d,\widetilde{d}\neq 0$\\   
Inequalities (\ref{ocena3adad}) yield $ |ax^{2}|\geq\widetilde{a}-\epsilon-|d|\delta$, $|ay|^{2}\leq \tfrac{\big(|d|\sqrt{\delta(1+\delta)}+\epsilon\big)^{2}}{\widetilde{a}-\epsilon-|d|\delta}$ and
\[
\epsilon+ \tfrac{\big(|d|\sqrt{\delta(1+\delta)}+\epsilon\big)^{2}}{\widetilde{a}-\epsilon-|d|\delta}\geq \big||d|-|\widetilde{d}|\big|-|d|\delta.
\]
If $|d|\neq |\widetilde{d}|$ it is straightforward to obtain a contradiction for any $\epsilon,\delta$ small enough. 

\quad
If $|d|=|\widetilde{d}|$, then $a|d|=\widetilde{a}|\widetilde{d}|$ yields $a=\widetilde{a}$, $d=|d|e^{i\vartheta}$, $\widetilde{d}=|d|e^{i\widetilde{\vartheta}}$ with $\widetilde{\vartheta}-\vartheta\neq 2l\pi$, $l\in \mathbb{Z}$.
The right-hand (left-hand) side of the first (third) inequality of (\ref{ocena3adad}) thus gives 
\[
x^{2}=
1+\delta_5,\quad |\delta_5|\leq \tfrac{1}{\widetilde{a}}(\epsilon + \delta|d|),\qquad v^{2}=e^{i(\widetilde{\vartheta}-\vartheta)}(1+\delta_6),\quad |\delta_6|\leq \tfrac{(|d|\sqrt{\delta(1+\delta)}+\epsilon)^{2}}{|\widetilde{d}|(\widetilde{a}-\epsilon-|d|\delta)}.
\]
If we choose $\epsilon,\delta$ such that $|\delta_5'|, |\delta_6'|\leq 1$, then using (\ref{ocenakoren}) it follows that
\small
\[
x=(-1)^{l_1}
(1+\delta_5'),\,\, |\delta_5'|\leq \tfrac{1}{\widetilde{a}}(\epsilon + \delta|d|),\quad 
v=(-1)^{l_2}e^{i\frac{\widetilde{\vartheta}-\vartheta}{2}}(1+\delta_6'),\,\, |\delta_6'|\leq \tfrac{(|d|\sqrt{\delta(1+\delta)}+\epsilon)^{2}}{|\widetilde{d}|(\widetilde{a}-\epsilon-|d|\delta)},
\]
\normalsize
where $l_1,l_2\in \mathbb{Z}$. Hence 
%
$v\overline{x}=(-1)^{l_1+l_2}e^{i\frac{\widetilde{\vartheta}-\vartheta}{2}}(1+\overline{\delta}_5')(1+\delta_6')$
%
and further
\begin{equation*}
\big|v\overline{x}-(-1)^{l_1+l_2}e^{i\frac{\widetilde{\vartheta}-\vartheta}{2}}\big|\leq \overline{\delta}_5'+\delta_6'+\overline{\delta}_5'\delta_6'.
\end{equation*}
Since $\frac{\widetilde{\vartheta}-\vartheta}{2} \neq l'\pi$, $l'\in \mathbb{Z}$, then combining the above inequality with (\ref{an0}) and using the triangle inequality yields a contradiction for every small $\epsilon,\delta$.

\item $d=\widetilde{d}= 0$, \; $a,\widetilde{a}\neq 0$,\; $a\neq \widetilde{a}$ (since $B\neq\widetilde{B}$)\\
%
%
The first equation of (\ref{eqBFadad}) and the last inequality of (\ref{lemapsi11i}) give $|x|^{2}=\frac{\widetilde{a}}{a}(1+\delta_1)$, $|\delta_1|\leq \frac{a}{\widetilde{a}}\epsilon$ and $|v|^{2}=1+\delta_2$, $|\delta_2|\leq \delta$, respectively. Further,
(\ref{ocenakoren}) leads to $|x|=\sqrt{\frac{\widetilde{a}}{a}}(1+\delta_1')$, $|\delta_1'|\leq \frac{a}{\widetilde{a}}\epsilon$ and $|v|=1+\delta_2'$, $|\delta_2'|\leq \delta \leq \frac{1}{2}$. Therefore we obtain $\big||xv|-\sqrt{\frac{\widetilde{a}}{a}}\big|\leq \sqrt{\frac{\widetilde{a}}{a}}\delta+\frac{3a}{2\widetilde{a}}\epsilon+\delta$. From (\ref{an0}) for $d=\widetilde{d}=0$ and applying the triangle inequality we deduce
\[
\big|1-\sqrt{\tfrac{\widetilde{a}}{a}}\big|\leq \big(\sqrt{\tfrac{\widetilde{a}}{a}}\delta+\tfrac{3a}{2\widetilde{a}}\epsilon+\delta\big)+\big(\delta+\sqrt{\tfrac{\epsilon\delta}{a}}\big), \textrm{ which failes for  }
\epsilon=\delta<\tfrac{\big|1-\sqrt{\frac{\widetilde{a}}{a}}\big|}{\sqrt{\frac{\widetilde{a}}{a}}+\frac{3a}{2\widetilde{a}}+2+\sqrt{\frac{1}{a}}}.
\]

\end{enumerate}

\end{enumerate}


\item \label{p11-101i}
$
(1\oplus -1,\widetilde{B})
\dashrightarrow
\big(\begin{bsmallmatrix}
0 & 1\\
1 & i
\end{bsmallmatrix},
B\big)$

From Lemma \ref{lemapsi1} (\ref{lemapsi11}) with (C\ref{r5}) for $-\omega=\alpha=1$, $\beta=0$ we obtain
\begin{equation}\label{ocena4adad}
\big|2\Rea(\overline{x}u)-(-1)^{k}\big|\leq \delta, \quad \big|2\Rea(\overline{y}v)+(-1)^k\big|\leq \delta, \, k\in \mathbb{Z}, \quad |u|^{2},|v|^2\leq \delta
\end{equation}

\begin{enumerate}[label=(\alph*),wide=0pt,itemindent=2em,itemsep=6pt]

\item
$B=a\oplus d
$, \quad $a\geq 0$

\begin{enumerate}[label=(\roman*),wide=0pt,itemindent=4em,itemsep=3pt]
\item
$\widetilde{B}=\widetilde{a}\oplus \widetilde{d}
$, \quad $\widetilde{a},\widetilde{d}\geq 0$\\
For $c(s)=1$, $P(s)=\begin{bsmallmatrix}
\frac{1}{2}s^{-1} & -\frac{1}{2}s^{-1}\\
s & s
\end{bsmallmatrix}$ in (\ref{cPepsi}) we obtain 
$(1\oplus -1,0_2)\to\big(\begin{bsmallmatrix}
0 & 1\\
1 & i
\end{bsmallmatrix},0 \oplus d\big)$, $d> 0$.

\quad
We have the same equations as in (\ref{eqBFadad}).
Since $|u|^{2},|v|^{2}\leq \delta$ (see (\ref{ocena4adad})) we deduce that the first (third) of these equation fails for $a=0$, $\widetilde{a}\neq 0$, $\epsilon=\delta<\frac{\widetilde{a}}{|d|+1}$ ($a=0$, $\widetilde{d}\neq 0$, $\epsilon=\delta<\frac{\widetilde{d}}{|d|+1}$).
Next, $d=0$, $a\neq 0$ yields $x^2=\frac{\widetilde{a}+\epsilon_1}{a}$, hence $(1-\delta)^{2}\leq |2\Rea (\overline{x}u)|^{2}\leq |2xu|^{2}\leq 4\frac{\widetilde{a}+\epsilon}{a}\delta$, which fails for any $\delta,\epsilon$ chosen sufficiently small. Finally, for $a,d\neq 0$ (hence $\widetilde{a},\widetilde{d}\neq 0$) the first and the last equation of (\ref{eqBFadad}) (with $|u|^{2},|v|^{2}\leq \delta$) give $|ax^2|\geq \widetilde{a}-|d| \delta-\epsilon $, $|ay^2|\geq \widetilde{d}-|d| \delta-\epsilon $, respectively. Provided that $\widetilde{a},\widetilde{d}\geq \epsilon+|d|\delta$ we multiply these inequalities and a comparison to the second equation of (\ref{eqBFadad}) implies:
\[
\big( \widetilde{a}-|d| \delta-\epsilon\big)\big(\widetilde{d}-|d| \delta-\epsilon \big)\leq |axy|^{2}=|\epsilon_2-duv|^{2}\leq \big(\epsilon+|d|\delta\big)^{2},
\]
which fails to hold for $\epsilon=\delta<\tfrac{\widetilde{a}\,\widetilde{d}}{(|d|+1)(\widetilde{a}+\widetilde{d})}$.

\item
$\widetilde{B}=\begin{bsmallmatrix}
0 & \widetilde{b}\\
\widetilde{b} & 0
\end{bsmallmatrix}$, $\widetilde{b}>0$\\ 
We have equations (\ref{eqbad}). Its first and its last equation yield $|ax^2|,|ay^2|\leq |d| \delta+\epsilon $, what further with the second equation and $|u|^{2},|v|^{2}\leq \delta$ (see (\ref{ocena4adad})) gives
\[
\widetilde{b}-\epsilon\leq |axy+duv|\leq |axy|+|duv|\leq \big(|d| \delta+\epsilon\big) +|d|\delta.
\]
Since $\widetilde{b}> 0$ this inequality fails for $\epsilon=\delta<\frac{\widetilde{b}}{2(|d|+1)}$. 

\end{enumerate}

\item
$
B=\begin{bsmallmatrix}
0 & b\\
b & 0
\end{bsmallmatrix}
$, $b>0$

\begin{enumerate}[label=(\roman*),wide=0pt,itemindent=4em,itemsep=3pt]
\item
$\widetilde{B}=\begin{bsmallmatrix}
0 & \widetilde{b}\\
\widetilde{b} & 0
\end{bsmallmatrix}$, $\widetilde{b}> 0$\\
The first equation of (\ref{eqBF1}) for $a=d=\widetilde{a}=0$ is $2bux=\epsilon_1$, which implies $|2\Rea \overline{x}u|\leq \frac{\epsilon}{b}$. 
For $\epsilon=\delta <\frac{b}{b+1}$ the first inequality of (\ref{ocena4adad}) then fails to hold.

\item
$\widetilde{B}=\widetilde{a}\oplus \widetilde{d}
$, \quad $\widetilde{a},\widetilde{d}>0$, \qquad $\widetilde{a}\widetilde{d}=b^{2}$ (see (\ref{pogoj}))\\
The first and the last equation of (\ref{eqadb}) give $|2\overline{x}u|\leq \tfrac{\widetilde{a}+\epsilon}{b}$, $|2\overline{y}v|\leq \tfrac{\widetilde{d}+\epsilon}{b}$.
Combined with the first two estimates in (\ref{ocena4adad}), it leads to
\begin{align*}
1-\delta\leq \big|2\Rea(\overline{x}u)\big|\leq |2\overline{x}u|\leq \tfrac{\widetilde{a}+\epsilon}{b}, \qquad 1-\delta\leq \big|2\Rea(\overline{y}v)\big|\leq |2\overline{y}v|\leq \tfrac{\widetilde{d}+\epsilon}{b}.
\end{align*}
Since $\widetilde{a}\widetilde{d}=b^{2}$ then for $\widetilde{a}\neq \widetilde{d}$ we either have $\frac{\widetilde{a}}{b}<1$ or $\frac{\widetilde{d}}{b}<1$. In both cases one of the above inequalities fails for $\epsilon=\delta<\frac{b-\min\{\widetilde{a},\widetilde{d}\}}{b+1}$.
Moreover, when $\widetilde{a}= \widetilde{d}=b$, then $c(s)=-1$, $P(s)=\frac{1}{\sqrt{2}}\begin{bsmallmatrix}
i s^{-1} & s^{-1}\\
-i s & s
\end{bsmallmatrix}$ in (\ref{cPepsi}) proves 
$(1\oplus -1,bI_2)\to\big(\begin{bsmallmatrix}
0 & 1\\
1 & i
\end{bsmallmatrix},\begin{bsmallmatrix}
0 & b\\
b & 0
\end{bsmallmatrix}\big)$.

\end{enumerate}

\end{enumerate}

\item \label{p011001i}
$
\big(\begin{bsmallmatrix}
0 & 1\\
1 & 0
\end{bsmallmatrix},
\widetilde{B}\big)
\dashrightarrow
\big(\begin{bsmallmatrix}
0 & 1\\
1 & i
\end{bsmallmatrix},
B\big)
$

Lemma \ref{lemapsi1} (\ref{lemapsi11}) with (C\ref{r7}) for $\alpha=\omega=0$, $\beta=0$ yields
\begin{equation}\label{ocena2psi01}
|u|^{2},|v|^{2}\leq \delta,\quad \big|2\Rea(\overline{y}v)\big|\leq \delta, \quad \big|2\Rea(\overline{x}u)\big|\leq \delta, \quad \big|\overline{x}v+\overline{u}y-(-1)^{k}\big|\leq \delta,\,k\in \mathbb{Z}. 
\end{equation}

\begin{enumerate}[label=(\alph*),wide=0pt,itemindent=2em,itemsep=6pt]
\item
$
B=a\oplus d
$, $a\geq 0$  
\\ 
Setting $u^{2}=\delta_1$, $v^{2}=\delta_2$ (with $|\delta_1|,|\delta_2|\leq \delta$), and rearranging the terms 
in the first and the last equation of (\ref{eqBF1}) for $b=0$, gives
\begin{align*}
ax^{2}=\widetilde{a}+\epsilon_1-d\delta_1,\qquad 
ay^2=\widetilde{d}+\epsilon_4-d\delta_2.
\end{align*}
Since $\widetilde{B}\neq 0$ constants $\widetilde{a},\widetilde{d}$ do not both vanish (see Lemma \ref{lemalist}). For $a=0$ 
at least one of the above equations fails, provided that $\epsilon=\delta<\frac{\max\{|\widetilde{a}|,|\widetilde{d}|\}}{|d|+1}$.
Further, if $a> 0$ we obtain 
$|x|^{2}\leq \frac{1}{a}\big( |\widetilde{a}|+\epsilon+d\delta \big)$ and $|y|^{2}\leq \frac{1}{a}\big(|\widetilde{d}|+\epsilon+d\delta  \big)$, 
and using (\ref{ocena2psi01}) we get 
\[
\sqrt{\tfrac{\delta}{a}}\big(\sqrt{|\widetilde{a}|+\epsilon+d\delta}+\sqrt{|\widetilde{d}|+\epsilon+d\delta}\big)\geq |\overline{x}v+\overline{u}y|\geq 1-\delta.
\]
It is not difficult to contradict this inequality for any sufficiently small $\epsilon,\delta$. 

\item \label{ha1}
$
B=\begin{bsmallmatrix}
0 & b\\
b & 0
\end{bsmallmatrix}$, $b>0$

\begin{enumerate}[label=(\roman*),wide=0pt,itemindent=4em,itemsep=3pt]
\item \label{hi1}
$
\widetilde{B}=\begin{bsmallmatrix}
0 & \widetilde{b}\\
\widetilde{b} & 1
\end{bsmallmatrix}
$, $\widetilde{b}>0$\\ 
Since $1=\frac{|\det \widetilde{A}|}{|\det A|}=\frac{|\det \widetilde{B}|}{|\det B|}=\frac{\widetilde{b}^{2}}{b^{2}}$ (see (\ref{pogoj})), we have  (\ref{eqBF1}) for $a=d=\widetilde{a}=0$, $\widetilde{d}=1$, $\widetilde{b}=b$:
\begin{align}\label{eqabdb}
&2bux=\epsilon_1,\nonumber \\
&bvx+buy=b+\epsilon_2,\\ 
&2byv=1+\epsilon_4. \nonumber
\end{align}
Furthermore, Lemma \ref{lemadet} (\ref{lemadetb}) gives $\det P=vx-uy=(-1)^{l}+\delta'$, $l\in \mathbb{Z}$, with $|\delta'|\leq \frac{\epsilon}{b^{2}}(4\max\{1,b\}+2)$. By combining it with the second equation of (\ref{eqabdb}) we obtain:
\begin{equation}\label{vx1uy}
vx=\tfrac{1}{2}\big(\tfrac{b+\epsilon_2}{b}+(-1)^{l}+\delta'\big), \quad uy=\tfrac{1}{2}\big(\tfrac{b+\epsilon_2}{b}-(-1)^{l}-\delta'\big), \qquad l\in \mathbb{Z}.
\end{equation}
If $\epsilon,\delta$ are small, then for $l$ even (odd), $uy$ is close to $0$ (close to $1$) and $vx$ is close to $1$ (close to $0$), and hence $\overline{x}v$ (or $\overline{u}y$) is close to $(-1)^{k}$, $k\in \mathbb{Z}$ by (\ref{ocena2psi01}).  
Using the notation (\ref{polarL}) for $x,y,u,v$ we apply (\ref{ocenah}) to the third equation of (\ref{eqabdb}), to (\ref{vx1uy}) and to the last estimate of (\ref{ocena2psi01}), provided that $\frac{\epsilon}{b}+|\delta'|+\delta \leq \frac{1}{2}$. We deduce that
\[
\psi=\varphi+\kappa, \quad
\psi_l=\left\{
\begin{array}{ll}\phi+\kappa, & l=1\\
                 \varphi+\eta, & l=0
 \end{array}, \quad               
\right.
\psi_l'=\left\{\begin{array}{ll}\kappa-\phi-k\pi, & l=1\\
              \varphi-\eta-k\pi, & l=0
\end{array}
\right.
\]\\
(so $\varphi-\kappa=(-1)^{l}(k\pi+\psi_l+\psi_l'-\psi)$), where
$|\sin \psi|\leq 2\epsilon $, 
$|\sin \psi_l|\leq \frac{2\epsilon}{b^{2}}(b+4\max\{1,b\}+2)$, $|\sin \psi_l'|\leq \frac{2\epsilon}{b^{2}}(b+4\max\{1,b\}+2)+2\delta$, $l\in \{0,1\}$.
Therefore using 
$|\sin (\varphi-\kappa)|\leq |\sin \psi_l'|+|\sin\psi_l|+|\sin \psi|$ together with the squared second estimate of (\ref{ocena2psi01}) 
and the third equation of (\ref{eqabdb}) we get
\begin{align*}
\delta^{2}&\geq\big|2\Rea (\overline{v}y)\big|^{2}=4|vy|^{2}\big|\cos (\varphi-\kappa)\big|^{2}
\geq (\tfrac{1-\epsilon}{2b})^{2}\big(1-|\sin (\varphi-\kappa) |^{2}\big)\\
&\geq (\tfrac{1-\epsilon}{2b})^{2}\Big(1-4\big|\epsilon+\tfrac{\epsilon}{b^{2}}(b+4\max\{1,b\}+2)+\delta \big|^{2}\Big).
\end{align*}
It is straightforward to choose $\epsilon,\delta$ small enough so that this inequality fails.

\item \label{c6bii}
$
\widetilde{B}=1\oplus \widetilde{d}
$, \quad $\Ima \widetilde{d} >0$ ($\widetilde{d}=|\widetilde{d}|e^{i\widetilde{\vartheta}}$, $0<\widetilde{\vartheta}<\pi$)\\
From (\ref{eqBF1}) for $a=d=0$ and $\widetilde{b}=0$, $\widetilde{a}=1$ we get
\begin{align}\label{eqBF12}
&2bux=1+\epsilon_1,\nonumber \\
&bvx+buy=\epsilon_2,\\ 
&2byv=\widetilde{d}+\epsilon_4. \nonumber
\end{align}
As $1=\frac{|\det \widetilde{A}|}{|\det A|}=\frac{|\det \widetilde{B}|}{|\det B|}=\frac{|\widetilde{d}|}{b^{2}}$, Lemma \ref{lemadet} gives $\det P=vx-uy=(-1)^{l}ie^{i\frac{\widetilde{\vartheta}}{2}}+\delta'$, $l\in \mathbb{Z}$, $|\delta'|\leq \frac{\epsilon(4\max\{1,|\widetilde{d}|\}+2)}{b\sqrt{|\widetilde{d}|}}$. Adding (subtracting) it from second equation of (\ref{eqBF12}) yields
\[
vx=\tfrac{1}{2}\big(\tfrac{\epsilon_2}{b}+(-1)^{l}ie^{i\frac{\widetilde{\vartheta}}{2}}+\delta'\big), \quad 
uy=\tfrac{1}{2}\big(\tfrac{\epsilon_2}{b}-(-1)^{l}ie^{i\frac{\widetilde{\vartheta}}{2}}-\delta'\big).
\]
We combine this with the last estimate in (\ref{ocena2psi01}) and write it in the notation (\ref{polarL}):
\begin{align*}
(-1)^{k}+\delta_1=\overline{x}v+\overline{u}y  & =e^{-2i\phi}(xv+yu)+e^{-2i\phi}(e^{2i(\phi-\eta)}-1)yu\\
&=e^{-2i\phi}\tfrac{\epsilon_2}{b}+2i\sin (\phi-\eta)e^{-i(\phi+\eta)}\tfrac{1}{2}\big(\tfrac{\epsilon_2}{b}-(-1)^{l}ie^{i\frac{\widetilde{\vartheta}}{2}}-\delta'\big),
\end{align*}
%
where $|\delta_1|\leq \delta$.
Using (\ref{ocenah}) we deduce
$\psi=\frac{\widetilde{\vartheta}}{2}-\eta-\phi+l\pi-k\pi+2\pi m$, $m\in \mathbb{Z}$, $|\sin \psi|\leq 2\big(\frac{2\epsilon}{b}+\delta+|\delta'|\big)$,
while the first equation of (\ref{eqBF12}) gives $\psi'=\phi+\varphi+2\pi m'$, $|\sin \psi'|\leq 2\epsilon$. Thus $\psi'+\psi=\frac{\widetilde{\vartheta}}{2}+n\pi$, $n\in \mathbb{Z}$ and by applying sin we conclude
\[
|\sin \tfrac{\widetilde{\vartheta}}{2}|\leq |\sin \psi|+|\sin \psi'|\leq  2\epsilon+2\big(\tfrac{2\epsilon}{b}+\delta+\tfrac{\epsilon(4\max\{1,|\widetilde{d}|\}+2)}{b\sqrt{|\widetilde{d}|}}\big),
\]
which fails for $\epsilon=\delta<|\sin \frac{\widetilde{\vartheta}}{2}| \big(4+\frac{4}{b}+\tfrac{2(4\max\{1,|\widetilde{d}|\}+2)}{b\sqrt{|\widetilde{d}|}}\big)^{-1}$ (remember $0< \vartheta<\pi$).

\end{enumerate}

\end{enumerate}

\item \label{p1-1v0110}
$
(1\oplus -1
,\widetilde{B})
\dashrightarrow
\big(\begin{bsmallmatrix}
0 & 1\\
1 & 0
\end{bsmallmatrix},
B\big)
$

Lemma \ref{lemapsi1} (\ref{lemapsi11}) with (C\ref{r7}) for $\alpha=-\omega=1$, $\beta=0$ gives
\begin{equation}\label{eqBF8}
2\Rea (\overline{x}u)=(-1)^{k}+\delta_1, \quad 2\Rea (\overline{y}v)=-(-1)^{k}+\delta_4, \quad \overline{x}v+\overline{u}y=\delta_2, \qquad k\in \mathbb{Z},
\end{equation}
with $|\delta_1|,|\delta_2|,|\delta_4|\leq \delta$.

\begin{enumerate}[label=(\alph*),wide=0pt,itemindent=2em,itemsep=6pt]
\item $\det B=\det \widetilde{B}=0$

First, consider
$(1\oplus -1, 0\oplus \widetilde{d}) \dashrightarrow
\big(\begin{bsmallmatrix}
0 & 1\\
1 & 0
\end{bsmallmatrix},
1\oplus 0 \big)
$, $\widetilde{d}>0$.
By applying (\ref{ocenakoren}) to the first and the last equation of
(\ref{eqBF1}) for $a=1$, $b=d=\widetilde{a}=0$ we get
%
\begin{equation}\label{enakyu}
x=(-1)^{l_1}\sqrt{\epsilon_1},\quad y=(-1)^{l_4}\big(\sqrt{\widetilde{d}}+\epsilon_4'\big),\quad l_1,l_4\in \mathbb{Z},\; |\epsilon_4'|\leq \tfrac{\epsilon_4}{\sqrt{|\widetilde{d}|}}
\end{equation}
respectively.
Next, by manipulation of the third expression of (\ref{eqBF8}) we deduce
\begin{align*}
\delta_2=\overline{x}v+\overline{u}y  & =e^{-2i\phi}(xv-yu)+e^{-2i\phi}(e^{2i(\phi-\eta)}+1)yu\\
&=e^{-2i\phi}\det P +2\cos (\phi-\eta)e^{-i(\phi+\eta)}uy
\end{align*}
By multiply it with $x$, rearranging the terms and slightly simplify leads to 
\[
\delta_2x-e^{-2i\phi}x\det P= 
 2\cos (\phi-\eta)|ux|\,y =2y\Rea (\overline{x}u) 
\]
To conclude, Lemma \ref{lemadet} (\ref{lemadeta}) gives
%
$|\det P|=1+\delta'$, $|\delta'|\leq 6\|E\|\leq  \tfrac{6\delta}{\nu}$ with $ \nu >0$,
%
and combining it with the first equation of (\ref{eqBF8})  and (\ref{enakyu}) eventually contradicts the above equation for any sufficiently small $\epsilon,\delta$.

%
%
 
\quad
For $c(s)=1$, $P(s)=\frac{1}{2}\begin{bsmallmatrix}
2\sqrt{1+s}-2 & 2-2\sqrt{1+s}\\
1+\sqrt{1+s} & 1+\sqrt{1+s}
\end{bsmallmatrix}$ in (\ref{cPepsi})   
we get
$(1\oplus -1,0_2)\to
\big(\begin{bsmallmatrix}
0 & 1\\
1 & 0
\end{bsmallmatrix},1 \oplus 0\big)$.

\item 
$
B=\begin{bsmallmatrix}
0 & b\\
b & 1
\end{bsmallmatrix}
$, $b>0$

From (\ref{eqBF1}) for $a=0$, $d=1$ we obtain
\begin{align}\label{eqBF5}
&2bux+u^2=\widetilde{a}+\epsilon_1\nonumber \\
&bvx+buy+uv=\widetilde{b}+\epsilon_2\\ 
&v^2+2byv=\widetilde{d}+\epsilon_4. \nonumber
\end{align}

Multiplying $\det P=vx-uy$ with $b$ and then adding and subtracting it from the second equation of (\ref{eqBF5}) yields
\[
2bvx+uv=b\det P+\widetilde{b}+\epsilon_2, \qquad 
2buy+uv=\widetilde{b}+\epsilon_2-b\det P,
\]
respectively.
Multiplying the first (the second) equation by $u$ (by $v$) and comparing it with the first (the last) equation of (\ref{eqBF5}), multiplied by $v$ ($u$), gives 
\[
u(b\det P+\widetilde{b}+\epsilon_2)=v(\widetilde{a}+\epsilon_4), \qquad v(-b\det P+\widetilde{b}+\epsilon_2)=u(\widetilde{d}+\epsilon_4).
\]
For $v,u\neq 0$ it follows that
\begin{equation}\label{vdivu}
\tfrac{u}{v}=\tfrac{\widetilde{b}-b\det P+\epsilon_2}{\widetilde{d}+\epsilon_4}=\tfrac{\widetilde{a}+\epsilon_4}{b\det P+\widetilde{b}+\epsilon_2}.
\end{equation}
%

\begin{enumerate}[label=(\roman*),wide=0pt,itemindent=4em,itemsep=3pt]
\item 
$\widetilde{B}=\begin{bsmallmatrix}
0 & \widetilde{b}\\
\widetilde{b} & 0
\end{bsmallmatrix}$, $\widetilde{b}>0$

The first and the last equation of (\ref{eqBF5}) for $\widetilde{a}=\widetilde{d}=0$ immediately imply that $\big|2b|ux|-|u|^2\big|\leq \epsilon$ and $\big|2b|vy|-|v|^2\big|\leq \epsilon$.
Further, (\ref{eqBF8}) gives that $2|xu|,2|yv|\geq 1-\delta$ and hence $|u|^{2},|v|^{2}\geq b(1-\delta)-\epsilon$. For $\epsilon<\frac{b}{2}$ we have $u,v\neq 0$ (since $\delta\leq \frac{1}{2}$), so the first and the last equation of (\ref{eqBF5}) yield $2bx=-u+\frac{\epsilon_1}{u}$ and  
$2by=-v+\frac{\epsilon_1}{v}$.
Therefore 
\[
\overline{x}u=\tfrac{1}{2b}(-\overline{u}+\tfrac{\overline{\epsilon_1}}{\overline{u}})u=\tfrac{1}{2b}(-|u|^{2}+\tfrac{\overline{\epsilon_1}u}{\overline{u}}), \quad 
\overline{y}v=\tfrac{1}{2b}(-\overline{v}+\tfrac{\overline{\epsilon_1}}{\overline{v}})v=\tfrac{1}{2b}(-|v|^{2}+\tfrac{\overline{\epsilon_4}v}{\overline{v}}).
\]
Adding the real parts of these equalities, applying the triangle inequality and using the first two equations of (\ref{eqBF8}) with the lower estimates on $|u|^{2}$, $|v|^{2}$ gives
\[
2\delta\geq 2\big|\Rea (\overline{x}u)+\Rea (\overline{y}v)\big|=\big|\tfrac{1}{b}(|u|^2+|v|^2)-\tfrac{1}{b}\Rea(\tfrac{\overline{\epsilon_1}u}{\overline{u}}+\tfrac{\overline{\epsilon_4}v}{\overline{v}})\big|\geq 2(1-\delta)-\tfrac{4\epsilon}{b},
\]
which fails for $\epsilon=\delta<\frac{2b}{4b+4}$.

\item
$\widetilde{B}=\widetilde{a}\oplus \widetilde{d}
$, \quad
$0<\widetilde{a}\leq\widetilde{d}$

If $uv=0$, then (\ref{eqBF5}) fails for $\epsilon<\widetilde{a}$.
Next, for $v\neq 0$
%
we easily validate the following calculation
\begin{align}\label{calcxov}
x\overline{v}+u\overline{y}  & =e^{-2i\kappa}\big((xv-yu)+(e^{2i(\kappa-\varphi)}+1)yu\big)=e^{-2i\kappa}\det P +2\cos (\kappa-\varphi)e^{-i(\kappa+\varphi)}uy\nonumber\\
&=e^{-2i\kappa}\det P +2\Rea (\overline{y}v) \tfrac{u}{v}.
\end{align}
Combining this with the second and the third equation of (\ref{eqBF8}), and comparing it with (\ref{vdivu}) for $\widetilde{b}=0$, leads to
\[
\frac{u}{v}=\tfrac{\overline{\delta}_2-e^{-2i\kappa}\det P}{-(-1)^{k}+\delta_4}
=\tfrac{-b\det P+\epsilon_2}{\widetilde{d}+\epsilon_4}=\tfrac{\widetilde{a}+\epsilon_4}{b\det P+\epsilon_2}.
\]
Since $1=|\frac{\det \widetilde{A}}{\det A}|=|\frac{\det \widetilde{B}}{\det B}|=|\frac{\widetilde{a}\widetilde{d}}{-b^{2}}|$ (see (\ref{pogoj})), then $b^{2}=\widetilde{a}\widetilde{d}$ and Lemma \ref{lemadet} (\ref{lemadetb}) yields $\det P=vx-uy=(-1)^{l}i+\delta'$, $l\in \mathbb{Z}$, $|\delta'|\leq \frac{\epsilon(4\widetilde{d}+2)}{b^{2}}$. Thus either $\widetilde{a}=\widetilde{d}=b=1$ or we obtain a contradiction for any suitably small $\epsilon,\delta$.

\quad
For $c(s)=1$, $P(s)=\frac{1}{2}\begin{bsmallmatrix}
\frac{1}{2s} & -\frac{i}{2s}\\
s & is
\end{bsmallmatrix}$ in (\ref{cPepsi})   
it follows that
$(1\oplus -1,I_2)\to
\big(\begin{bsmallmatrix}
0 & 1\\
1 & 0
\end{bsmallmatrix},
\begin{bsmallmatrix}
0 & 1\\
1 & 1
\end{bsmallmatrix}\big)$.

\end{enumerate}

\item \label{p1-1v0110c}
$B=1\oplus d
$, \quad $\Ima d>0$

From (\ref{eqBF1}) for $b=0$, $a=1$ we obtain
\begin{align}\label{eqBF10}
&x^{2}+du^2=\widetilde{a}+\epsilon_1,\nonumber \\
&xy+duv=\widetilde{b}+\epsilon_2,\\ 
&y^2+dv^{2}=\widetilde{d}+\epsilon_4. \nonumber
\end{align}

\begin{enumerate}[label=(\roman*),wide=0pt,itemindent=4em,itemsep=3pt]
\item \label{p1-1v0110ci}
$\widetilde{B}=
\begin{bsmallmatrix}
0 & \widetilde{b}\\
\widetilde{b} & 0
\end{bsmallmatrix}$, $\widetilde{b}>0$

The first (the second) equation of (\ref{eqBF8}) yields that at least one of $|x|^2,|u|^2$ (and $|y|^2,|v|^2$) is greater or equal to $\frac{1}{4}$ (remember $\delta\leq \frac{1}{2}$). 
Since equations (\ref{eqBF10}) for $\widetilde{a}=\widetilde{d}=0$ are the same as in (\ref{eqbad}) for $a=1$, $d\in \mathbb{C}$, also inequalities (\ref{eqbad13}) for $a=1$ are valid; if $|x|^2\geq \frac{1}{4}$, then $|du|^{2}\geq \frac{1}{4}-\epsilon$ (if $|u|^2\geq \frac{1}{4}$ we have $|x|^{2}\geq \frac{|d|}{4}-\epsilon$). Similarly holds for $y,v$ as well, so:
\begin{equation}\label{ocenaux4}
|x|^{2},|y|^{2},|du|^{2},|dv|^{2}\geq \tfrac{1}{4}\min\{|d|,1\}-\epsilon.
\end{equation}
%
Applying (\ref{ocenah}) with $\epsilon\leq\frac{1}{2}\big(\frac{1}{4}\min\{1,|d|\}-\epsilon\big)$ to the first and to the last equation of (\ref{eqBF10}) for $\widetilde{a}=\widetilde{d}=0$, $d=|d|e^{i\vartheta}$ and $x,y,u,v$ as in (\ref{polarL}), leads to
\[
\psi_1=2\phi-2\eta-\vartheta-\pi, \quad
\psi_2=2\varphi-2\kappa-\vartheta-\pi, \quad 
|\sin \psi_1|,|\sin \psi_2|\leq \tfrac{8\epsilon}{\min\{1,|d|\}-4\epsilon}.
\]
The last equation of (\ref{eqBF8}) (using (\ref{ocenah})) similarly gives 
\[
\psi=-\phi+\kappa-(-\eta +\varphi), \qquad 
|\sin \psi|\leq \tfrac{8\delta}{|d|}\big(\min\{1,|d|\}-4\epsilon\big)^{-1}.
\]
Collecting everything together yields $\frac{1}{2}(\psi_1+\psi_2)+\psi=-2\pi-\vartheta$ and
\[
|\sin \vartheta|=\big|\sin (\tfrac{1}{2}\psi_1+\tfrac{1}{2}\psi_2+\psi)\big|\leq \tfrac{1}{2}|\sin \psi_1|+\tfrac{1}{2}|\sin \psi_2|+|\sin\psi|\leq  \tfrac{4(\epsilon+\frac{\delta}{|d|})}{\min\{1,|d|\}(1-\delta)-2\epsilon}.
\]
Since $\Ima d>0$ we can easily choose $\epsilon$, $\delta$ to contradict this inequality.

\item \label{p1-1v0110cii}
$\widetilde{B}=\widetilde{a}\oplus \widetilde{d}
$, \quad $0<\widetilde{a}\leq \widetilde{d}$

The first two equations of (\ref{eqBF8}) imply $|xu|,|vy|\geq \frac{1}{2}(1-\delta)$. 
Hence $|xvuy|\geq \frac{1}{4}(1-\delta)^{2}$, so the second equation of (\ref{eqBF10}) for $\widetilde{b}=0$ (the last equation of (\ref{eqBF8})) yields that either $|duv|$ or $|xy|$ (either $|\overline{x}v|$ or $|\overline{u}y|$) or both are at least $\frac{1}{2}\sqrt{|d|}(1-\delta)$ (at least $\frac{1}{2}(1-\delta)$). 
By setting $x,y,u,v$ as in (\ref{polarL}), $d=|d|e^{i\vartheta}$ and using (\ref{ocenah}) we also deduce:
\begin{equation}\label{ocenakot1}
\psi_1=(\phi+\varphi)-(\eta +\kappa+\vartheta+\pi)+2l_1\pi,\qquad \psi_2=(-\phi+\kappa)-(\varphi-\eta+\pi)+2l_2\pi
\end{equation}
with $|\sin \psi_1|\leq \frac{4\epsilon}{\sqrt{|d|}(1-\delta)}$, $|\sin \psi_2|\leq \frac{4\delta}{1-\delta}$, $l_1,l_2\in \mathbb{Z}$. Hence $\psi_1+\psi_2=\vartheta+(2l_1+2l_2-1)\pi$,
\begin{equation}\label{ocenakot2}
|\sin \vartheta|=\big|\sin (\psi_1+\psi_2)\big|\leq |\sin \psi_1|+|\sin \psi_2|\leq \tfrac{4\epsilon}{\sqrt{|d|}(1-\delta)}+\tfrac{4\delta}{1-\delta},
\end{equation}
which fails for $\epsilon=\delta<\frac{\sqrt{|d|}|\sin \vartheta|}{4+5\sqrt{|d|}+\sqrt{|d|}|\sin \vartheta|}$ (remember $\Ima d>0$). 
\end{enumerate}

\end{enumerate}

\item \label{p0110v1-1}
$
\big(\begin{bsmallmatrix}
0 & 1\\
1 & 0
\end{bsmallmatrix},
\widetilde{B}\big)\dashrightarrow
(1\oplus -1,B)
$

Lemma \ref{lemapsi1} (\ref{lemapsi11}) for (C\ref{r12}) gives 
\begin{equation}\label{est12}
 |x|^2-|u|^2= \delta_1,\;\; \overline{x}y-\overline{u}v-(-1)^{k}= \delta_2,\;\; |y|^2-|v|^2= \delta_4,\;\; |\delta_1|,|\delta_2|,|\delta_4|\leq \delta,\,k\in\mathbb{Z}.
\end{equation}

\begin{enumerate}[label=(\alph*),wide=0pt,itemindent=2em,itemsep=6pt]

\item $
B=a\oplus d
$, \quad$0<a\leq d$

\begin{enumerate}[label=(\roman*),wide=0pt,itemindent=4em,itemsep=3pt]
\item
$
\widetilde{B}=1\oplus \widetilde{d}
$,\quad $\Ima \widetilde{d}>0$\\
We have equations (\ref{eqBFadad}) for $\widetilde{a}=1$, $\Ima \widetilde{d}>0$.
By combining calculations (\ref{exu2v2}), (\ref{exuv}) for $\sigma=-1$ with the first two equations of (\ref{eqBFadad}) for $\widetilde{a}=1$ and the first two inequalities of (\ref{est12}) we deduce 
\[
\tfrac{\epsilon_2}{a}+\tfrac{1}{a}= e^{2i\phi}\delta_1-u^{2} (- e^{2i(\phi-\eta)}-\tfrac{d}{a}),\qquad 
(-1)^{k}+ \delta_2=e^{-2i\phi}\big(\tfrac{\epsilon_2}{a}+uv(- e^{2i(\phi-\eta)}-\tfrac{d}{a})\big).
\]
By rearranging the terms and appying the triangle inequality we further get
\[
\Big|\big|u^{2}(e^{2i(\phi-\eta)}+\tfrac{d}{a})\big|-\tfrac{1}{a}\Big|\leq \tfrac{\epsilon}{a}+\delta, \qquad 
\Big|\big|uv( e^{2i(\phi-\eta)}+\tfrac{d}{a})\big|-1\Big|\leq \tfrac{\epsilon}{a}+\delta.
\]
It is immediate that $\big|a|u|^{2}-|uv|\big|\,| e^{2i(\phi-\eta)}+\tfrac{d}{a}|\leq (a+1)(\tfrac{\epsilon}{a}+\delta)$. Moreover, for $a\neq d$: 
\[
\tfrac{\frac{1}{a}-\frac{\epsilon}{a}-\delta}{|1+\frac{d}{a}|} \leq |u|^{2}\leq \tfrac{\frac{1}{a}+\frac{\epsilon}{a}+\delta}{|1-\frac{d}{a}|},\qquad
\big|a|u|-|v|\big|\leq 
\tfrac{(a+1)(\frac{\epsilon}{a}+\delta)\sqrt{|1+\frac{d}{a}|}}{| 1-\frac{d}{a}|\sqrt{\frac{1}{a}-\frac{\epsilon}{a}-\delta}}.
\]
It follows (see also (\ref{est12})) that for sufficiently small $\epsilon,\delta$ we have $|x|,|u|$ (and $|y|,|v|$) arbitrarily close and bounded (for $a\neq d$), so it is straightforward to get a contradiction with the second equation of (\ref{eqBFadad}).

\quad
Next, let $a=d$. By combining (\ref{exu2v2}), (\ref{exuv}) for $\sigma=-1$, $a=d$ with equations (\ref{est12}), (\ref{eqBFadad}) for $\widetilde{a}=1$,  and slightly simplifying, we obtain 
\begin{align*}
&2u^{2}\cos (\phi-\eta)e^{i(\phi-\eta)}=\tfrac{1+\epsilon_1}{a}-e^{2i\phi}\delta_1, \; -2uv\cos (\phi-\eta)e^{-i(\phi+\eta)}=(-1)^{k}+\delta_2-e^{-2i\phi}\epsilon_2,\\
& 2v^{2}\cos (\varphi-\kappa)e^{i(\varphi-\kappa)}=\tfrac{\widetilde{d}+\epsilon_4}{a}-e^{2i\varphi}\delta_4,
 \; -2uv\cos (\varphi-\kappa)e^{-i(\varphi+\kappa)}=(-1)^{k}+\delta_2-e^{-2i\varphi}\epsilon_2.
\end{align*}

By applying (\ref{ocenah}) to these equations with $\widetilde{d}=|\widetilde{d}|e^{i\widetilde{\vartheta}}$ 
we then deduce
\begin{align*}
&\psi_1=(\phi+\eta)-\pi l_1,\quad \psi_3=(\kappa-\phi)-\pi(k+l_3),\quad |\sin \psi_1|\leq 2\epsilon+2a\delta,\; |\sin \psi_3|\leq 2\epsilon+\delta, \\
& \psi_2=(\varphi+\kappa)-\pi l_2-\widetilde{\vartheta}, \;\;\; \psi_4=(\eta-\varphi)-\pi(k+l_4), \;\;\; |\sin \psi_2|\leq \tfrac{2\epsilon+2a\delta}{|\widetilde{d}|},\, |\sin \psi_4|\leq 2\epsilon+2\delta,
\end{align*}
where $l_1,l_2,l_3,l_4\in \mathbb{Z}$. It implies further that 
$\psi_1+\psi_3-\psi_2-\psi_4=2\pi(l_1-l_2+l_3-l_4)+\widetilde{\vartheta}$, which yields an inequality that fails for $\epsilon=\delta=\frac{|\Ima \widetilde{d}|}{8+(1+a)(1+|\widetilde{d}|)}$ (recall $\Ima \widetilde{d}>0$):
\[
|\sin \vartheta|\leq (2\epsilon+2a\delta)(1+\tfrac{1}{|\widetilde{d}|})+4\epsilon+\delta.
\]


\item
$
\widetilde{B}=\begin{bsmallmatrix}
0 & \widetilde{b}\\
\widetilde{b} & 1
\end{bsmallmatrix}
$, $\widetilde{b}>0$, \quad $ad=\widetilde{b}^{2}$ (see (\ref{pogoj}))\\
%
%
Combining the first and the last equality of (\ref{est12}) with the first and the last equation of (\ref{eqBF1}) for $\widetilde{d}=1$, $\widetilde{a}=b=0$, and applying the triangle inequality, we deduce
\begin{align*}
&|a-d||u|^{2}\leq \big|a|x|^2-d|u|^{2}\big|+\big|a|u|^{2}-a|x|^{2}\big|\leq |ax^2+du^{2}|+a\big||u|^{2}-|x|^{2}\big|\leq \epsilon+a\delta,\\
&|a-d| |v|^{2}\leq \big|a|y|^2-d|v|^{2}\big|+\big|a|v|^{2}-a|y|^{2}\big|\leq |ay^2+dv^{2}|+a\big||y|^{2}-|v|^{2}\big|\leq 1+\epsilon+a\delta .\nonumber
\end{align*}
These inequalities and (\ref{est12}) imply the upper estimates for $|x|,|y|,|u|,|v|$ in case $a\neq d$, which further gives an inequality that fails for any $\epsilon,\delta$ sufficiently small: 
\[
1-\delta\leq |\overline{x}y-\overline{u}v|\leq  \tfrac{\sqrt{(\epsilon+a\delta)(1+\epsilon+a\delta)}}{|a-d|}+(\tfrac{\epsilon+a\delta}{|a-d|}+\delta )(\tfrac{1+\epsilon+a\delta}{|a-d|}+\delta),\qquad a\neq d.
\]

\quad
Next, when $a=d$ (hence $\widetilde{b}=d$) it is convenient to conjugate the first pair of matrices with 
$\frac{1}{2}
\begin{bsmallmatrix}
2 & 2\\
1 & -1
\end{bsmallmatrix}
$:
\[
\big(\begin{bsmallmatrix}
0 & 1\\
1 & 0
\end{bsmallmatrix},
\begin{bsmallmatrix}
0 & d\\
d & 1
\end{bsmallmatrix}\big)
\approx
\big(1\oplus -1,
\tfrac{1}{4}\begin{bsmallmatrix}
4d+1 & 1\\
1 & -4d+1
\end{bsmallmatrix}\big)
\dashrightarrow
(
1\oplus -1,
d I_2 ), \quad d>0
.
\]

Using (\ref{lemapsi111}) for $\sigma=-1$ and (\ref{eqBF1}) for $\widetilde{a}=\frac{1}{4}(4d+1)$, $\widetilde{d}=\frac{1}{4}(-4d+1)$, $\widetilde{b}=\frac{1}{4}$, $a=d$, $b=0$, we can write the identities  (\ref{exu2v2}), (\ref{exuv}) for $a=d$, $\sigma=-1$ in the form
\begin{align}\label{equvu2v2}
&\qquad \tfrac{1+4d}{4d}+ \tfrac{\epsilon_1}{d}=e^{2i\phi}\big(\delta_1'+(-1)^{k}\big)+u^{2}(e^{2i(\phi-\eta)}+1), \\
&\qquad \tfrac{1-4d}{4d}+ \tfrac{\epsilon_3}{d}=e^{2i\varphi}\big(\delta_3'-(-1)^{k}\big)+v^{2}(e^{2i(\varphi-\kappa)}+1),\nonumber\\
\delta_2'=& e^{-2i\phi}\big((\tfrac{1}{4d}+\tfrac{\epsilon_2}{d})-uv(e^{2i(\phi-\eta)}+1)\big), \quad
\delta_2'=e^{-2i\varphi}\big((\tfrac{1}{4d}+\tfrac{\epsilon_2}{d})-uv(e^{2i(\varphi-\kappa)}+1)\big),\nonumber
\end{align}
where $|\delta_1'|,|\delta_2'|,|\delta_3'|\leq \delta$, $k\in \mathbb{Z}$. By combining the first two and the last two equations (rearranging the terms and then multiplying the equations) we obtain
\small
\begin{align*}
u^{2}v^{2}(e^{2i(\phi-\eta)}+1)(e^{2i(\varphi-\kappa)}+1) = & \big(  \tfrac{1+4d}{4d}+\tfrac{\epsilon_1}{d}-e^{2i\phi}(\delta_1' +(-1)^{k})\big)\big(  \tfrac{1-4d}{4d}+\tfrac{\epsilon_3}{d}-e^{2i\varphi}(\delta_3'-(-1)^{k})\big)\\
u^{2}v^{2}(e^{2i(\phi-\eta)}+1)(e^{2i(\varphi-\kappa)}+1) = &(  \tfrac{1}{4d}+\tfrac{\epsilon_2}{d}-e^{2i\phi}\delta_2' )(  \tfrac{1}{4d}+\tfrac{\epsilon_2}{d}-e^{2i\varphi}\delta_2' ),
\end{align*}
\normalsize
respectively. By comparing the right-hand sides of the equations and rearranging the terms we eventually conclude
\begin{align*}
\tfrac{1}{16d^{2}}&=\tfrac{1-16d^{2}}{16d^{2}}+(-1)^{k}( -e^{2i\phi}\tfrac{1-4d}{4d}+\tfrac{1+4d}{4d}e^{2i\varphi})-e^{2i\phi+2i\varphi}-\epsilon'\\
\epsilon'&=-1-e^{2i(\phi+\varphi)}+(-1)^{k}(e^{2i\phi}+e^{2i\varphi})+\tfrac{1}{4d}(-1)^{k}(e^{2i\varphi}-e^{2i\phi})\\
\epsilon'&=\big(4\sin (\phi+\tfrac{k\pi}{2})\sin (\varphi+\tfrac{k\pi}{2})+i\tfrac{1}{2d}\sin (\phi-\varphi)\big) e^{i(\phi+\varphi+k\pi)},
\end{align*}
where $|\epsilon'|\leq 2(\frac{1+4d}{4d}+\frac{\epsilon}{d}+\delta)(\tfrac{\epsilon}{d}+\delta)+2(\frac{1}{4d}+\frac{\epsilon}{d}+\delta)(\frac{\epsilon}{d}+\delta)$. Thus $\tfrac{1}{2d}\big|\sin (\phi-\varphi)\big|\leq \epsilon'$ and either $2\big|\sin (\phi+\tfrac{k\pi}{2})\big|\leq \sqrt{\epsilon'}$ or $2\big|\sin (\varphi+\tfrac{k\pi}{2})\big|\leq \sqrt{\epsilon'}$ (or both). Observe also that if one of the expressions $2|\sin (\phi+\tfrac{k\pi}{2})|$ or $2|\sin (\varphi+\tfrac{k\pi}{2})|$ is bounded by $\sqrt{\epsilon'}$, then by using the angle diffence formula for sine and applying the triangular inequality we deduce that the other expression is bounded by $ \tfrac{\sqrt{\epsilon'}+\epsilon'}{\sqrt{1-\epsilon'}}$.  

\quad
In the same manner, but with $x,u,y,v,\phi,\eta,\varphi,\kappa,\delta_1',\delta_2',\delta_3',k$ replaced by $u,x,v,y,\eta$, $\phi,\kappa,\varphi,-\delta_1',-\delta_2',-\delta_3',k+1$, respectively, in (\ref{exuv}), (\ref{exu2v2}) for $a=d=\widetilde{b}$, $\sigma=-1$ and (\ref{equvu2v2}), we conclude that 
$2\big|\sin (\eta+\tfrac{(k+1)\pi}{2})\big|,2\big|\sin (\kappa+\tfrac{(k+1)\pi}{2})\big| \leq \tfrac{\sqrt{\epsilon'}+\epsilon'}{\sqrt{1-\epsilon'}}$. It follows that
$\big|\cos (\phi+\kappa)\big|=\big|\sin (\phi+\frac{k\pi}{2}+\kappa+\frac{(k+1)\pi}{2})\big|\leq 2\tfrac{\sqrt{\epsilon'}+\epsilon'}{\sqrt{1-\epsilon'}}$.
On the other hand, the third equation of (\ref{equvu2v2}) with its terms rearranged and simplified can be written as
\[
e^{2i\phi}\delta_2'- \tfrac{\epsilon_2}{d}=\tfrac{1}{4d}-2|uv|\cos (\phi-\eta)e^{i(\phi+\kappa)}.
\]
Provided that $\frac{1}{8d}\geq \delta+\frac{\epsilon}{d}$, then applying (\ref{ocenah}) yields $|\sin(\phi+\kappa)|\leq 8d(\delta+\frac{\epsilon}{d})$. We have a contradiction for any small enough $\epsilon,\delta$.
\end{enumerate}

\item
$
B=\begin{bsmallmatrix}
0 & b\\
b & 0
\end{bsmallmatrix}$, $b>0$

\begin{enumerate}[label=(\roman*),wide=0pt,itemindent=4em,itemsep=3pt]

\item
$
\widetilde{B}=\begin{bsmallmatrix}
0 & \widetilde{b}\\
\widetilde{b} & 1
\end{bsmallmatrix}
$, $\widetilde{b}>0$\\
The first equation of (\ref{eqabdb}) and (\ref{est12}) yield that at least one of $|x|^{2}$ or $|u|^{2}$ is not larger that $\frac{\epsilon}{2b}$ and the other is not larger than $\frac{\epsilon}{2b}+\delta$. Similarly, the last equation of (\ref{eqabdb}) and the last equation of (\ref{est12}) give $|y|^{2}, |v|^{2}\leq \frac{1+\epsilon}{2b}+\delta$. The second estimate of (\ref{est12}) finally implies a contradiction for any  sufficiently small $\epsilon,\delta$:
\[
1-\delta\leq |\overline{x}y-\overline{u}v|\leq 2(\tfrac{\epsilon}{2b}+\delta)(\tfrac{1+\epsilon}{2b}+\delta).
\]

\item
$\widetilde{B}=1\oplus \widetilde{d}
$, \quad
$\Ima \widetilde{d} >0$\\
The same proof as in \ref{p011001i} \ref{ha1} \ref{c6bii} works here, too. 

\end{enumerate}

\item $B=0\oplus d$, $d>0$,\quad $\widetilde{B}=1\oplus 0$\\
We consider
$\big(\begin{bsmallmatrix}
0 & 1\\
1 & 0
\end{bsmallmatrix},1 \oplus 0 \big)\approx \big( 1\oplus-1,
\begin{bsmallmatrix},
1 & 1\\
1 & 1
\end{bsmallmatrix}\big)\dashrightarrow (1\oplus -1,0\oplus d)$, $d>0$. The first and the last equation of (\ref{eqBF1}) for $a=b=0$, $\widetilde{a}=\widetilde{b}=\widetilde{d}=1$ yield that $u^2=\frac{1+\epsilon_1}{d}$ and $v^2=\frac{1+\epsilon_4}{d}$. Hence the first and the last estimate of (\ref{lemapsi111}) for $\sigma=-1$ give 
\begin{equation*}
\big||x|^2-(\tfrac{1}{d}+(-1)^{k})\big|\leq \delta+\tfrac{\epsilon}{d} ,\quad \big||y|^2-(\tfrac{1}{d}-(-1)^{k})\big|\leq\delta+\tfrac{\epsilon}{d},\qquad k\in \mathbb{Z},
\end{equation*}
respectively. For $d\geq 1$ at least one of these inequalities fails (provided that $\delta+\frac{\epsilon}{d}<1-\frac{1}{d}$), while for $d<1$
by combining the inequalities with $(\frac{1-\epsilon}{d}-\delta)^{2}\leq \big(|\overline{u}v|-\delta\big)^{2} \leq |\overline{x}y|^{2}$ (see the second estimate of (\ref{lemapsi111}) for $\sigma=-1$) and simplifying, we get
\begin{align*}
(\tfrac{1-\epsilon}{d}-\delta)^{2}
& \leq 
\big(\tfrac{1}{d}+(-1)^{k}+\tfrac{\epsilon}{d}+\delta\big)\big(\tfrac{1}{d}-(-1)^{k}+\tfrac{\epsilon}{d}+\delta\big),\\
-2\tfrac{\epsilon}{d}+\tfrac{\epsilon^{2}}{d^{2}}
& \leq 
-1+2(\tfrac{\epsilon}{d}+\delta)\tfrac{2}{d}+(\tfrac{\epsilon}{d}+\delta)^{2}.
\end{align*}
It is not difficult to choose $\epsilon,\delta$ small enough to contradict this inequality.
\end{enumerate}

\item \label{p01t001t01}
$
\bigl(\begin{bsmallmatrix}
0 & 1\\
1 & 0 
\end{bsmallmatrix}, \widetilde{B}
\bigr)\dashrightarrow
\bigl(\begin{bsmallmatrix}
0 & 1\\
1 & 0
\end{bsmallmatrix},B
\bigr)
$, \quad $\det B=\det \widetilde{B}\neq 0$ (see Lemma \ref{lemalist} and (\ref{pogoj})) 

Lemma \ref{lemapsi1} (\ref{lemapsi11}) with (C\ref{r7}) for $\alpha=\omega=0$, $\beta=1$ and with some $|\delta_1|,|\delta_2|,|\delta_4|\leq \delta$ gives
\begin{equation}\label{eqBF88}
2\Rea (\overline{x}u)=\delta_1, \quad \overline{x}v+\overline{u}y=(-1)^{k}+\delta_2, \quad 2\Rea (\overline{y}v)=\delta_4, \qquad k\in \mathbb{Z}.
\end{equation}

\begin{enumerate}[label=(\alph*),wide=0pt,itemindent=2em,itemsep=6pt]


\item
$
B=1\oplus d
$,
$
\widetilde{B}=1\oplus \widetilde{d}
$, \quad $d=|d|e^{i\vartheta}$, $\widetilde{d}=|\widetilde{d}|e^{i\widetilde{\vartheta}}$, $0<\vartheta,\widetilde{\vartheta}<\pi$, \quad
$|d|=|\widetilde{d}|$, $\vartheta\neq \widetilde{\vartheta}$ 

Due to the second equation of (\ref{eqBF10}) for $\widetilde{b}=0$ it follows that given a positive constant $s\leq \frac{1}{16}$ (to be choosen) it suffices to consider the following two cases:

\begin{enumerate}[label=(\roman*),wide=0pt,itemindent=4em,itemsep=3pt]

\item $|xy|,|duv|\leq s$\\
It is clear that either $|x|^{2}\leq s$ or $|y|^{2}\leq s$. We only consider $|x|^{2}\leq s$ (the case $|y|^{2}\leq s$ is treated similarly, we just replace $x,y,u,v,\widetilde{a},\widetilde{d},\widetilde{\vartheta},\arg \widetilde{a}$ with $y,x,v,u,\widetilde{d},\widetilde{a},\arg \widetilde{a},\widetilde{\vartheta}$). The first equation of (\ref{eqBF10}) yields $|du^{2}-\widetilde{a}|\leq \epsilon+s$ and (\ref{ocenah}) gives further $\psi_1=2\eta+\vartheta-\arg \widetilde{a}\in (-\frac{\pi}{2},\frac{\pi}{2})$ with $|\sin \psi_1|\leq \frac{2(\epsilon+s)}{|\widetilde{a}|}$, provided that $s+\epsilon\leq \frac{1}{2|\widetilde{a}|}$. Since $|duv|\leq s$ we also get $|dv^{2}|\leq \frac{s^{2}}{|\widetilde{a}|-\epsilon-s}\leq \frac{2s^{2}}{|\widetilde{a}|}$. Applying (\ref{ocenah}) to the third equation of (\ref{eqBF10}) and to the second equation of (\ref{eqBF88}) implies
\[
\psi_2=2\varphi-\widetilde{\vartheta}, \quad \psi=\varphi-\eta-k\pi,\quad  |\sin \psi_2|\leq 2(\tfrac{2s^{2}}{|\widetilde{a}\widetilde{d}|}+\tfrac{\epsilon}{|\widetilde{d}|}),  \; |\sin \psi|\leq 2\sqrt{\tfrac{2s^{3}}{|\widetilde{a}d|}}+2\delta.
\]
By choosing $s<\min \{\frac{1}{4|\widetilde{a}|},\sqrt{\frac{|\widetilde{a}\widetilde{d}|}{8}}\}$ then for any sufficiently small $\epsilon,\delta$ we provide $|\sin \psi_1|,|\sin \psi_2|,|\sin 2\psi|\leq \frac{1}{2}$ ($2\psi,\psi_1,\psi_2\in (-\frac{\pi}{3},\frac{\pi}{3})$).
Thus $2\psi+\psi_1-\psi_2=\widetilde{\vartheta}+\vartheta-\arg \widetilde{a}-2k\pi\in (-\frac{\pi}{2},\frac{\pi}{2})\setminus \{0\}$ (recall $0<\vartheta,\widetilde{\vartheta}<\pi$, $\vartheta\neq\widetilde{\vartheta}$, $\widetilde{a}=1$) and if in addition $s<\min\big\{\big|\frac{\sin (\widetilde{\vartheta}+\vartheta-\arg \widetilde{a})}{4\widetilde{a}}\big|,\sqrt{\frac{|\widetilde{a}\widetilde{d}\sin (\widetilde{\vartheta}+\vartheta-\arg \widetilde{a})|}{16}}
\big\}$ the next inequality fails for any small $\epsilon,\delta$:
\[
0\neq \big|\sin (\widetilde{\vartheta}+\vartheta-\arg \widetilde{a})\big|=\big|\sin (2\psi+\psi_1-\psi_2)\big|\leq 4\sqrt{\tfrac{2s^{3}}{|\widetilde{a}d|}}+4\delta+2(\epsilon+s)+\tfrac{4s^{2}}{|\widetilde{a}\widetilde{d}|}+\tfrac{2\epsilon}{|\widetilde{d}|}.
\]

\item $|xy|,|duv|\geq s-\delta$

We write the first and the last equality of (\ref{eqBF88}) in the form:
\[
2\big|xu\cos (\eta-\phi)\big|\leq \delta, \qquad 2\big|yv\cos (\kappa-\varphi)\big|\leq \delta,
\]
hence either $8|\cos (\eta-\phi)|^{3},8|\cos (\kappa-\varphi)|^{3}\leq \delta$ or $|v|^{3}\leq \delta$ or $|y|^{3}\leq \delta$ or $|u|^{3}\leq \delta$ or $|x|^{3}\leq \delta$ (or more at the same time).

\quad
First, let $8|\cos (\eta-\phi)|^{3},8|\cos (\kappa-\varphi)|^{3}\leq \delta$. Applying (\ref{ocenah}) to the second equation of (\ref{eqBF10}) for $\widetilde{b}=0$ yields
\[
\psi=(\phi+\varphi)-(\eta+\kappa+\vartheta)=(\phi-\eta)+(\varphi-\kappa)-\vartheta, \qquad |\sin \psi|\leq \tfrac{2\epsilon}{s-\delta}.
\]
Using the angle sum formula for sine and the triangle inequality we obtain an inequlity that fails for any small $\epsilon,\delta$:
\[
|\sin \theta|\sqrt{1-(\tfrac{2\epsilon}{s-\delta})^{2}}-\tfrac{2\epsilon}{s-\delta}|\cos \vartheta|\leq \big|\sin (\psi+\vartheta)\big|=\big|\sin ((\phi-\eta)+(\varphi-\kappa))\big|\leq \sqrt[3]{\delta}.
\]

\quad
Next, suppose $|v|^{3}\leq \delta$. The first and the third equation of (\ref{eqBF10}) lead to 
$\big||x|-|\sqrt{d}u|\big|^{2}\leq \big||x|^{2}-|du^{2}|\big|\leq |\widetilde{a}|+\epsilon $ and 
$\big||y|^{2}-|\widetilde{d}|\big|\leq |y^{2}-\widetilde{d}|\leq \epsilon+|d|\sqrt[3]{\delta^{2}} $.
Combining these facts with the second equation of (\ref{eqBF10}) for $\widetilde{b}=0$ gives
\begin{align*}
\epsilon & \geq |xy+duv|
\geq \big(|\sqrt{d} u|-\sqrt{|\widetilde{a}|+\epsilon}\big)|y|-|duv|
 =|\sqrt{d}u|\big(|y|-|\sqrt{d}v|\big)-|y|\sqrt{|\widetilde{a}|+\epsilon}\\
& \geq |\sqrt{d}u|\big(\sqrt{|\widetilde{d}|-\epsilon-|d|\sqrt[3]{\delta^{2}}}-|\sqrt{d}\sqrt[3]{\delta^{2}}|\big)-\sqrt{|\widetilde{d}|+\epsilon+|d|\sqrt[3]{\delta^{2}}}\sqrt{|\widetilde{a}|+\epsilon}
\end{align*}
For $v\neq 0$ (hence $|u|=\frac{|duv|}{|dv|}\geq \frac{s-\delta}{|d|\sqrt[4]{\delta}}$) we get a contradiction for any small $\epsilon,\delta$. 

\quad
If $v=0$, then (\ref{eqBF10}) for $\widetilde{b}=0$, $\widetilde{}$ yields 
$y^{2}=\widetilde{d}+\epsilon_4$, $xy=\epsilon_2$ (hence $x^{2}=\frac{\epsilon_2^{2}}{\widetilde{d}+\epsilon_4}$), $du^{2}=\widetilde{a}+\epsilon_1-\frac{\epsilon_2^{2}}{\widetilde{d}+\epsilon_4}$. 
Combining this with the squared second equation of (\ref{eqBF88}) gives the equality which fails to hold for any $\epsilon,\delta$ sufficiently small ($\widetilde{\vartheta}\neq \vartheta$, $|d|=|\widetilde{d}|$):
\[
1+2(-1)^{k}\delta_2+\delta_2^{2}=\big ( (-1)^{k} +\delta_2\big)^{2}=\overline{u}^{2}y^{2}=e^{i(\widetilde{\vartheta}-\vartheta)}+e^{i(\widetilde{\vartheta}-\vartheta)}(\epsilon_1-\tfrac{\epsilon_2^{2}}{\widetilde{d}+\epsilon_4})+\tfrac{\epsilon_4}{d}.
\]

\quad
The case $|y|^{3}\leq \delta$ is due to a symmetry treated similarly, with $x,y,du,dv,\widetilde{d},\widetilde{a}$ replaced by $du,dv,x,y,\widetilde{a},\widetilde{d}$; the cases when $|u|^{3}\leq \delta$ or $|x|^{3}\leq \delta$ are dealth likewise.

\end{enumerate}

\item 
$
B=1\oplus d
$, $\Ima d> 0$, \quad
$\widetilde{B}=\begin{bsmallmatrix}
0 & \widetilde{b}\\
\widetilde{b} & 1
\end{bsmallmatrix}
$, $\widetilde{b}>0$\\
In this case it is convenient to conjugate $(A,B)$
with
$\frac{1}{2}
\begin{bsmallmatrix}
2 & 2\\
1 & -1
\end{bsmallmatrix}
$ and consider
\[
\big(\begin{bsmallmatrix}
0 & 1\\
1 & 0
\end{bsmallmatrix},
\begin{bsmallmatrix}
0 & \widetilde{b}\\
\widetilde{b} & 1
\end{bsmallmatrix}\big)\dashrightarrow
\big(\begin{bsmallmatrix}
1 & 0\\
0 & -1
\end{bsmallmatrix},
\tfrac{1}{4}\begin{bsmallmatrix}
4+d & 4-d\\
4-d & 4+d
\end{bsmallmatrix}\big)
\approx
\big(\begin{bsmallmatrix}
0 & 1\\
1 & 0
\end{bsmallmatrix},
\begin{bsmallmatrix}
1 & 0\\
0 & d
\end{bsmallmatrix}\big).
\]
%
%

By setting $y,v$ as in (\ref{polarL}) we conclude from the first equation of (\ref{eqBF1}) that
\small
\begin{align}\label{eq4dux}
\epsilon_1& =\tfrac{1}{4}(4+d)x^2+\tfrac{1}{2}(4-d)xu+\tfrac{1}{4}(4+d)u^2\nonumber\\
& =\tfrac{4+d}{4}\big((|x|^{2}-|u|^{2})e^{2i\phi}+|u|^{2}(e^{2i\phi}+e^{2i\eta})+\tfrac{32-2|d|^{2}-16i\Ima (d)}{|4+d|^{2}}|xu|e^{i(\phi+\eta)}\big)\\
   & =\tfrac{4+d}{4}\big(|x|^{2}-|u|^{2}\big)e^{2i\phi}+\tfrac{4+d}{4}e^{i(\phi+\eta)}\big(|u|^{2}\cos (\phi-\eta)+\tfrac{32-2|d|^{2}}{|4+d|^{2}}|xu|-i\tfrac{16\Ima (d)}{|4+d|^{2}}|xu|\big).\nonumber
\end{align}
\normalsize

Observe that the first summand on the right-hand side is smaller than $\delta$ (as in \ref{p0110v1-1} we have (\ref{est12})), thus we get the  estimate $\frac{16|\Ima (d)|}{|4+d|^{2}}|xu|\leq \frac{4\epsilon }{|4+d|} +\delta $. Combining it with the first equality of (\ref{est12}) we deduce that $|x|^{2},|u|^{2}\leq \frac{|4+d|^{2}}{16|\Ima (d)|} (\frac{4\epsilon }{|4+d|} +\delta )+\delta $.

\quad
Similarly as in (\ref{eq4dux}) we write the third equation of (\ref{eqBF1}) in the form
\small
\begin{align*}
1+\epsilon_4 =\tfrac{4+d}{4}\big(|y|^{2}-|v|^{2}\big)e^{2i\varphi)}+\tfrac{4+d}{4}e^{i(\varphi+\kappa)}\big(|v|^{2}\cos (\varphi-\kappa)+\tfrac{32-2|d|^{2}}{|4+d|^{2}}|vy|-i\tfrac{16\Ima (d)}{|4+d|^{2}}|vy|\big)
\end{align*}
\normalsize
and using the first and the last equality of (\ref{est12}) we deduce 
$\frac{16|\Ima (d)|}{|4+d|^{2}}|yv|\leq \frac{4+4\epsilon }{|4+d|} +\delta $, $|y|^{2},|v|^{2}\leq \frac{|4+d|^{2}}{16|\Ima (d)|} (\frac{4+4\epsilon }{|4+d|} +\delta )+\delta $. Applying the triangle inequality to the squared second equality of (\ref{est12}) now leads to an inequality that fails for any small $\epsilon,\delta$:
\[
(1-\delta)^{2}\leq |\overline{x}v+\overline{u}y|^{2}\leq 4\big(\tfrac{|4+d|^{2}}{16|\Ima (d)|} (\tfrac{4+4\epsilon }{|4+d|} +\delta )+\delta\big)\big(\tfrac{|4+d|^{2}}{16|\Ima (d)|} (\tfrac{4\epsilon }{|4+d|} +\delta )+\delta \big).
\]

\item 
$B=\begin{bsmallmatrix}
0 & b\\
b & 1
\end{bsmallmatrix}
$, $b>0$,\quad 
$\widetilde{B}=1\oplus \widetilde{d}
$, $\widetilde{d}=|\widetilde{d}|e^{i\widetilde{\vartheta}}$, $\pi> \widetilde{\vartheta}> 0$, \qquad $|\widetilde{d}|=b^{2}$ (see (\ref{pogoj}))
\\
We have equations (\ref{eqBF5}) for $\widetilde{a}=1$, $\widetilde{b}=0$ and (\ref{vdivu}) is valid for $\widetilde{b}=0$, $\widetilde{a}=1$, $u,v\neq 0$, thus $\frac{v}{u}=\frac{b\det P+\epsilon_2}{1+\epsilon_4}$. (If $u=0$ ($v=0$) the first (the last) of equations (\ref{eqBF5}) fails for $\epsilon<1$ ($\epsilon<|\widetilde{d}|$).)
Further, from (\ref{calcxov}), (\ref{eqBF88}) we obtain $(-1)^{k}+\overline{\delta}_2=e^{-2i\kappa}\det P +\delta_4 \tfrac{u}{v}$. Similarly, with $x,y,u,v,\varphi,\kappa$ replaced by $y,x,v,u,\phi,\eta$ in (\ref{calcxov}) we get $(-1)^{k}+\overline{\delta}_2=e^{-2i\eta}\det P +\delta_1 \tfrac{v}{u}$. Since $|\frac{\det \widetilde{B}}{\det B}|=|\frac{\widetilde{a}\widetilde{d}}{-b^{2}}|=1$, then Lemma \ref{lemadet} (\ref{lemadetb}) yields $\det P=(-1)^{l}ie^{i\frac{\widetilde{\vartheta}}{2}}+\delta'$, $l\in \mathbb{Z}$, $|\delta'|\leq \frac{\epsilon(4|\widetilde{d}|+2)}{b^{2}}$.
We now collect everything together to deduce:
\[
\frac{|v|^{2}}{|u|^{2}}=\frac{v^{2}}{u^{2}}\frac{(-1)^{k}e^{2i\eta}}{(-1)^{k}e^{2i\kappa}}=\tfrac{(b\det P+\epsilon_2)^{2}}{(1+\epsilon_4)^{2}}
\cdot
\tfrac{\det P +e^{2i\eta}\delta_1 \tfrac{v}{u}-\overline{\delta}_2e^{2i\eta}}{\det P +e^{2i\kappa}\delta_4 \tfrac{u}{v}-e^{2i\kappa}\overline{\delta}_2}=\tfrac{b^{2}(-1)^{3l}i^{3}e^{i\frac{3\widetilde{\vartheta}}{2}}+\widetilde{\epsilon}'}{(-1)^{l}ie^{i\frac{\widetilde{\vartheta}}{2}}+\widetilde{\epsilon}''}=\tfrac{-b^{2}e^{i\widetilde{\vartheta}}+\epsilon'}{1+\epsilon''},
\]
$|\widetilde{\epsilon}'|,|\widetilde{\epsilon}''|,|\epsilon'|,|\epsilon''|\leq K \max \{\epsilon,\delta\} $, where a constant $K>0$ can be obtained by a straightforward computation. If $\epsilon, \delta$ are small enough it contradicts $\pi> \widetilde{\vartheta}> 0$.


\end{enumerate}

\item \label{p0010}
$
(0_2,
1\oplus \sigma)\dashrightarrow (A,B)
$

\begin{enumerate}[label=(\alph*),wide=0pt,itemindent=2em,itemsep=6pt]
\item $\sigma=0$

\begin{enumerate}[label=(\roman*),wide=0pt,itemindent=4em,itemsep=3pt]
\item \label{p0010t}
$A=1\oplus e^{i\theta}
$, \quad $0\leq \theta\leq \pi$

For $\theta \neq \pi$ we have $\|P\|^2\leq \delta$ (Lemma \ref{lemapsi1} (\ref{lemapsi11}) for (C\ref{r3}) with $\alpha=0$), hence $1=\|\widetilde{B}\|=\|P^{T}BP-F\|\leq 4\delta \|B\|+\epsilon $ fails for $\epsilon=\delta<\frac{1}{4\|B\|+1}$.
If $\theta=\pi$ then $P(s)=\frac{1}{\sqrt{a+d+2b}}\begin{bsmallmatrix}
1 & 0\\
1 & s
\end{bsmallmatrix}$, $c(s)=1$ in (\ref{cPepsi}) proves  
$(0_2,1\oplus 0)\to (1\oplus -1,B)$ for $B=\begin{bsmallmatrix}
0 & b\\
b & 0
\end{bsmallmatrix}$, $b>0$ and $B=a\oplus d, d> 0$.

\item \label{p00101i}
$A=\begin{bsmallmatrix}
0 & 1\\
1 & i
\end{bsmallmatrix}$\\
From Lemma \ref{lemapsi1} (\ref{lemapsi11}) with (C\ref{r10}) for $\alpha=0$ we get $|u|^{2}\leq \delta$. Further, the first equation of (\ref{eqBF1}) for $a=b=0$, $\widetilde{a}=1$ is $du^{2}=1+\epsilon_1$ and fails for $\delta=\epsilon<\frac{1}{1+|d|}$.
Taking $P(s)=\frac{1+i}{2\sqrt{b}}
\begin{bsmallmatrix}
s^{-1} & s^{2}\\
-is & s^{2}
\end{bsmallmatrix}
$ and
$P(s)=\frac{1}{\sqrt{a}}\oplus s
$ with $c(s)=1$ in (\ref{cPepsi}), we prove 
$
(0_2,
1\oplus 0)\to
\big(\begin{bsmallmatrix}
0 & 1\\
1 & i
\end{bsmallmatrix},
\begin{bsmallmatrix}
a & b\\
b & d
\end{bsmallmatrix}\big)
$ for $a=d=0,b>0$ and $a>0$, $b=0$, $ d\in \mathbb{C}$, respectively.

\item 
$A=\begin{bsmallmatrix}
0 & 1\\
\tau & 0 
\end{bsmallmatrix}
$, $0\leq \tau\leq 1$\\
When $a\neq 0$ and $d\neq 0$ we can take 
$P(s)=\frac{1}{\sqrt{a}}\oplus s
$
and 
$P(s)=\begin{bsmallmatrix}
0 & s \\
\frac{1}{\sqrt{d}} & 0
\end{bsmallmatrix}$ with $c(s)=1$ in (\ref{cPepsi}), respectively, to prove 
$
(0_2,
1\oplus 0)\to
\big(\begin{bsmallmatrix}
0 & 1\\
\tau & 0 
\end{bsmallmatrix},
\begin{bsmallmatrix}
a & b\\
b & d
\end{bsmallmatrix}\big)
$. 
Next, let $a=d=0$ (hence $0\leq \tau<1$ by Lemma \ref{lemalist}). From Lemma \ref{lemapsi1} (\ref{lemapsi11}) for (C\ref{r4}) with $\alpha=0$ we get $|ux|\leq \delta$, so the first equation of (\ref{eqBF1}) for $a=d=0$, $\widetilde{a}=1$ 
fails for $\epsilon=\delta\leq \frac{1}{1+2|b|}$.

\item 
$A=1\oplus 0
$\\
%
To prove $
(0_2,
1\oplus 0)\to  
(1\oplus 0,
a \oplus 1)
$, $a\geq 0$, and  
$
(0_2,
1\oplus 0)\to \big(
1\oplus 0,
\begin{bsmallmatrix}
0 & 1\\
1 & 0
\end{bsmallmatrix}\big)
$
we can take
$P(\epsilon)=\begin{bsmallmatrix}
s & s \\
1 & s
\end{bsmallmatrix}$
and 
$P(s)=\frac{1}{\sqrt{2}}\begin{bsmallmatrix}
s & s^{2} \\
s^{-1} & s^{2}
\end{bsmallmatrix}$, both with $c(s)=1$ in (\ref{cPepsi}), respectively.
From Lemma \ref{lemapsi1} (\ref{lemapsi11}) with (C\ref{r11}) for $\alpha=0$ we have $|x|^{2}\leq \delta$. When $\widetilde{B}=\widetilde{a}\oplus 0$, $\widetilde{a}>0$ the first equation of (\ref{eqBF1}) for $a=b=d=0$ (i.e. $0=\widetilde{a}+\epsilon_2$) fails for $\epsilon<\widetilde{a}$. 

\end{enumerate}

\item \label{p0011}
$\sigma=1$

\begin{enumerate}[label=(\roman*),wide=0pt,itemindent=4em,itemsep=3pt]
\item 
$A=1\oplus 0$\\
We prove $(0_2,I_2)\to
\big(1\oplus 0,\begin{bsmallmatrix}
0 & 1\\
1 & 0
\end{bsmallmatrix}\big)
$ by taking 
$P(\epsilon)=\frac{1}{\sqrt{2}}\begin{bsmallmatrix}
s & i s \\
s^{-1} & -i s^{-1}
\end{bsmallmatrix}$ with $c(s)=1$ in (\ref{cPepsi}).
For $B=a\oplus d$, $a,d\geq 0$ we have $|x|^{2},|y|^{2}\leq \delta$ (Lemma \ref{lemapsi1} (\ref{lemapsi11}) with (C\ref{r11}) for $\alpha=0$). Combining it with (\ref{eqBFadad}) for $\widetilde{a}=\widetilde{d}=1$ yields $|du|^{2},|dv|^{2}\geq 1-\epsilon-a\delta$ and $|duv|\leq \epsilon+a\delta$. Thus we obtain a contradiction for $\epsilon=\delta<\frac{1}{2(a+1)}$. 

\item 
$A=\begin{bsmallmatrix}
0 & 1\\
0 & 0 
\end{bsmallmatrix}$\\
By Lemma \ref{lemapsi1} (\ref{lemapsi11}) for (C\ref{r4}) with $\tau=0$ we have $|ux|, |uy|, |vx|,|vy|\leq \delta$.
%
It implies that either $|x|^{2},|y|^{2}\leq \delta$ or $|u|^{2},|v|^{2}\leq \delta$ (or both). If $|x|^{2},|y|^{2}\leq \delta$ then the first and the last equation of (\ref{eqBF1}) for $\widetilde{b}=0$, $\widetilde{a}=\widetilde{d}=1$ give $|d||u|^{2},|d||v|^{2}\geq 1-2b\delta-|a|\delta-\epsilon$, and the  the application of the triangle inequality to second equation further gives
\[
\epsilon+2b\delta \geq |axy+duv|\geq \big| |duv|-|axy| \big|\geq (1-2b\delta-|a|\delta-\epsilon)-|a|\delta,
\]
which fails for $\epsilon=\delta<\frac{1}{2+4b+2|a|}$.
The case $|u|^{2},|v|^{2}\leq \delta$ is for the sake of symmetry treated in a similar fashion, with $u,v,x,y,a,d$ replaced by $x,y,u,v,d,a$, respectively.
\end{enumerate}

\end{enumerate}

\item \label{u1000}
$
(1\oplus 0,
\widetilde{B})\to 
(1\oplus 0,
B)
$\\
From Lemma \ref{lemapsi1} (\ref{lemapsi11}) with (C\ref{r11}) for $\alpha=1$ we get that 
\begin{equation}\label{est13}
\big||x|^2-1\big|\leq \delta ,\qquad |y|^{2}\leq \delta.
\end{equation}

\begin{enumerate}[label=(\alph*),wide=0pt,itemindent=2em,itemsep=6pt]

\item 
$B=\begin{bsmallmatrix}
0 & 1\\
1 & 0
\end{bsmallmatrix}$ \qquad (hence 
$\widetilde{B}=
\widetilde{a}\oplus \widetilde{d}
$, 
$\widetilde{a}\geq 0$, $\widetilde{d}\in \{0,1\}$ by Lemma \ref{lemalist} we have)
The first equation of (\ref{eqadb}) for $b=1$ and $|x|^{2}\geq 1-\delta\geq \frac{1}{2}$ (see (\ref{est13})) yield $|u|\leq \frac{\widetilde{a}+\epsilon}{\sqrt{2}}$,
while the second equation (multiplied with $y$) and the third equation give $\frac{1}{2}x(\widetilde{d}+\epsilon_4)+buy^{2}=bxvy+buy^{2}=y\epsilon_2$.
Using the upper (the lower) estimates on $|u|$, $|y|$ (on $|x|$) we  get an inequality that fails for $\widetilde{d}\neq 0$ and small enough $\epsilon,\delta$:
\[
\tfrac{\widetilde{d}-\epsilon}{2\sqrt{2}}-b\tfrac{\widetilde{a}+\epsilon}{\sqrt{2}}\delta \leq 
x\tfrac{\widetilde{d}+\epsilon_4}{2}+buy^{2}=y\epsilon_2
\leq \sqrt{\delta}\epsilon.
\]
Taking $c(s)=1$, $P(s)=\begin{bsmallmatrix}
\frac{\widetilde{a}}{2} & s \\
1 & 0
\end{bsmallmatrix}$ in (\ref{cPepsi}) proves $
{\ (1\oplus 0,
\widetilde{a}\oplus 0)\to 
\big(1\oplus 0,
\begin{bsmallmatrix}
0 & 1\\
1 & 0
\end{bsmallmatrix}\big)}
$.

\item $B=a\oplus d$, \quad $a\geq 0$, $d\in \{0,1\}$, \qquad $a^{2}+d^{2}\neq 0$ (since $B\neq 0$)\\ 
From (\ref{eqBF1}) for $b=0$ and by slightly rearranging the terms we obtain
\begin{align}\label{eq10111022}
&du^{2}=(\widetilde{a}-a)+\epsilon_1-a(x^{2}-1),\nonumber \\
&duv-\widetilde{b}=\epsilon_2-axy,\\ 
&dv^{2}=\widetilde{d}+\epsilon_4-ay^{2}. \nonumber
\end{align}
Thus using (\ref{est13}) and applying the triangle inequality to (\ref{eq10111022}) we deduce 
\begin{align}\label{est14}
&\big||du|^{2}- |\widetilde{a}-a|\big|\leq a\delta+\epsilon,\nonumber\\
&\big||duv|-|\widetilde{b}|\big| \leq \epsilon+a\sqrt{\delta(1+\delta)},\\
&\big||dv|^{2}- \widetilde{d}| \leq \epsilon+a\delta.\nonumber
\end{align}

\begin{enumerate}[label=(\roman*),wide=0pt,itemindent=4em,itemsep=3pt]

\item $\widetilde{B}=\widetilde{a}\oplus \widetilde{d}$, \quad $\widetilde{a}\geq 0$, $\widetilde{d}\in \{0,1\}$\\
Combining all the estimates in (\ref{est14}) for $\widetilde{b}=0$, $\widetilde{d}=1$ yields
\[
\big(\epsilon+a\sqrt{\delta(1+\delta)}\big)^{2}\geq |duv|^{2}\geq \big(|\widetilde{a}-a|-a\delta-\epsilon\big)\big(1-\epsilon-a\delta\big),
\]
which fails for $a\neq \widetilde{a}$ and appropriately small $\epsilon,\delta$. (Note that $\widetilde{d}=1$ implies $a\neq \widetilde{a}$ (see (\ref{pogoj}),  Lemma \ref{lemalist}) for $\widetilde{B}\neq B$.)
Next, when $B=a\oplus 0$, $\widetilde{B}=\widetilde{a}\oplus 0$, $a\neq \widetilde{a}$, the first inequality of (\ref{est14}) for $d=0$ fails for $\epsilon=\delta<\frac{|\widetilde{a}-a|}{1+a}$.
Finally, 
$c(s)=1$, 
$P(s)=\begin{bsmallmatrix}
1 & 0 \\
\sqrt{\widetilde{a}-a} & s
\end{bsmallmatrix}$ in (\ref{cPepsi})
proves 
$(1\oplus 0,\widetilde{a}\oplus 0)\to(1\oplus 0,a\oplus 1)$.

\item $\widetilde{B}=\begin{bsmallmatrix}
0 & 1 \\
1 & 0
\end{bsmallmatrix}$\\
We have $B=a\oplus 1$, $a>0$. The inequalities in (\ref{est14}) for $\widetilde{a}=\widetilde{d}=0$, $\widetilde{b}=1$ then give
\[
\big(1-\epsilon-a\sqrt{\delta(1+\delta)}\big)^{2}\leq |duv|^{2}\leq \big(a+a\delta+\epsilon\big)\big(\epsilon+a\delta\big),
\]
which fails for $\epsilon,\delta$ small enough.
\end{enumerate}

\end{enumerate}

\item \label{case10to1t}
$(1\oplus 0,
\widetilde{a}\oplus \widetilde{d})\dashrightarrow
(1\oplus e^{i\theta},B)$, \;$0\leq \theta \leq \pi$, $\widetilde{d}\in \{0,1\}$, $\widetilde{a}\geq 0$,\quad $\widetilde{a}\widetilde{d}=0$ (see (\ref{pogoj}))

\begin{enumerate}[label=(\alph*),wide=0pt,itemindent=2em,itemsep=6pt]

\item \label{case10to1ta} $0\leq \theta < \pi$\\
By Lemma \ref{lemapsi1} (\ref{lemapsi11}) for (C\ref{r3}) we have 
$|y|^{2},|v|^{2}\leq \delta$.
If $\widetilde{d}=1$, then the third equation of (\ref{eqBF1}) fails for $ \epsilon=\delta<\big(|a|+2|b|+|d|+1\big)^{-1}$.

\quad
Next, let $\widetilde{d}=0$. To prove (disprove) the existence of a path we need to solve (to see that there are no solutions) the equations (\ref{eqABEF}) for arbitrarily small $E,F$. Given any small $s>0$ 
we must (not) find $x,y,u,v,c$ and $\epsilon_1,\epsilon_2,\epsilon_4$ with
$|\epsilon_1|,|\epsilon_2|,|\epsilon_4|\leq s$ satisfying (\ref{eqBF1}) and such that the expressions in Lemma \ref{lemapsi1} for (C\ref{r3}) are bounded from above by $s$. Observe that by choosing $v,y$ sufficiently small we achieve that the last two equations in (\ref{eqBF1}) for $\widetilde{b}=\widetilde{d}=0$ are fulfilled trivially for some small $\epsilon_2,\epsilon_4$, and the first two expressions in Lemma \ref{lemapsi1} (C\ref{r3}) are arbitrarily small. It is thus left to consider the remaining first equation of (\ref{eqBF1}) and the third expression of Lemma \ref{lemapsi1} (C\ref{r3}); note that they do not depend on $y,v$:
%
\begin{equation}\label{2Req}
ax^{2}+2bux+du^{2}=\widetilde{a}+\epsilon_1, \quad 
|x|^{2}+e^{i\theta}|u|^{2}-c^{-1}= c^{-1}\delta_1, \qquad |\epsilon_1|,|\delta_1|\leq s.
\end{equation}
The second equation can be rewritten as 
$\big(|x|^{2}+\cos \theta |u|^{2}\big)+i\sin \theta |u|^{2}=c^{-1}(1+\delta_1)$ 
and it is clearly equivalent to compare only the (squared) moduli of both sides \[
\big(|x|^{2}+\cos \theta |u|^{2}\big)^{2}+(\sin^{2}\theta) |u|^{4}=|1+\delta_1|^{2}=|1+2\delta_1+\delta_1^{2}|.
\]
By (\ref{ocenakoren}) for $|\delta_1|\leq \frac{1}{3}$ it further implies $|x|^{4}+2\cos \theta |u x|^{2}+ |u|^{4}=1+\delta_2$ for some $\delta_2\in \mathbb{R}$, $|\delta_2|\leq 3|\delta_1|\leq 3s$.
By writting $x,u$ in view of (\ref{polarL}) we can see (\ref{2Req}) as
\begin{align}\label{2Req2}
& e^{i(\phi+\eta)}\big(a|x|^{2}e^{i(\phi-\eta)}+2b|x|\,|u|+d|u|^{2}e^{-i(\phi-\eta)}\big)=\widetilde{a}+\epsilon_1,\qquad |\epsilon_1|\leq s,\\
& |x|^{4}+2\cos \theta |u|^{2} |x|^{2}+ |u|^{4}=1+\delta_2,\qquad \delta_2\in \mathbb{R}, \quad|\delta_2|\leq 3s. \nonumber
\end{align}
Clearly, (\ref{2Req2}) for $\delta_2=0$ is equivalent to (\ref{2Req}) for $\delta_1=0$.

\quad
Next, we observe the range of the  function $f(r,t,\beta)=|ar^{2}e^{i\beta}+2brt+dt^{2}e^{-i\beta}|$ given with a constraint $r^{4}+2r^{2}t^{2}\cos \theta + t^{4}=1$ (see (\ref{izrazf7})).
%
%
Provided that $(R,T)=(r^{2}$, $t^{2})$ lies on on an ellipse $R^{2}+2RT \cos \theta+ T^{2}=1$, we can further assume that either $\frac{r}{t}$ or $\frac{t}{r}$ is any real nonegative number. With a suitable choice of $\beta$ we achieve finally that the following expression (and hence $f$) vanishes:
\begin{align*}
ar^{2}e^{i\beta}+2brt+dt^{2}e^{-i\beta} & =r^{2}e^{i\beta}\big(a+2b(\tfrac{t}{r}e^{-i\beta})+d(\tfrac{t}{r}e^{-i\beta})^{2}\big)\\
&  =t^{2}e^{-i\beta}\big(a(\tfrac{r}{t}e^{i\beta})^{2}+2b(\tfrac{r}{t}e^{i\beta})+d\big).
\end{align*}
If the maximum of $f$ given with a constraint is $M$, its range is thus $[0,M]$. 

\quad
Provided that $\phi-\eta$, $|x|$, $|u|$ (corresponding to $\beta$, $r$, $t$ in (\ref{izrazf7})) are chosen appropriately, the modulus of the left-hand side of the first equation in (\ref{2Req2}) (hence $|ax^{2}+2bux+du^{2}|$) can be any number from the interval $[0,M]$ and the second equation of (\ref{2Req2}) for $\delta_2=0$ (thus $|x|^{2}+e^{i\theta}|u|^{2}-c^{-1}= 0$) is valid, simultaneously.
By a suitable choice of $\phi+\eta$ (see (\ref{2Req2})) we arrange $ax^{2}+2bux+du^{2}\in \mathbb{R}_{\geq 0}$ and so $(1\oplus 0,
\widetilde{a}\oplus 0)\to
(1\oplus e^{i\theta},B)$ for 
$\widetilde{a}\in[ 0,M]$ is proved. In particular, for $a=b=0$ we get $M=|d|$, for $d=b=0$ we obtain $M=|a|$, and for $a=d=0$ we have $M=2|b|\max \{rt\mid r^{4}+2(\cos \theta) r^{2}t^{2}+ t^{4}=1\}$. Moreover, if $\theta=0$, $0\leq a\leq d$, $d>0$ we get $M=d$, and to prove the existence of a path we take 
$P(s)=\frac{1}{\sqrt{a+d}}\begin{bsmallmatrix}
\sqrt{\widetilde{a}+d} & 0 \\
i\sqrt{a-\widetilde{a}} & s
\end{bsmallmatrix}$ for $0\leq \widetilde{a}\leq a \leq d$
and 
$P(s)=\frac{1}{\sqrt{d-a}}\begin{bsmallmatrix}
\sqrt{d-\widetilde{a}} & 0 \\
\sqrt{\widetilde{a}-a} & s
\end{bsmallmatrix}$ for $0\leq a<\widetilde{a}\leq d$, both with $c(s)=1$ in (\ref{cPepsi}).

\quad
Furthermore 
\begin{align*}
\big||f(r,t,\beta)|-|f(r',t',\beta')|\big| &\leq \big|f(r,t,\beta)-f(r',t',\beta')\big|\\
    &\leq |a|\,|r^{2}-(r')^{2}|+2|b|\,|r|\,|t-t'|+2|b|\,|t|\,|r-r'|+|d|\,|t^{2}-(t')^{2}|.
\end{align*}
Applying (\ref{ocenakoren}) to $1+\delta_2'$ and its root yields $\sqrt[4]{1+\delta_2'}=1+\delta_3$, $|\delta_3|\leq |\delta_2'|$. For $r'=r\sqrt[4]{1+\delta_2'}$, $t'=t\sqrt[4]{1+\delta_2'}$, $\delta_2'\in \mathbb{R}$, $|\delta_2'|\leq |\delta_2|\leq 3s$ with  $|r|,|t|\leq 1$ (note that $r^{4}+2\cos \theta r^{2}t^{2}+ t^{4}=1$), we then deduce
\[
\big||f(r,t,\beta)|-|f(r',t',\beta')|\big|
\leq |1-\sqrt[4]{1+\delta_1}|\big(|a|r^{2}+4|b|rt+|d|t^{2}\big)\leq 3s \big(|a|+4|b|+|d|\big).
\]
Since $M$ is the maximum of $f$ with respect to a constraint  (\ref{izrazf7}), 
it follows that the maximum of $f$ on a compact domain given by $|r^{4}+2\cos \theta r^{2}t^{2}+t^{4}-1|\leq |\delta_2|\leq 3s$ 
is at most $M_1=M+ 3s\big(|a|+2|b|+|d|\big)$. Assuming that the second equation of (\ref{2Req2}) is valid, the modulus of the left-hand side of the first equation of (\ref{2Req2}) is then at most $M_1$; if $\widetilde{a}> M$, this equation fails for $s<\frac{\widetilde{a}-M}{1+3|a|+6|b|+3|d|)}$.

\item $\theta=\pi$
\begin{enumerate}[label=(\roman*),wide=0pt,itemindent=4em,itemsep=3pt]

\item 
$
B=a\oplus d
$, \quad $0\leq a\leq d$, $d\neq 0$ (see Lemma \ref{lemalist} , (\ref{pogoj}); $B\neq 0$)\\
From Lemma \ref{lemapsi1} (\ref{lemapsi11}) for (C\ref{r9}) for $\sigma=-1$, $\alpha=1$, $\omega=0$ we get
\begin{equation}\label{est17}
\big||x|^2-|u|^2-(-1)^{k}\big|\leq \delta, \quad |\overline{x}y-\overline{u} v|\leq \delta, \quad \big| |y|^2-|v|^2 \big|\leq \delta, \qquad k\in \mathbb{Z}.
\end{equation}

\quad
As in \ref{p10011001} \ref{p1t1tbi} we obtain $|a-d|\,|xy|\leq d\delta+\epsilon$ (see (\ref{caseIb1})).
Further, from the first equation of (\ref{eqBFadad}) for $\widetilde{a}=0$ we get $a|x|^{2}+d|u|^{2}\geq |ax^{2}+du^{2}|\geq \epsilon$, and by combining it with the first estimate of (\ref{est17}) we deduce
$|a+d|\,|x|^{2}\geq d(1-\delta)+\epsilon$. 
Moreover,
\[
|a-d|^{2}\tfrac{d(1-\delta)+\epsilon}{a+d}|y|^{2}\leq |a-d|^{2}|xy|^{2}\leq (d\delta+\epsilon)^{2}.
\]
When $a\neq d$ we have
$|y|^{2},|v|^{2}\leq \tfrac{(a+d)(d\delta+\epsilon)^{2}}{|a-d|^{2}(d(1-\delta)+\epsilon)}+\delta
$ (see the last estimate in (\ref{est17})), and if further $\widetilde{d}=1$ the third equation of (\ref{eqBFadad}) gives a contradiction for any $\epsilon,\delta$ small enough. 
Next, for $a=d$ and $\widetilde{d}=1$ (hence $\widetilde{a}=0$) the first equation of (\ref{eqBFadad}) and the first estimate in (\ref{est17}) yield
\[
\epsilon \geq d\big| x^{2}+u^{2}\big|\geq d\big||x|^{2}-|u|^{2}\big|\geq d(1-\delta),
\]
which fails for $\epsilon=\delta<\frac{d}{1+d}$. The existence of a path for $\widetilde{d}= 0$ follows from (\ref{cPepsi}) for
\[
P(s)=\left\{
\begin{array}{ll}
\frac{1}{\sqrt{d-a}}
\begin{bsmallmatrix}
\sqrt{d-\widetilde{a}} & 0 \\
i\sqrt{a-\widetilde{a}} & s
\end{bsmallmatrix},  &  0\leq a\leq d, a\geq \widetilde{a} \\
\frac{1}{\sqrt{d+a}}\begin{bsmallmatrix}
\sqrt{d+\widetilde{a}} & 0 \\
\sqrt{\widetilde{a}-a} & s
\end{bsmallmatrix},  &  0\leq a\leq d,  \widetilde{a}\geq a
\end{array}\right., \qquad c(s)=1.
\]

\item
$
B=
\begin{bsmallmatrix}
0 & b\\
b & 0 
\end{bsmallmatrix}
$, $b>0$\\
We take $u^{2}=\frac{-b+ \sqrt{b^{2}+\widetilde{a}^{2}}}{2b}$, $x^{2}=\frac{b+ \sqrt{b^{2}+\widetilde{a}^{2}}}{2b}$, $y=v=s$, $c(s)=1$ to satisfy (\ref{cPepsi}).

\end{enumerate}

\end{enumerate}

\item \label{g100110}
$
(1\oplus 0,
\widetilde{a}\oplus \widetilde{d})\to 
\big(\begin{bsmallmatrix}
0 & 1\\
1 & 0 
\end{bsmallmatrix},
B\big)
$,\quad $\widetilde{d}\in \{0,1\}$, $\widetilde{a} \widetilde{d}=0$ (see Lemma \ref{lemalist} and (\ref{pogoj}))


Lemma \ref{lemapsi1} (\ref{lemapsi11}) for (C\ref{r7}) for $\alpha=1$, $\beta=\omega=0$ yields
\begin{equation}\label{est144}
\big|2\Rea(\overline{x}u)-(-1)^{k}\big|\leq \delta,\quad
\big|2\Rea(\overline{y}v)\big|\leq\delta, \quad  |x\overline{v}+\overline{u}y|\leq \delta, \quad k\in \mathbb{Z}.
\end{equation}

\begin{enumerate}[label=(\alph*),wide=0pt,itemindent=2em,itemsep=6pt]
\item \label{g100110a}
$B=
1\oplus d
$, \quad $\Ima (d)>0$

\begin{enumerate}[label=(\roman*),wide=0pt,itemindent=4em,itemsep=3pt]
\item $\widetilde{B}=0\oplus 1$\quad ($\widetilde{a}=0$, $\widetilde{d}=1$)\\ 
From the first estimate in (\ref{est144}) and the first equation of (\ref{eqBF10}) for $\widetilde{a}=0$ we deduce that $|du|^2,|x|^2\geq \frac{\min\{1,|d|\}}{4}-\epsilon$ (see \ref{p1-1v0110} \ref{p1-1v0110c} \ref{p1-1v0110ci} (\ref{ocenaux4}) for the same conclusion).
The last equation of (\ref{eqBF10}) for $\widetilde{d}=1$ further yields that either $|y|^2\geq \frac{1-\epsilon}{2}\geq \frac{1}{4}$ or $|v|^2\geq \frac{1-\epsilon}{2|d|}\geq \frac{1}{4|d|}$. By combining these facts we get that either $|xv|^2$ or $|uy|^2$ is at least $\frac{1}{4|d|}\big(\frac{\min\{1,|d|\}}{4}-\epsilon\big)$.
Similarly, either $|xy|^2$ or $|duv|^2$ is greater or equal to $\frac{1}{4}(\frac{\min\{1,|d|\}}{4}-\epsilon)$.
As in \ref{p1-1v0110} \ref{p1-1v0110c} \ref{p1-1v0110cii} we obtain precisely (\ref{ocenakot1}), (\ref{ocenakot2}), but now with estimates 
\[
|\sin \psi_1|\leq \tfrac{1}{4|d|}(\tfrac{\min\{1,|d|\}}{4}-\epsilon), \qquad |\sin \psi_1|\leq \tfrac{1}{4}(\tfrac{\min\{1,|d|\}}{4}-\epsilon),
\]
which give a contradiction for any suitably small $\epsilon,\delta$ as well.
   
%

\item \label{g100110ai}
$\widetilde{B}=\widetilde{a}\oplus 0$\\ 
To prove the existence of a path for any $\widetilde{a}\geq 0$, then given an arbitrary $s>0$ it is sufficient to solve the system of equations $ax^2+du^{2}=\widetilde{a}+\epsilon_1$ and $2\Rea (\overline{x}u)=-1+\delta_1$, where $|\epsilon_1|,|\delta_1|\leq s$. Indeed, provided that $y,v$ are chosen small enough, the last two equations in (\ref{eqBF10}) for $\widetilde{b}=0$ are satisfied for some $|\epsilon_2|,|\epsilon_4|\leq s$, and the absolute values of expressions in Lemma \ref{lemapsi1} (C\ref{r7}) for $k=0$ are bounded by $s$ (it implies $\|E\|\leq\rho \sqrt{s}$ for some constant $\rho>0$ independent of $s$; see also (\ref{eqABEF})). 

\quad
In view of the notation (\ref{polarL}) with $d=|d|e^{i\vartheta}$, $0<\vartheta<\pi$ we obtain  
\begin{equation}\label{izraz1}
e^{i(\eta+\phi)}\big(|x|^{2}e^{i(\phi-\eta)}+|d|\,|u|^{2}e^{-i(\phi-\eta)+i\vartheta}\big)=\widetilde{a}+\epsilon_1, \qquad 2|x|\,|u| \cos (\eta-\phi)=1+\delta_1,
\end{equation}
where $|\epsilon_1|,|\delta_1|\leq s$. We now observe the range of the following functions: 
\begin{align}\label{izrazf1}
& f(r,t,\beta)=\big|r^{2}e^{i\beta}+|d|t^{2}e^{-i\beta+i\vartheta}\big|, \qquad  2rt\cos \beta=-1,\quad r,t\geq 0,\,\, \beta\in \mathbb{R},\\
%
& g(r)=f\big(r,\tfrac{1}{2r\cos (\frac{\vartheta}{2}+\frac{\pi}{2})},\tfrac{\vartheta}{2}+\tfrac{\pi}{2}\big)=\big|ir^{2}e^{i\frac{\vartheta}{2}}-\tfrac{i|d|}{4r^{2}\sin^{2} (\frac{\vartheta}{2})}e^{i\frac{\vartheta}{2}}\big|.\nonumber
\end{align}
Since $g\to \infty$ as $r\to\infty$ and $g(r)\to 0$ as  $r\to\frac{4\sqrt{|d|}}{\sin \frac{\Delta}{2}}$, $d\ne 0$ or $r\to 0$, $d=0$, the range of $g$ (and hence $f$ with a constraint) is $(0,\infty)$. Thus the modulus of the left-hand side of the first equation in (\ref{izraz1}) 
can be any number from the interval $(0,\infty)$ and the second equation in (\ref{izraz1}) is valid at the same time, provided that $\phi-\eta$, $|u|$, $|x|$ (corresponding to $\beta,t,r$ in (\ref{izrazf1})) are chosen appropriately.
Finally, a suitable choice of $\phi+\eta$ guarantees that the left-hand side of (\ref{izraz1}) is real and positive. Thus (\ref{izraz1}) is solvable fo any $\widetilde{a}\geq 0$ (and any $d$, $\Ima d>0$).
\end{enumerate}

\item 
$B=\begin{bsmallmatrix}
0 & b\\
b & 1
\end{bsmallmatrix}$, $b>0$\\

\begin{enumerate}[label=(\roman*),wide=0pt,itemindent=4em,itemsep=3pt]

\item
$\widetilde{B}=0\oplus 1$

Let $x,u$ be as in (\ref{polarL}). Since $|u\overline{x}|\geq |\Rea (\overline{x}u)|\geq 1-\delta\geq \frac{1}{2}$ (see (\ref{est144})), then assuming $\epsilon \leq \frac{b}{2}$ and applying (\ref{ocenah}) to the first equation of (\ref{eqBF5}) for $\widetilde{a}=0$ yields
\begin{equation}\label{razlika1}
\tfrac{4\epsilon}{b} \geq \Big|\sin \big ((\phi+\eta)-(2\eta+\pi)+2l\pi\big)\Big| =\big|\sin (\phi-\eta)\big|, \quad l\in \mathbb{Z}.
\end{equation}
%
%
We multiply the first equation of (\ref{eqBF5}) for $\widetilde{a}=0$ with $\frac{v}{u}$ to get
$2bvx=(\frac{\epsilon_1}{u^{2}}-1)uv$. 
Together with the last estimate in (\ref{est144}) and providing $\epsilon\leq b$ we obtain:
\begin{align}\label{vxuv}
&|u|^{2}\geq 2b|ux|-\epsilon\geq b-\epsilon,\quad
|vx|\leq \tfrac{1}{2b}\big(1+\tfrac{\epsilon}{b-\epsilon}\big)|uv|,\\
&|\overline{u}y|\leq \delta+|v\overline{x}|\leq \delta+ \tfrac{1}{2b}(1+\tfrac{\epsilon}{b-\epsilon})|uv|.\nonumber
\end{align}
%
Further, we have
\begin{align*}
vx+uy &=e^{2i\eta}(v\overline{x}e^{2i(\phi-\eta)}+y\overline{u})=e^{2i\eta}\big(v\overline{x}(e^{2i(\phi-\eta)}-1)+(v\overline{x}+y\overline{u})\big)\\
      &=e^{2i\eta}\big(iv\overline{x}\sin (\phi-\eta)e^{i(\phi-\eta)}+(v\overline{x}+y\overline{u})\big).\nonumber
\end{align*}
%
Combining it with the second equation of (\ref{eqBF5}) for $\widetilde{b}=0$ (estimate of (\ref{vxuv})) gives
\[
|uv|=|\epsilon_2-b(vx+yu)|\leq \epsilon+b\delta+ b|vx||\sin (\phi-\eta)|\leq \tfrac{b}{2b-2\epsilon}|uv||\sin (\phi-\eta)|+ b\delta+\epsilon.
\]
After simplifying and using (\ref{razlika1}) we get 
%
\begin{equation*}
|uv|\leq (b\delta+\epsilon)\big( 1-\tfrac{4\epsilon}{b-\epsilon}\big)^{-1}.
\end{equation*}
Multiplying the third equation of (\ref{eqBF5}) for $\widetilde{d}=1$ by $u^{2}$, and rearranging the terms gives $2b(uv)(uy)+(uv)^{2}=(1+\epsilon_4)u^{2}$. Combining this with (\ref{vxuv}) and with the upper estimate on $|uv|$ finally yields an inequality that fails for any $\epsilon,\delta$ small enough.

\item $\widetilde{B}=\widetilde{a}\oplus 0$, $\widetilde{a}\geq 0$\\
To prove the existence of a path we use the same argument as in \ref{g100110} \ref{g100110a} \ref{g100110ai}. For a given $s>0$ it suffices to solve the equations $2bux+u^{2}=\widetilde{a}+\epsilon_1$, $|\epsilon_1|\leq s$ (see (\ref{eqBF5})) and $-\Rea (\overline{x}u)=1+\delta_1$, $|\delta_1|\leq s$ ($k=1$ in Lemma \ref{lemapsi1} (C\ref{r7})).
In view of (\ref{polarL}) they can be written in the form
\begin{equation}\label{izraz2}
e^{i(\eta+\phi)}(2b|ux|+|u|^{2}e^{i(\eta-\phi)})=\widetilde{a}+\epsilon_1, \quad - |ux| \cos (\eta-\phi)=1+\delta_1, \qquad |\epsilon_1|,|\delta_1|\leq s.
\end{equation}
As in the case mentioned before we define the function given with constraint:
\begin{equation*}
f(r,t,\beta)=|2brt+r^{2}e^{i\beta}|, \qquad - rt\cos \beta=1,\quad r,t\geq 0,\quad \beta\in \mathbb{R}.
\end{equation*}
Since $g(r)=f(r,\frac{1}{r},\pi)=|2b-r^{2}|$ with $g(\sqrt{2b})=0$ and $g(r)\to\infty$ as $r\to\infty$, the range of $g$ (and hence $f$) is $[0,\infty)$. Thus the left-hand side of the first equation in (\ref{izraz2}) 
can be any number from $[0,\infty)$ and the second equation in (\ref{izraz2}) stays valid for $\delta_1=0$, simultaneously, provided that $\phi-\eta$, $|u|$, $|x|$ (corresponding to $\beta,t,r$ above), $\phi+\eta$ (makes the left-hand side of (\ref{izraz2}) real and positive) are chosen appropriately. 
\end{enumerate}

\item $B=1\oplus 0$

It is equivalent to consider 
$\big(1\oplus 0,
\widetilde{B}\big)\dashrightarrow
\big(1\oplus -1,
\begin{bsmallmatrix}
1 & 1\\
1 & 1 
\end{bsmallmatrix}\big)\approx
\big( \begin{bsmallmatrix}
0 & 1\\
1 & 0 
\end{bsmallmatrix},
1\oplus 0 \big)
$.
In case $\widetilde{B}=\widetilde{a}\oplus 0$, $\widetilde{a}\geq 0$ we apply (\ref{cPepsi}) for $c(s)=1$ and 
$P(s)=\begin{bsmallmatrix}
\sqrt{p^{2}+1} & 0 \\
-p & s^{2}
\end{bsmallmatrix}$, 
$p=\left\{
\begin{array}{ll}
s^{-1}, & \widetilde{a}=0 \\
\frac{|1-\widetilde{a}|}{2\sqrt{\widetilde{a}}}, & \widetilde{a}>0 
\end{array}
\right.$.

\quad
When $\widetilde{B}=0\oplus 1$ ($\widetilde{d}=1$, $\widetilde{a}=0$) and $a=b=d=1$ we get from (\ref{eqBF1}):
\[
(u+x)^{2}=\epsilon_1, \qquad (u+x)(v+y)=\epsilon_2,\qquad (v+y)^{2}=1+\epsilon_4.
\]
Hence $u+x=(-1)^{l_1'}\sqrt{\epsilon_1}$ and $v+y=(-1)^{l_2'}(1+\epsilon_4')$ with $|\epsilon_4'|\leq |\epsilon_4|\leq \epsilon$, $l_1',l_2'\in \mathbb{Z}$ (see (\ref{ocenakoren})).
From the first estimate of (\ref{est17}) 
it follows that either $|x|^{2}\geq \frac{1}{4}$ or $|u|^{2}\geq \frac{1}{4}$ (note $\delta\leq \frac{1}{2}$), and furthermore $u+x=(-1)^{l_1'}\sqrt{\epsilon_1}$ yields $|x|,|u|\geq \frac{1}{2}-\epsilon$.
Next, on one hand we have $|y|+|v|\geq |y+v|\geq 1-\epsilon\geq \frac{1}{2}$ and on the other hand (see (\ref{est17})):
\begin{equation}\label{est1177}
\delta\geq \big||y|^{2}-|v|^{2}\big|= \big||y|-|v|\big| \, \big||y|+|v|\big|\geq  \big||y|-|v|\big|^{2},
\end{equation}
hence $|y|,|v|\geq \frac{1}{4}-\sqrt{\delta}$. Using the lower estimates on $x$, $y$, $u$, $v$ and
applying (\ref{ocenah}) to $u+x=(-1)^{l_1'}\sqrt{\epsilon_1}$ and to $|\overline{x}y-\overline{u} v|\leq \delta$ (see (\ref{est17})) and symplifying, gives 
\begin{align}\label{ocenapsi2yv} 
&\psi_1=\phi-\eta-\pi+2l_1\pi,\qquad |\sin \psi_1|\leq \tfrac{2\sqrt{\epsilon}}{\frac{1}{2}-\epsilon},\quad l_1\in \mathbb{Z},\nonumber\\
%
&\psi_2=(\varphi-\phi)-(\kappa-\eta)+2l_2\pi,\qquad |\sin \psi_2|\leq \tfrac{2\delta}{(\frac{1}{4}-\sqrt{\delta})(\frac{1}{2}-\epsilon)},\quad l_2\in \mathbb{Z},\\
&\varphi-\kappa=\psi_1+\psi_2-\pi+2(l_2+l_2)\pi, \quad |2\cos (\tfrac{\varphi-\kappa}{2})| \nonumber
\leq 2|\sin \psi_1|+2|\sin \psi_2|.
\end{align}
Further we denote $y,v$ as in (\ref{polarL}) and manipulate the following expression:
\[
y+v=
(|y|-|v|)e^{i\varphi}+|v|(e^{i\varphi}+e^{i\kappa})=(|y|-|v|)e^{i\varphi}+2|v|\cos (\tfrac{\varphi-\kappa}{2})e^{i\frac{\varphi+\kappa}{2}}.
\]
Using this and (\ref{est1177}), (\ref{ocenapsi2yv}) we deduce that
\begin{equation*} 
1-\epsilon\leq |1+\epsilon_4'|=|y+v|\leq \sqrt{\delta}+\tfrac{8\sqrt{\epsilon}}{1-2\epsilon}+\tfrac{32\delta}{(1-4\sqrt{\delta})(1-2\epsilon)}. 
\end{equation*}
It is not too dificult to find arbitrarily small $\epsilon,\delta$ to contradict the last inequality.
\end{enumerate}


\item \label{g100110tau}
$ (1\oplus 0,
\widetilde{B})\dashrightarrow 
\big(\begin{bsmallmatrix}
0 & 1\\
\tau & 0 
\end{bsmallmatrix},
\begin{bsmallmatrix}
a & b\\
b & d 
\end{bsmallmatrix}\big)
$, \quad $0\leq \tau <1$, $b\geq 0$


From Lemma \ref{lemapsi1} (\ref{lemapsi11}) for (C\ref{r4}) with $\alpha=0$ we get
\begin{equation}\label{ocena5psi1}
|\overline{y}v|, |\overline{x}v|, |\overline{u}y|\leq \delta, \quad
\big|(1+\tau)\Rea(\overline{x}u)+i(1-\tau)\Ima(\overline{x}u)-\tfrac{1}{c}\big|\leq \delta.
\end{equation}
%
%
Denoting $\delta_2=xv$, $\delta_3=yu$, $\delta_4=yv$ and rearranging the terms in (\ref{eqBF1}) yields:
\begin{align}\label{eqBF10t}
&ax^2+2bux+du^{2}=\widetilde{a}+\epsilon_1,\nonumber \\
&duv+axy=\widetilde{b}+\epsilon_2-b\delta_2-b\delta_3,\\ 
&dv^2+ay^2=\widetilde{d}+\epsilon_4-2b\delta_4. \nonumber
\end{align}

\begin{enumerate}[label=(\alph*),wide=0pt,itemindent=2em,itemsep=6pt]

\item \label{g100110taua}
$
\big(1\oplus 0,
\begin{bsmallmatrix}
0 & 0\\
0 & 1
\end{bsmallmatrix}\big)\dashrightarrow
\big(\begin{bsmallmatrix}
0 & 1\\
\tau & 0 
\end{bsmallmatrix},
B\big)
$, \quad $0\leq \tau <1$

From (\ref{ocena5psi1}) it follows that either $|x|^{2},|y|^{2}\leq \delta$ or $|v|^{2},|y|^{2}\leq \delta$ or $|u|^{2},|v|^{2}\leq \delta$. 

\quad
If $|v|^{2},|y|^{2}\leq \delta$ the last equation of (\ref{eqBF10t}) with $\widetilde{d}=1$ fails for $\epsilon=\delta<\frac{1}{1+|a|+|d|+2b}$. 

\quad
Next, if $|x|^{2},|y|^{2}\leq \delta$ the first equation of (\ref{eqBF10t}) for $\widetilde{a}=0$ and the last estimate of  (\ref{ocena5psi1}) give
\begin{equation}\label{estkonec1}
\tfrac{2b(1-\delta)}{1+\tau}\leq \tfrac{2b}{1+\tau}\big|2\tau\Rea(\overline{x}u)+(1-\tau)\overline{x}u\big|
\leq |2bux|\leq \epsilon+|a|\delta+|du^{2}|.
\end{equation}
For $d=0$, $\epsilon=\delta<\frac{2b(1+\tau)^{-1}}{1+|a|+2b(1+\tau)^{-1}}$ we immediately get a contradiction. The last equation of (\ref{eqBF10t}) for $\widetilde{d}=1$ and the second equation of (\ref{eqBF10t}) for $\widetilde{b}=0$ yield 
\[
|du^2|\leq  \tfrac{|duv|^2}{|dv^{2}|}=\tfrac{(-axy+\epsilon_2-b\delta_2-b\delta_3)^{2}}{\widetilde{d}+\epsilon_4-2b\delta_4}\leq \tfrac{(\epsilon+2b\delta+|a|\delta)^2}{(1-a\delta-\epsilon-2b\delta)|d|}.
\]
(Note that for $v=0$ the last equation of (\ref{eqBF10t}) with $\widetilde{d}=1$ fails for $\epsilon=\delta<\frac{1}{1+|a|+2b}$.)
Using this and (\ref{estkonec1}) we conclude that
\[
2b \tfrac{1-\delta}{1+\tau}\leq 
|a|\delta+\epsilon+\tfrac{(\epsilon+2b\delta+|a|\delta)^2}{1-|a|\delta-\epsilon-2b\delta}.
\]
It is straightforward to see that this fails for any sufficiently small $\epsilon,\delta$.
The case $|u|^{2},|v|^{2}\leq \delta$ is done in a similar manner, with $a,x,u$ replaced by $d,y,v$, respectively.

\item 
$
(1\oplus 0,
\widetilde{a}\oplus 0)\to \big(
\begin{bsmallmatrix}
0 & 1\\
\tau & 0 
\end{bsmallmatrix},
B\big)
$,\quad  $0\leq \tau <0$, $\widetilde{a}\geq 0$


As in \ref{case10to1t} \ref{case10to1ta} we argue that (\ref{eqABEF}) has solutions for arbitrarily small $E$, $F$ precisely when given $s\in (0,1]$ we can solve the system of equations 
\begin{equation}\label{pogojtau}
ax^2+2bux+du^2=\widetilde{a}+\epsilon_1, \qquad (1+\tau)\Rea(\overline{x}u)+i(1-\tau)\Ima(\overline{x}u)=c^{-1}(1+\delta_1),
\end{equation}
where $ |\epsilon_1|,|\delta_1|\leq s$. 
By choosing $y,v$ small enough the remaining two equations of (\ref{eqBF10t}) are satisfied for some $|\delta_2|,|\delta_4|\leq s$ and the moduli of the first three expressions in Lemma \ref{lemapsi1} (C\ref{r4}) for $\alpha=1$ are bounded by $s$.

\begin{enumerate}[label=(\roman*),wide=0pt,itemindent=4em,itemsep=3pt]
\item $\begin{bsmallmatrix}
0 & b\\
b & 0
\end{bsmallmatrix}$, $b>0$\\
Observe that since $|c|=1$ the second equation of (\ref{pogojtau}) is equivalent to the equation which compares the moduli (squared) of both sides of the equation: 
\begin{equation}\label{elipsa}
(1+\tau)^2\big(\Rea(\overline{x}u)\big)^2+(1-\tau)^2\big(\Ima(\overline{x}u)\big)^2=|1+\delta_1|^2, \qquad |\delta_1|\leq s.
\end{equation}
In particular it follows that $\frac{1+s}{1-\tau}\geq |\overline{x}u|\geq \tfrac{1-s}{1+\tau}$. Using the first equation of (\ref{pogojtau}) for $a=d=0$ we further get $\big|\widetilde{a}-2b|ux|\big|\leq \epsilon $, and thus
\[
\tfrac{1+s}{1-\tau}+\tfrac{\epsilon}{2b}\geq \tfrac{\widetilde{a}}{2b}\geq \tfrac{1-s}{1+\tau}-\tfrac{\epsilon}{2b}.
\]
Clearly, when $\widetilde{a}\not \in [\frac{2b}{1+\tau},\frac{2b}{1-\tau}]$ we see that the above inequality fails for every sufficiently small $s>0$. For $\widetilde{a}\in [\frac{2b}{1+\tau},\frac{2b}{1-\tau}]$ we choose $u,x$ such that $2bxu=\widetilde{a}$ and so that $R+iT=\overline{x}u=\frac{\widetilde{a}}{2b}\frac{\overline{x}}{x}$ lies on an ellipse $(1+\tau)^2 R^2+(1-\tau)^2 T^2=1$ (see (\ref{elipsa})). Therefore (\ref{pogojtau}) for $a=d=\epsilon_1=\delta_1=0$ is satisfied and it yields the existence of the path.

\item $B=\begin{bsmallmatrix}
a & b\\
b & d
\end{bsmallmatrix}$, \quad $|a|+|d|\neq 0$, $b>0$\\
Set $x_0=|x_0|e^{\phi_0}$, $u_0=|u_0|e^{i\eta_0}$ and let $\frac{x}{u}=\frac{x_0}{u_0}=R_0e^{i\beta_0}$ (with $\beta_0=\phi_0-\varphi_0$) solve the equation $0=ax^2+2bux+du^{2}=u^2\big(a(\frac{x}{u})^2+2b\frac{x}{u}+d\big)
$ (see (\ref{pogojtau}) for $\widetilde{a}=0$).
A direct computation shows that by choosing $|x_0|^{2}=R_0\big|(1+\tau)\cos (-\beta_0) +(1-\tau)\sin (-\beta_0)\big|^{-1}$ we assure that $\overline{x}_0u_0$ fullfils the second equation of (\ref{pogojtau}) for $\delta_1=0$ and some $|c|=1$. 

\quad
Denoting further $x,y,u,v$ as in (\ref{polarL}) with $|x|=r$, $|u|=t$, $a=|a|e^{i\iota}$, $d=|d|e^{i\vartheta}$ the equations (\ref{pogojtau}) for $\epsilon_1=\delta_1=0$ 
are seen as 
\begin{align}\label{izraz4}
&e^{i(\phi+\eta)}\big(|a|r^{2}e^{i(\phi-\eta+\iota)}+2brt+dt^{2}e^{-i(\phi-\eta-\vartheta)}\big)=\widetilde{a}, \\ 
&rt\big|(1+\tau)\cos(\eta-\phi)+i(1-\tau)\sin (\eta-\phi)\big|=1.\nonumber
\end{align}
Define a function with constraint
\[
f(r,t,\beta)=\big||a|r^{2}e^{i(\beta+\iota)}+2brt+dt^{2}e^{-i(\beta+\vartheta)}\big|, \quad rt\big|(1+\tau)\cos (\beta)+i(1-\tau)\sin(\beta)\big |=1,
\]
where $r,t\geq 0, \beta\in \mathbb{R}$.
Fot $t=\frac{\Gamma(\beta)}{r}$, $\Gamma(\beta)=\big|(1+\tau)\cos (\beta)+i(1-\tau)\sin (\beta)\big|^{-1}$ we set
\[
g(r)=f\big(r,\tfrac{1}{r}\Gamma(\beta_0),\beta_0\big)=\big||a|r^{2}e^{i(\beta_0+\iota)}+2|b|\Gamma(\beta_0)+|d|\tfrac{(\Gamma(\beta_0))^{2}}{r^{2}}e^{-i(\beta_0+\vartheta)}\big|.
\]
Since $\Gamma$ is bounded we get $g(r)\to \infty$ as $r\to \infty$. As also $g(|x_0|)=0$ (see the above observation) and $g$ is continuous for $r\in [|x_0|,\infty)$, the range of $g$ (and hence $f$ with constraint) is $[0,\infty)$. 
Thus the left-hand side of the first equation in (\ref{izraz4}) 
can be any number from the interval $[0,\infty)$ and the second equation in (\ref{izraz4}) stays valid, provided that $r,t$, $\phi+\eta$ ($\phi-\eta=\beta_0$) are chosen appropriately.
Thus (\ref{izraz4}) and hence (\ref{pogojtau}) are solvable for any $\widetilde{a}\geq 0$, which proves the existence of a path.

\item $B=a \oplus d$, \quad $a,d\in \{0,1\}$, $ad=0$\\
If $a\neq 0$, $d=0$ (or $d\neq 0$, $a=0$) we take $y=v=s$, $x=\sqrt{\frac{\widetilde{a}+s}{a}}$ (or $u=\sqrt{\frac{\widetilde{a}+s}{a}}$) and $u$ (or $x$) such that $\overline{x}u=\frac{1}{1+\tau}$ in (\ref{cPepsi}) to prove $
(1\oplus 0,
\widetilde{a}\oplus 0)\to \big(
\begin{bsmallmatrix}
0 & 1\\
\tau & 0 
\end{bsmallmatrix},
a\oplus d\big)
$.

\end{enumerate}

\item 
$
(1\oplus 0,
\widetilde{a}\oplus 1)\dashrightarrow
\big(\begin{bsmallmatrix}
0 & 1\\
0 & 0 
\end{bsmallmatrix},B
\big)
$ 


Since $|vy|\leq \delta$ by (\ref{ocena5psi1}), then either $|y|^2\leq \delta$ or $|v|^{2}\leq \delta$ (or both). Let us consider $|y|^2\leq \delta$. Then the last equation of (\ref{eqBF10t}) for $\widetilde{d}=1$ yields that $|d||v|^{2}\geq 1-2b\delta-|a|\delta$. If $d=0$, we get a contradiction for $\delta<\frac{1}{|a|+2b}$, while for $d\neq 0$ we get  
$|u|\geq  \frac{1-\delta}{|x|}\geq \frac{1-\delta}{\delta} \sqrt{
\frac{1-2b\delta-|a|\delta}{|d|}
}$ (by (\ref{ocena5psi1}) for $\tau=0$ we have $|vx|\leq \delta$,  $\big||\overline{x}u|-1\big|\leq \delta$). Note that $x=0$ would yield $1\leq \delta$, and $\delta=0$ would imply $x=y=0$, both is not possible. Applying the triangle inequality to the right-hand side of the second equation of (\ref{eqBF10t}) for $\widetilde{b}=0$, and using the lower estimates for $|u|,|v|$ give:
\[
2b\delta+\epsilon+|a|\delta\sqrt{\tfrac{|d|}{1-2b\delta-|a|\delta}}\geq |\epsilon_2-axy-b\delta_2-b\delta_3|=|duv|\geq 
\tfrac{1-\delta}{|d|\delta} \big( 1-2b\delta-|a|\delta  \big).
\]
By taking $\epsilon,\delta$ sufficiently small this inequality fails.
The case $|v|^{2}\leq \delta$ is for the sake of symmetry done in a similar manner, but with $a,x,u$ replaced by $d,y,v$.

\item 
$
\big(1\oplus 0,
\begin{bsmallmatrix}
0 & 1\\
1 & 0 
\end{bsmallmatrix}\big)\dashrightarrow
\big(
\begin{bsmallmatrix}
0 & 1\\
0 & 0 
\end{bsmallmatrix},
B\big)
$\\ 
If $a=d=0$ the second equation of (\ref{eqBF10t}) for $\widetilde{a}=\widetilde{d}=0$, $\widetilde{b}=1$ fails for $\epsilon=\delta<\frac{1}{1+2b}$. 

\quad
Next, suppose $|a|+|d|\neq 0$.
As in \ref{g100110tau} \ref{g100110taua} we get that either $|x|^{2},|y|^{2}\leq \delta$ or $|v|^{2},|y|^{2}\leq \delta$ or $|u|^{2},|v|^{2}\leq \delta$. 

\quad
If $|x|^{2},|y|^{2}\leq \delta$ the first (the last) equation of (\ref{eqBF10t}) for $\widetilde{a}=0$ (or $\widetilde{d}=0$) and $|\overline{x}u|\leq 1+\delta$ (see (\ref{ocena5psi1}) for $\tau=0$) yield $|du^{2}|\leq |a|\delta+2b(1+\delta)+\epsilon$ (or $|dv^{2}|\leq \epsilon+2b\delta+|a|\delta$). Combining this with the second equation of (\ref{eqBF10t}) for $\widetilde{b}=1$ gives  
\[
\big(1-\epsilon-|a|\delta-2b\delta\big)^{2}\leq |duv|^{2}\leq \big(\epsilon+2b\delta+|a|\delta\big)\big(|a|\delta+2b(1+\delta)+\epsilon\big),
\]
which failes for any sufficiently small $\epsilon,\delta$.
In case $|u|^{2},|v|^{2}\leq \delta$ we get allmost the same estimates, but with $a,x,y$ replaced by $d,u,v$.

\quad
Finally, let $|v|^{2},|y|^{2}\leq \delta$.
Since $|ux|\leq 1+\delta$, then either $|x|^{2}\leq 1+\delta$ or $|u|^{2}\leq 1+\delta$ (or both). If $|ax^{2}|\leq |a|(1+\delta)$ (or $|du^{2}|\leq |d|(1+\delta)$) the first equation of (\ref{eqBF10t}) for $\widetilde{a}=0$ yields $|du|^{2}\leq \epsilon+2b\delta+|a|(1+\delta)$ (or $|ax|^{2}\leq \epsilon+2b\delta+|d|(1+\delta)$). Thus we deduce that
\[ 
|d|\,|uv|^{2},|a|\,|xy|^{2}\leq \delta \big(\epsilon+2b\delta+(1+\delta)\max\{|d|,|a|\}\big).
\]
From the second equations of (\ref{eqBF10t}) for $\widetilde{b}=1$ (using the triangle inequality) we get
\[
\delta\big(|a|+|d|\big)  \big(\epsilon+2b\delta+(1+\delta)\max\{|d|,|a|\}\big)
\geq 1-\epsilon-2b\delta,
\]
which fails for any sufficiently small $\epsilon,\delta$.
\end{enumerate}


\item \label{b1010}
$
(1\oplus 0,
\widetilde{a}\oplus \widetilde{d})\dashrightarrow
\big(\begin{bsmallmatrix}
0 & 1\\
1 & i 
\end{bsmallmatrix},B\big)
$, \quad $\widetilde{d}\in \{0,1\}$, $\widetilde{a}\geq 0$, $\widetilde{a}\widetilde{d}=0$ (see Lemma \ref{lemalist}, (\ref{pogoj}))

\begin{enumerate}[label=(\alph*),wide=0pt,itemindent=2em,itemsep=6pt]
\item 
$
B=
\begin{bsmallmatrix}
0 & b\\
b & 0 
\end{bsmallmatrix}
$, $b>0$

If $\widetilde{d}=0$ then $P(s)=
\frac{1}{\sqrt{2}}\begin{bsmallmatrix}
\frac{\widetilde{a}}{2b}(1-i) & s \\
1+i & s
\end{bsmallmatrix}$, $c(s)=-i$ in (\ref{cPepsi}) yields
$(1\oplus 0,
\widetilde{a}\oplus 0)\to
\big(\begin{bsmallmatrix}
0 & 1\\
1 & i 
\end{bsmallmatrix},\begin{bsmallmatrix}
0 & b\\
b & 0 
\end{bsmallmatrix}\big)$. 

\quad
Next, let $\widetilde{a}=0$, $\widetilde{d}=1$. 
From Lemma \ref{lemapsi1} (\ref{lemapsi11}) for (C\ref{r10}) with $\alpha=1$, $c^{-1}=e^{i\Gamma}$ we get
\begin{equation}\label{est16}
|\overline{x}v+\overline{u}y|\leq \delta, \quad  |v|^{2},|\overline{u}v|,\big|2\Rea(\overline{y}u)\big|\leq \delta , \quad 
\big|2\Rea(\overline{x}u)+i|u|^{2}-e^{i\Gamma}\big|\leq \delta. 
\end{equation}
We have equations (\ref{eqadb}); the first one for $\widetilde{a}=0$ (the third one for $\widetilde{d}=1$) gives $\frac{\epsilon}{2b}\geq |ux|$ (and $\frac{1-\epsilon}{2b}\leq |vy|$). The last estimate in (\ref{est16}) further implies 
$\frac{\epsilon}{2b} \geq |ux|\geq \big|\Rea (\overline{x}u)\big|\geq \frac{|\cos \Gamma|-\delta}{2}$ and  
\[
|u|^{2}\geq |\sin \Gamma|-\delta \geq \sqrt{1-(\tfrac{\epsilon}{b}+\delta)^{2}}-\delta, \qquad 
|x|^{2}\leq \tfrac{\epsilon^{2}}{(2b)^{2}}\big(\sqrt{1-(\tfrac{\epsilon}{b}+\delta)^{2}}-\delta\big)^{-1}.
\]
Multiplying the second (rearranged) equation of (\ref{eqadb}) with $\frac{v}{b}$, and using the estimates on $|x|$, $|u|$, $|v|$, $|vy|$ we get an inequality that fails to hold for $\epsilon,\delta$ small enough:
\[
\big(\sqrt{1-(\tfrac{\epsilon}{4b}+\delta)^{2}}-\delta\big)\tfrac{1-\epsilon}{2b}\leq \big|u(vy)\big|=|\tfrac{\epsilon_2v}{b}-v^{2}x|\leq \tfrac{\epsilon \sqrt{\delta}}{b}+\tfrac{\delta\epsilon}{2b}\big(\sqrt{1-(\tfrac{\epsilon}{4b}+\delta)^{2}}-\delta\big)^{-1}.
\]


\item 
$
B=
a\oplus d
$, \quad $a\geq 0$

%

\begin{enumerate}[label=(\roman*),wide=0pt,itemindent=4em,itemsep=3pt]

\item $\widetilde{B}=0\oplus 1$ \quad($\widetilde{d}=1$, $\widetilde{a}=0$)

We have equations (\ref{eqBFadad}) for $a=0$, $\widetilde{d}=1$ and it is apparent that the last of these equations fails to hold for  $\epsilon=\delta<\frac{1}{\epsilon+|d|\delta}$. Next, we set set $uv=\delta_1$, $v^{2}=\delta_2$, $|\delta_1|,|\delta_2|\leq \delta$. The last equation of (\ref{eqBFadad}) for $\widetilde{d}=1$ gives $ay^{2}=1+\epsilon_4-d\delta_2$. Combining it with the second and the first equation of (\ref{eqBFadad}) for $\widetilde{a}=0$ then yields $ax^{2}=\frac{(axy)^{2}}{ay^{2}}=\frac{(\epsilon_2-d\delta_1)^{2}}{1+\epsilon_4-d\delta_2}$ and $du^{2}=\epsilon_1-ax^{2}=\epsilon_1-\frac{(\epsilon_2-d\delta_1)^{2}}{1+\epsilon_4-d\delta_2}$, respectively. From the last estimate in (\ref{est16}) we finally get an inequality that fails for $a,d\neq 0$ and $\epsilon,\delta$ small enough:
\begin{align*}
(1-\delta)^{2} &\leq \big|2\Rea (\overline{x}u)+i|u|^{2}\big|^{2}\leq 4|xu|^{2}+|u|^{4} \\
               &\leq \tfrac{4}{a|d|}\tfrac{(\epsilon-|d|\delta)^{2}}{1-\epsilon-|d|\delta}\big(\epsilon+\tfrac{(\epsilon+|d|\delta)^{2}}{1-\epsilon-|d|\delta}\big)+\tfrac{1}{|d|^{2}}\big(a\epsilon+\tfrac{(\epsilon+|d|\delta)^{2}}{1-\epsilon-|d|\delta}\big)^{2}.
\end{align*}
When $d=\widetilde{a}=0$, $\widetilde{d}=1$, $a\neq 0$ the first and the last equation of (\ref{eqBFadad}) for $\widetilde{a}=0$ give $x^{2}=\frac{\epsilon_1}{a}$, $|y|^{2}=|\frac{1+\epsilon_4}{a}|\leq \frac{3}{2a}$. Further, the last estimate in (\ref{est16}) first leads to the upper estimate $|u|^{2}\leq 1+\delta\leq \frac{3}{2}$ (hence $\frac{3\epsilon}{2a}\geq |xu|^{2}\geq |\Rea (\overline{x}u)|$) and then the lower estimate $|u|^{2}\geq \frac{1}{2}-\frac{3\epsilon}{a}\geq \frac{1}{2}-\frac{3\epsilon}{a}$. Combining these facts with the first estimate in (\ref{est16}) implies the inequality that fails for $\epsilon=\delta<\frac{3}{4}\big((1+\sqrt{a})^{2}+\frac{9}{2a}  \big)^{-1}$:
\[
(\tfrac{1}{2}-\tfrac{3\epsilon}{a})\tfrac{3}{2a}\leq |\overline{u}y|^{2}\leq \big(|\overline{x}v|+\delta\big)^{2}\leq \big(\sqrt{\tfrac{\epsilon\delta}{a}}+\delta\big)^{2}
\]

\item
$\widetilde{B}=\widetilde{a}\oplus 0$, \quad $\widetilde{a}\geq 0$\\
In case $d=0$, $a\neq 0$ the path is guaranteed by 
$P(s)=\begin{bsmallmatrix}
\sqrt{\frac{\widetilde{a}}{a}} & s \\
i & 0
\end{bsmallmatrix}$, $c(s)=-i$ in (\ref{cPepsi}).

\quad
Next, let $a=0$ (hence $d>0$). The first equation of (\ref{eqBFadad}) for $a=0$ and the last estimate in (\ref{est16}) imply $1+\delta\geq |u|^2\geq \frac{\widetilde{a}-\epsilon}{d}$, a contradiction for $\widetilde{a}>d$, $\epsilon=\delta<\frac{\widetilde{a}-d}{d+1}$. When $\widetilde{a}\leq d$ the existence of a path for $\widetilde{a}=0$ and $\widetilde{a}> 0$ is proved with  (\ref{cPepsi}) for $c(s)=1$, 
$P(s)=\frac{1}{\sqrt{2}}\begin{bsmallmatrix}
s^{-1} & s \\
s & 0
\end{bsmallmatrix}$ and $c(s)=\frac{1}{d}(\sqrt{d^{2}-\widetilde{a}^{2}}-\widetilde{a})$,  
$P(s)=\frac{1}{2\sqrt{\widetilde{a}d}}\begin{bsmallmatrix}
\sqrt{d^{2}-\widetilde{a}^{2}} & s \\
2\widetilde{a} & 0
\end{bsmallmatrix}$, respectively.

\quad
It is left to consider the case $a>0$, $d\neq 0$.
Using a similar argument as in \ref{g100110} \ref{g100110a} \ref{g100110ai} we prove the existence of a path by solving the system of equations
$ax^{2}+du^{2}=\widetilde{a}$ and $\big|2\Rea (\overline{x}u)+i|u|^{2}\big|=1$; by choosing $v,y$ sufficiently small we assure that the remaining expressions in Lemma \ref{lemapsi1} for (C\ref{r10}) are arbitrarily small (then $\|E\|$ is arbitrarily small) and the last two equations in (\ref{eqBFadad}) are satisfied.
By setting $x=re^{i\phi}$, $u=te^{i\eta}$ with $d=|d|e^{i\vartheta}$ we can write the above equations as:
\begin{equation}\label{izrazeq6}
e^{i(\phi+\eta)}\big(ar^{2}e^{i(\phi-\eta)}+|d|t^{2}e^{-i(\phi-\eta)+i\vartheta}\big)=\widetilde{a},\qquad  
|2rt\cos (\psi-\eta)+it^{2}|=1.
\end{equation}
We show that the range of the following function given with a constraint is $[0,\infty)$:
\begin{equation*}
f(r,t,\beta)=\big|ar^{2}e^{i\beta}+|d|t^{2}e^{-i\beta+i\vartheta}\big|, \qquad \big(2rt\cos (\beta)\big)^{2}+t^{4}=1, \quad r,t\geq 0,\, \beta \in \mathbb{R}.
\end{equation*}
If $\vartheta=2l\pi$, $l\in \mathbb{Z}$ then $\beta=\frac{\pi}{2}$, $t=1$ satisfy the constraint and
$f(r,1,\frac{\pi}{2})=\big|ar^{2}-|d|\big|$, thus the range of $f$ is $[0,\infty)$.
When $\vartheta\neq 2l\pi$, $l\in \mathbb{Z}$ then by taking $\beta=\frac{\vartheta+\pi}{2}$ the constraint can be rewritten as $r^{2}\big( 4t^{2}\sin ^{2}(\frac{\vartheta}{2}) \big)=1-t^{4}$. Letting $t\to 0$ implies $r\to \infty$ and hence  
$f(r,t,\frac{\vartheta+\pi}{2})=\big|ar^{2}-|d|\big|\to \infty$. For $r_0=\sqrt{\frac{|d|}{a}}$ and by setting $t_0$ to be the positive root of the equation $t^{4}+4\frac{|d|}{a}\sin^{2}(\frac{\vartheta}{2})t^{2}-1=0$ 
we have $f(r_0,t_0,\frac{\vartheta+\pi}{2})=0$. The range of $f$ is again $[0,\infty)$. Therefore (\ref{izrazeq6}) has a solution for any $\widetilde{a}\geq 0$, provided that $\phi-\eta$, $|u|$, $|x|$ (corresponding to $\beta$, $t$, $r$), $\phi+\eta$ are chosen appropriately.

\end{enumerate}

\end{enumerate}

\end{enumerate}

Thic completes the proof of the theorem.
\end{proof}

\smallskip
\textit{Acknowledgement.}\\
The research was supported by Slovenian Research Agency (grant no. P1-0291).

The author wishes to thank M. Slapar for helpful discussions considering the topic of the paper.

\end{document}